\documentclass[10pt]{amsart}
\usepackage[utf8]{inputenc}
\usepackage{amsmath,amssymb,amsthm}
\usepackage{todonotes}
\usepackage{enumitem}
\usepackage{tabularx}
\usepackage{longtable}
\usepackage[unicode=true]{hyperref}
\usepackage{bm}
\usepackage[left,pagewise]{lineno}
\usepackage{tikz-cd}

\usepackage{thmtools}

\DeclareMathOperator{\Aut}{Aut}
\DeclareMathOperator{\GL}{GL}

\DeclareMathOperator{\irk}{irk}
\DeclareMathOperator{\im}{Im}
\DeclareMathOperator{\rk}{rk}
\DeclareMathOperator{\Mat}{Mat}
\DeclareMathOperator{\Tor}{Tor}
\DeclareMathOperator{\Ext}{Ext}
\DeclareMathOperator{\Gr}{Gr}
\DeclareMathOperator{\Hom}{Hom}

\DeclareMathOperator{\cd}{cd}
\DeclareMathOperator{\len}{len}

\newcommand{\D}{\mathcal D}
\newcommand{\CTor}{\mathfrak{Tor}}
\newcommand{\F}{\mathbb{F}}
\newcommand{\Z}{\mathbb{Z}}
\newcommand{\bbn}{\mathbb{N}}

\newcommand{\bbq}{\mathbb{Q}}
\newcommand{\bbd}{\mathbb{D}}
\newcommand{\Falgebra}[1]{\F_p[\![#1]\!]}

\newcommand{\fraks}{\mathfrak{s}}
\newcommand{\frakL}{\mathfrak{L}}
\newcommand{\frakh}{\mathfrak{h}}
\newcommand{\frakr}{\mathfrak{r}}
\newcommand{\cotimes}{\widehat{\otimes}}

\let\emph\relax 
\DeclareTextFontCommand{\emph}{\bfseries\em}

\newtheorem{thm}{Theorem}[section]
\newtheorem{lem}[thm]{Lemma}
\newtheorem{claim}[thm]{Claim}
\newtheorem{cor}[thm]{Corollary}
\newtheorem{prop}[thm]{Proposition}
\theoremstyle{definition}
\newtheorem{rem}[thm]{Remark}
\newtheorem*{rem*}{Remark}
\newtheorem{defn}[thm]{Definition}
\newtheorem{exmp}[thm]{Example}
\newtheorem{Question}{Question}
\newtheorem{Conjecture}[Question]{Conjecture}

\newcounter{henriquecomments}

\newcounter{andreicomments}

\date{\today}
 
\begin{document}

\title{Sylvester domains and pro-$p$ groups}
\author{Andrei Jaikin-Zapirain}
 \address{Departamento de Matem\'aticas, Universidad Aut\'onoma de Madrid \and  Instituto de Ciencias Matem\'aticas, CSIC-UAM-UC3M-UCM}
\email{andrei.jaikin@uam.es}

\author{Henrique Souza}
\address{Departamento de Matem\'aticas, Universidad Aut\'onoma de Madrid}
\email{henrique.mendesdasilva@uam.es}

\begin{abstract} Let $G$ be  a finitely generated torsion-free pro-$p$ group containing an open free-by-\(\Z_p\) pro-$p$ subgroup. We show that the completed group algebra $\Falgebra{G}$ is a Sylvester domain. Moreover the inner rank $\irk_{\Falgebra{G}}(A)$ of a matrix $A$ over  $\Falgebra{G}$ can be calculated by approximation by ranks corresponding to finite quotients of $G$. As a consequence, we obtain a particular case of the mod \(p\) Lück approximation for abstract finitely generated subgroups of free-by-\(\Z_p\) pro-\(p\) groups.
 \end{abstract}

\maketitle

\section{Introduction}

A Sylvester rank function on a ring $R$ is a function taking non-negative real values on matrices over $R$ and satisfying a series of conditions  (see Paragraph \ref{subsec-intro-division-rings}) resembling the conditions of the rank function of matrices over a field. The value of a Sylvester rank function $\rk$ on a matrix $A$ over $R$ is bounded from above by its inner rank $\rk(A)\le \irk_R(A)$. If $\irk_R$ is itself a Sylvester rank function, then the ring $R$ is called a \emph{Sylvester domain}. This notion appeared first implicitly in the works of P. Cohn and was explicitly defined by W. Dicks and E. Sontag \cite{DS78}. One remarkable property of a Sylvester domain $R$ is that it has a universal division ring of fractions $\D_R$ into which it embeds and such that for  every matrix $A$ over $R$, its rank over $\D_R$ is equal to its inner  rank $\irk_R(A)$. In particular, $R$ has no zero divisors, see \cite[Sec. 0]{DS78}. 

In the case $R=K[\Gamma]$ is a group algebra of a free group $\Gamma$ over a field $K$, it was proven by P. Cohn \cite[Theorem 7.11.8, Proposition 5.5.1]{cohnFreeIdealRings2006} that \(K[\Gamma]\) is a Sylvester domain. In \cite{HL22}, F. Henneke and D. L\'opez-\'Alvarez considered the case where $\Gamma$ is a free-by-$\Z$ group and showed that $K[\Gamma]$ is a Sylvester domain if and only if every left  finitely generated  projective $K[\Gamma]$-module is free. 

We would like to notice that, for an abstract group $\Gamma$, if $K$ is a field of characteristic 0 and $K[\Gamma]$ is a Sylvester domain, then $\irk_{K[\Gamma]}$ coincides with the von Neumann Sylvester rank function $\rk_\Gamma$ (for the definition of $\rk_\Gamma$, see \cite[Section 2]{campbell_l2-betti_2019}).

In this paper we want to understand for which pro-$p$ groups $G$ the completed group algebra $\Falgebra{G}$ is a Sylvester domain. Often, the results in combinatorial group theory and combinatorial pro-$p$ group theory run in parallel; although their proofs are quite different. For example,  the theorem of J. Stallings (\cite{stallingsTorsionFreeGroupsInfinitely1968}) saying that a torsion-free and virtually free group is itself free was inspired by the previous and analogous result on free pro-$p$ groups due to J.-P. Serre in \cite{Se65}. Later in
\cite{lubotzkyCombinatorialGroupTheory1982}, A. Lubotzky described how the classical theorems of M. Hall, L. Greenberg and A. Howson about finitely generated subgroups of free groups  all have analogues for free pro-\(p\) groups.  More recently, a new proof of the Hanna Neumann Conjecture (first proven independently by J. Friedman and I. Mineyev in 2011) was given by the first author in \cite{jaikin-zapirainApproximationSubgroupsFinite2017} using homological methods that apply both to the abstract and pro-\(p\) versions of the conjecture. This result was later extended to non-solvable pro-\(p\) Demushkin groups and their discrete counterparts, the hyperbolic surface groups, in \cite{JS19} and \cite{antolinHannaNeumannConjecture2022} respectively. Still in the theme of Demushkin groups, in \cite{shustermanVirtualRetractionHowson2020} M. Shusterman and P. Zalesskii extended the pro-\(p\) version of Howson's theorem to non-solvable Demushkin groups, for which they also proved the virtual retractions property that was stablished for surface groups by P. Scott in \cite{scottSubgroupsSurfaceGroups1978}.

In the case of a pro-$p$ group $G$ there exists a Sylvester function on $\Falgebra{G}$, which is  an analogue of the von Neumann Sylvester rank function and we also denote it  by $\rk_G$ (see Example \ref{exmp-rk-G} for the definition). Its definition resembles the L\"uck approximation \cite{Lu94, campbell_l2-betti_2019}.  

If  $F$ is a finitely generated free pro-$p$ group, then $\Falgebra{F}$ is isomorphic to a ring of non-commutative formal power series over $\F_p$, and so from
  \cite{Co74} we know that $\Falgebra{F}$ is a Sylvester domain.
 In \cite{jaikin-zapirainExplicitConstructionUniversal2020}, the first author proved that in this case $\irk_{\Falgebra{F}}=\rk_F$.
In this paper we extend this result to   finitely generated free-by-\(\Z_p\) pro-\(p\) groups.
\begin{thm}\label{main}
Let $G$ be a finitely generated  torsion-free pro-$p$ group containing an open free-by-\(\Z_p\) pro-$p$ subgroup. Then $\Falgebra{G}$ is a Sylvester domain and, moreover, $\irk_{\Falgebra{G}}=\rk_G$.
\end{thm}

   \begin{Question}
Let $G$ be a pro-$p$ group. Assume that $\Falgebra{G}$ is a Sylvester domain. Does $\rk_G$ coincide with the the inner rank of $\Falgebra{G}$?
\end{Question}

Recall that the pro-$p$ Atiyah conjecture predicts that if $G$ is a pro-$p$ group with the exponent of its torsion elements bounded by \(p^n\) then $\rk_G$ takes values in \(p^{-n}\Z\) (see \cite[Conj. 11.1]{campbell_l2-betti_2019}). The following consequence of Theorem~\ref{main} provides new cases where this conjecture holds.
   
 \begin{cor}\label{atiyah}
Let $G$ be a finitely generated  torsion-free pro-$p$ group containing an open free-by-\(\Z_p\) pro-$p$ subgroup. Then $\rk_G$ takes only integer values.
  \end{cor}
   Notice that the pro-$p$ Atiyah conjecture implies also pro-$p$ Kaplansky  conjecture: $\Falgebra{G}$ does not have non-trivial zero divisors if $G$ is torsion-free. Previously the pro-$p$ Atiyah conjecture was only known for pro-$p$ groups which are residually-(torsion-free $p$-adic analytic) (\cite[Cor. 5.5]{JS19}).  This class includes all infinite Demushkin groups, which are the pro-\(p\) groups that satisfy Poincaré duality in dimension \(2\). Infinite Demushkin \(G\) groups comprise an important class of free-by-\(\Z_p\) pro-\(p\) groups, and while the Atiyah conjecture for them was previously known, the statement that \(\Falgebra{G}\) is a Sylvester domain is new. In Paragraph \ref{vfbyZ} we present examples of  finitely generated torsion-free pro-$p$ groups containing an open free-by-\(\Z_p\) pro-$p$ subgroup but which are not free-by-\(\Z_p\) themselves. Observe that the free kernel is abelian if and only if \(G\) is a \(p\)-adic analytic group (see Remark~\ref{rem-on-p-adic} for a full characterization of the \(p\)-adic analytic groups \(G\) for which \(\Falgebra{G}\) is a Sylvester domain).

     Let us describe the structure of the paper and the main ideas behind the proof of Theorem \ref{main} comparing it  with the proof of the main result in \cite{HL22}.  We let \(d(G)\) denote the cardinality of a minimal set of topological generators of a pro-\(p\) group \(G\).

  On one hand, all projective \(\Falgebra{G}\)-modules are free for the completed group algebra \(\Falgebra{G}\), and so,  in contrast to the case of abstract free-by-cyclic groups considered in \cite{HL22}, there is no difference for  $\Falgebra{G}$ to be a pseudo-Sylvester domain or a Sylvester domain. 
  
  On the other hand, if $N$ is a normal subgroup of an abstract group $\Gamma$ such that $\Gamma/N\cong \Z$, then the group algebra $K[\Gamma]$ is isomorphic to the crossed product $K[N]*\Z$, which allows easily to construct an embedding of the group algebra of a free-by-cyclic group into a division ring. This decomposition as a crossed product does not exists in the pro-$p$ situation and we substitute it by
presenting of $\Falgebra{G}$, when 
$G$ is a finitely generated free-by-\(\Z_p\) pro-\(p\) group with free pro-\(p\) kernel \(N\),
as a skew power series ring \(\Falgebra{G}\cong \Falgebra{N}[\![s;\sigma,\delta]\!]\)  for the natural conjugation automorphism \(\sigma\colon \Falgebra{N} \to \Falgebra{N}\) and the inner \(\sigma\)-derivation \(\delta\) (see Paragraph \ref{spsr}). In Section~\ref{sec:setup} we discuss also preliminary results and definitions  about Sylvester matrix rank functions and universal embeddings that we will use to prove Theorem~\ref{main}.

In Section \ref{sec-Sylvester-virtual} we investigate pro-$p$ groups $G$ for which $\rk_G=\irk_{\Falgebra{G}}$ or, more generally, $\Falgebra{G}$ is a Sylvester domain. In particular, we show that if $\Falgebra{G}$ is a Sylvester domain then $G$ is of cohomological dimension 2 and the properties $\rk_G=\irk_{\Falgebra{G}}$ and  $\Falgebra{G}$ being a Sylvester domain are commensurability invariants for torsion-free pro-$p$ groups. Thus, it is enough to prove Theorem \ref{main} for an open subgroup.

 Let  \(\Lambda = \F_p\langle a_1,a_2,\ldots\rangle\) be the free associative \(\F_p\)-algebra with a universal division \(\Lambda\)-ring of fractions \(\mathcal{D}_{\Lambda}\). 
Now assume that \(G\) is a finitely generated free-by-$\Z_p$ pro-$p$ group  having a mild flag presentation in the sense of Definition~\ref{def:flag-presentation}, a condition that we show is virtually satisfied for every finitely generated free-by-\(\Z_p\) pro-$p$ group (Lemma~\ref{lem:virtual-flag}). In this case we show that  the maps \(\sigma\) and \(\delta\) extend through the series of embeddings 
$\Falgebra{N} \hookrightarrow \Lambda[\![t]\!] \hookrightarrow \mathcal{D}_{\Lambda}[\![t]\!]$
 in such a way that they induce embeddings of the respective skew power series rings.
  We also show that the ring \(\mathcal{D}_{\Lambda}[\![t]\!][\![s;\sigma,\delta]\!]\) is a Noetherian domain and therefore has a classical ring of fractions \(\mathcal{Q}\) given by the Ore localization of its non-zero elements. 
 This provides an embedding of $\Falgebra{G}$ into the 
 division ring $\mathcal{Q}$ (Section \ref{sec:group-embeddings}).  
 
 For an arbitrary finitely generated virtually free-by-\(\Z_p\) pro-$p$ group $G$, this construction provides an embedding of $\Falgebra{G}$ in an Artinian ring $\mathcal{Q}$. The next step is to show that under this embedding all full matrices over $\Falgebra{G}$ become invertible over $\mathcal Q$. The main tool to prove this is Theorem \ref{thm-jaikin-criterion-universality} proven in \cite{jaikin-zapirainExplicitConstructionUniversal2020}. We achieve this in Section~\ref{sec:hom-finite-prop}. This shows that $\Falgebra{G}$ is a Sylvester domain.
 
 The last step is to show that $\irk_{\Falgebra{G}}=\rk_G$, and we prove it in Section~\ref{sec:atiyah}. As a corollary we obtain  the pro-\(p\)  Atiyah conjecture for  finitely generated torsion-free virtually free-by-\(\Z_p\) pro-\(p\) groups.  
 
 In Section \ref{abstract} we discuss some applications of  our result to the L\"uck approximation in positive characteristic.
 
We finish the paper with an appendix where  we show that if $R$ is a profinite ring and $R_\Sigma$ is a localization of $R$ with respect to a collection  $\Sigma$ of a square matrices over $R$, then  $R_\Sigma$ admits  a Hausdorff ring topology such that the map $R\to R_\Sigma$ is continuous. This is used only in the proof of Lemma~\ref{Iclosed} (see also the remark afterwards). 
 \section*{Acknowledgments}

This work is partially supported by the grant  PID2020-114032GB-I00 of the Ministry of Science and Innovation of Spain  and by the ICMAT Severo Ochoa project  CEX2019-000904-S4. We would like to thank Alan Logan and Henry Wilton for pointing out several examples of torsion-free virtually free-by-$\mathbb{Z}$ but not free-by-$\mathbb{Z}$ groups. The authors also extend their gratitude to the anonymous reviewer for their valuable comments.

\section{The setup}\label{sec:setup}

\subsection{General notation}
If \(G\) is a pro-\(p\) group, we denote the \emph{Frattini subgroup} of \(G\) by \(\Phi(G) = G^p[G,G]\) and the \emph{lower \(\bm{p}\)-central series} of \(G\) by \(G_1 = G\) and \(G_i = G_{i-1}^p[G_{i-1},G]\), so that \(G_2 = \Phi(G)\). We stack the commutators \([a,b] = a^{-1}b^{-1}ab\) on the right: \[[a_1,a_2,a_3,\ldots,a_n] = [[\cdots[[a_1,a_2],a_3],\cdots],a_n]\,.\] If \(I_G\) denotes the augmentation ideal of \(G\) in the completed group algebra \(\Falgebra{G}\), the induced filtration \[D_i(G) = \{g \in G \mid g-1 \in I_G^i\}\] on \(G\) is called the \emph{dimension series} mod \(p\) (or also the \emph{\(\bm{p}\)-Zassenhaus filtration}). We remark that \(G_i \leq D_i(G)\) for every \(i \geq 1\).

When we say that a pro-\(p\) group \(G\) is finitely generated, or that \(X\) is a generating set for \(G\), it will always mean generation in the topological sense. We recall that if \(w_0(G)\) denotes the smallest cardinality of a fundamental system of neighborhoods of \(1\) in \(G\) -- it's local weight -- then \(w_0(G) = \max\{d(G), |\bbn|\}\) (\cite[Cor. 2.6.3]{ribesProfiniteGroups2010}). In particular, every closed subgroup of a finitely generated pro-\(p\) group is at most countably generated.

\subsection{Torsion-free virtually free-by-\texorpdfstring{$\Z_p$}{Zp} groups}
\label{vfbyZ}

There are many abstract torsion-free groups which are virtually free-by-$\mathbb{Z}$ but not free-by-$\mathbb{Z}$ themselves. One explicit family of such groups is given by the presentations \(\Gamma_e = \langle a,b,x,y\mid [a,b]^e = [x,y]^e\rangle\) for \(e > 1\) (see \cite[Example 7.2]{baumslagVirtualPropertiesCyclically2009}). More examples arise amongst 3-manifold groups such as fundamental groups of knot complements with Alexander polynomial \(1\). One family of such knots is given by the Pretzel links \(L(2m+1,2n+1,2q+1)\) satisfying \((m+n+1)(m+p+1)=m(m+1)\), such as \(L(-3,5,7)\). The condition on the Alexander polynomial implies that \([G,G]\) is perfect, and such groups virtually fibre over \(\Z\) by being virtually special as a combined consequence of the works of I. Agol, P. Przytycki and D. Wise(\cite{agolCriteriaVirtualFibering2008,agolVirtualHakenConjecture2013,przytyckiMixed3manifoldsAre2018}).

In the pro-$p$ case we can construct examples of torsion-free pro-$p$ groups which are virtually free-by-$\mathbb{Z}_p$ but not free-by-$\mathbb{Z}_p$ themselves using the ideas of \cite{baumslagVirtualPropertiesCyclically2009}.

\begin{prop}\label{prop-exmp-not-free-by-zp} Let \(G\) be the pro-\(p\) completion of the abstract group \[\Gamma  = \langle a_1,a_2,b_1,b_2 \mid [a_1,a_2]^p = [b_1,b_2]^p \rangle\,.\] Then \(G\) is a torsion-free pro-\(p\) group that is virtually free-by-\(\Z_p\) but not free-by-\(\Z_p\) itself.
\end{prop}

We divide the proof into two lemmas. Let \(F = F(x_1,x_2)\) be another free pro-\(p\) group on two generators and consider the surjective homomorphism \(\varphi\colon G \to F\) defined by \(\varphi(a_i) =  \varphi(b_i) = x_i\) for \(i = 1,2\).

\begin{lem} There exists an open subgroup \(V \leq F\) and a homomorphism \(\psi\colon V \to \Z_p\) such that \(\psi(x^{-1}[x_1,x_2]x) = 1\) for every \(x \in F(x_1,x_2)\).
\end{lem}
\begin{proof} Let \(V = \Phi(F)\). Then, it is clear that \([x_1,x_2]\) is not contained in \([V,V]\) and hence is a non-trivial element of \(M = V/[V,V] \simeq \Z_p^n\). Observe that \([x_1,x_2] [V,V]\) has finitely many \(F\)-conjugates which are linearly independent modulo \([V,V]\), and so one can find a homomorphism \(\psi_0\colon M \to \Z_p\) such that \(\psi_0(x^{-1}[x_1,x_2]x[V,V]) = 1\) for all \(x \in F\). The desired \(\psi\) is then the composition of the natural projection \(V \to M\) with \(\psi_0\).
\end{proof}

Let \(U = \varphi^{-1}(V)\). We claim that \(U\) is free-by-\(\Z_p\). More specifically, that if \(\theta\colon U \to \Z_p\) is the composition \(\psi\circ\varphi\), then:

\begin{lem}\label{lem:kernel-is-free} \(K = \ker\theta\) is free pro-\(p\).
\end{lem}
\begin{proof} Write \(G\) as the proper free pro-\(p\) product with amalgamation \(G = A \amalg_{C} B\) of the free pro-\(p\) groups \(A = F(a_1,a_2)\), \(B = F(b_1,b_2)\) and cyclic subgroup \(C = \langle [a_1,a_2]^p \rangle \simeq \langle [b_1,b_2]^p\rangle\) (\cite[Exer. 9.2.6a)]{ribesProfiniteGroups2010}). Let \((\mathcal{G},X)\) be the graph of pro-\(p\) groups consisting of two vertices \(A\) and \(B\) and an edge \(C\) connecting them such that \(\pi_1(\mathcal{G},X) \simeq G\), and let \(Y\) be the standard \(p\)-tree associated with \((\mathcal{G},X)\) (\cite[Sec. 4]{ribesPropTreesApplications2000}). We recall that:
\[V(Y) = G/A \bigsqcup G/B\,,\]
\[E(Y) = G/C\,,\]
\[d_0(gC) = gA\,, d_1(gC) = gB\,.\]

In particular, \(G\) acts on \(Y\) on the left and \(G\backslash Y \simeq X\), the vertex stabilizers are conjugate to \(A\) or \(B\) and the edge stabilizers are conjugate to \(C\).

Let \(Z = K\backslash Y\) and \(\mathcal{Z}\) be the associated graph of pro-\(p\) groups on \(Z\). The edge stabilizers of the \(K\) action on \(Y\) are the intersections \(K \cap  gCg^{-1}\) for \(g \in G\), which are trivial by construction. Since the vertex stabilizers are all free pro-\(p\), the fundamental group \(\pi_1(\mathcal{Z},Z)\) is also free pro-\(p\).

First, observe that \(U\) is a normal subgroup of \(G\) and \(G/U \simeq F/\Phi(F)\). Hence, \(UA = UB = G\). Therefore, \(G/A = UA/A \simeq U/(U\cap A)\) and similarly for \(G/B\). Since \(K\) is normal in \(U\) with quotient \(\Z_p\), we get homeomorphisms:

\[V(Z) \simeq K\backslash U/(U\cap A) \bigsqcup K\backslash U/(U \cap B) \simeq \Z_p/\theta(U\cap A) \bigsqcup \Z_p/\theta(U\cap B)\,.\]

Given that \([a_1,a_2] = [b_1,b_2]\) is an element of both \(U\cap A\) and \(U\cap B\) that is not trivial under \(\theta\), \(V(Z)\) is finite (and in fact consists of two elements).

Note that \(C\) is contained in \(U\), and since it generates the image of \(\theta\) we have \(U \simeq K \rtimes C\). Also observe that by the construction of \(\theta\) and \(K\) we have \([G,U] \leq K\). Combining this information with the homeomorphism \(K\backslash G \simeq K\backslash U \times U \backslash G\), we get that the induced right action of \(C\) in \(K\backslash U \times U \backslash G\) is given by \[(Kc,Ug)c' = (Kcc',Ug)\] for \(c,c' \in C\) and \(g \in G\). Therefore:
\[E(Z) = K\backslash G / C \simeq K\backslash U/C \times U \backslash G \simeq U\backslash G\] and \(Z\) is finite. From \cite[Thm. 6.6.1]{ribesProfiniteGraphsGroups2017}, we conclude that the induced map \(\pi_1(\mathcal{Z},Z) \to K\) is an isomorphism.
\end{proof}

\begin{proof}[Proof of Proposition~\ref{prop-exmp-not-free-by-zp}] Again write \(G\) as the proper free pro-\(p\) product with amalgamation \(G = A \amalg_{C} B\) as in Lemma~\ref{lem:kernel-is-free}. Since both \(A\) and \(B\) are torsion-free, \(G\) is also torsion-free by \cite[Thm. 4.2(b)]{ribesPropTreesApplications2000}. By \cite[Corollary 3.5]{KM93}, \(\Gamma\) is residually-\(p\), and so it embeds into \(G\). Since \([a_1,a_2]\) and \([b_1,b_2]\) are distinct elements in \([\Gamma,\Gamma] \subseteq [G,G]\) with the same \(p\)-th power, the group \([G,G]\) cannot be free pro-\(p\). Hence, \(G\) itself cannot be free-by-\(\Z_p\). However, we've shown in Lemma~\ref{lem:kernel-is-free} that \(G\) contains an open subgroup \(U\) that is free-by-\(\Z_p\).
\end{proof}

\subsection{Skew power series rings}\label{spsr}

Let \(R\) be a topological ring and \(\sigma\) a continuous automorphism of \(R\). We define the continuous map \(\delta = \sigma - \operatorname{id}\) and observe that \(\delta\) is a right \(\sigma\)-derivation, 
that is, \(\delta\) is an additive map and for all \(a, b \in R\) we have: \[\delta(ab) = \delta(a)b + \sigma(a)\delta(b)\,.\] Moreover, \(\sigma\) and \(\delta\) commute. A \emph{(right) skew power series ring} 
\(S = R[\![s;\sigma,\delta]\!]\) over \(R\) with automorphism \(\sigma\) and derivation \(\delta\) consists of the topological abelian group of all formal power series \[\sum_{i \geq 0} s^ia_i\,,\quad \text{with }a_i \in R\] together with a multiplication map defined by the rule:
\begin{equation}\label{eq-power-series-mult}
    \left(\sum_{i \geq 0} s^ia_i\right)\left(\sum_{j \geq 0}s^jb_j\right) = \sum_{m \geq 0}s^m\left(\sum_{n=0}^m \sum_{k \geq n} \binom{k}{n} \delta^{k-n}(\sigma^n(a_{m-n}))b_k\right)\,.
\end{equation}
While we will not use it, one can define a left skew power series ring analogously. When \(R\) is a Noetherian pseudocompact ring, our definition agrees with the one in \cite{schneiderCodimensionModulesSkew2006} by seeing \(S\) simultaneously as a right and left skew power series ring in the unique compatible way (see \cite[Sec. 1]{schneiderCodimensionModulesSkew2006}).

For the ring \(R[\![s;\sigma,\delta]\!]\) to exist, one needs to ensure that the infinite sums on the right hand side of~(\ref{eq-power-series-mult}) converge for every possible choice of elements \(a_i\) and \(b_j\). If this is the case, then~(\ref{eq-power-series-mult}) defines a continuous multiplication on \(S\) which makes it into a topological \(R\)-algebra satisfying \[as = s\sigma(a) + \delta(a)\,,\quad\text{ for every }a \in R\,.\]

In this paper, the ring \(R\) will either be the completed group algebra \(\Falgebra{N}\) over \(\F_p\) of a free pro-\(p\) subgroup \(N\) of a pro-\(p\) group \(G\) or a power series ring \(\Lambda[\![t]\!]\) over a discrete \(\F_p\)-algebra \(\Lambda\) for a fixed prime \(p\). In both cases, \(R\) is a complete ring whose topology is induced by a Hausdorff filtration \(R_k\), where \(R_k\) is either the intersection of the \(k\)-th power \(I_G^k\) of the augmentation ideal of \(\Falgebra{G}\) with \(\Falgebra{N}\) or the principal ideal generated by \(t^k\) of \(\Lambda[\![t]\!]\). Following the nomenclature of \cite{lazardGroupesAnalytiquesPadiques1965,labuteMildPropgroupsGalois2006}, we have a filtration function \[w(a) = p^{-\sup\{k \geq 0 \mid a \in R_k\}}\,,\] which is submultiplicative, satisfies the ultrametric inequality and is such that \(a_i \to 0\) if and only if \(w(a_i) \to 0\). In the \(\Falgebra{N}\) case, this filtration function is called a valuation in \cite[Sec. 2.2]{ershovGroupsPositiveWeighted2013}, though we shall reserve the name ``valuation'' for a stronger class of filtration functions (see Section~\ref{sec:strongly-free-sequences}).

Hence, we restrict ourselves to complete Hausdorff rings with a submultiplicative filtration function \(w\). We say that the derivation \(\delta = \sigma - \operatorname{id}\) is \emph{topologically nilpotent} if \(\delta^k \to 0\) pointwise, that is, \(w(\delta^k(a)) \to 0\) for any \(a \in R\). If \(\delta\) is topologically nilpotent, then each infinite sum \[\sum_{k\geq n} \binom{k}{n} \delta^{k-n}(\sigma^n(a_{m-n}))b_k\] appearing in~(\ref{eq-power-series-mult}) is convergent for any choices of \(a_i\) and \(b_j\) and therefore the ring \(R[\![s;\sigma,\delta]\!]\) exists. One of our preliminary results is that the completed \(\F_p\)-group algebra of any free-by-\(\Z_p\) pro-\(p\) group \(G\) with free kernel \(N\) is isomorphic to a right skew power series ring over \(\Falgebra{N}\).

\subsection{Universal division ring of fractions}\label{subsec-intro-division-rings}
 All ring homomorphisms in this paper are assumed to preserve the multiplicative identity \(1\).  A subring \(S\) of \(R\) is \emph{division closed} if for every unit \(x \in S\cap R^\times\) one has \(x^{-1} \in S\). The \emph{division closure} of a subring $S$ of $R$ is the smallest division closed subring of \(R\) containing \(S\).

Let $f:R\to S$ be a ring homomorphism. We say that $f$ is \emph{epic} if for every  ring $Q$ and homomorphisms $\alpha,\beta: S\to Q$, the equality $\alpha\circ f=\beta\circ f$ implies $\alpha=\beta$ (i.e., the natural map \(\Hom_{\text{Ring}}(S,Q) \to \Hom_{\text{Ring}}(R,Q)\) is injective for all rings \(Q\)).

As in \cite[Chap. 7.2]{cohnFreeIdealRings2006}, we define an \emph{epic division \(S\)-ring} as a division ring \(\mathcal{D}\) together with an epic homomorphism  \(\varphi\colon S \to \mathcal{D}\). The condition on $\varphi$ to be epic is equivalent to the condition that the division closure of $\varphi(S)$ is equal to $\D$ (\cite[Cor. 7.2.2]{cohnFreeIdealRings2006}).
 Let \(\varphi'\colon S \to \mathcal{D}'\) be  another epic division \(S\)-ring. A subhomomorphism of epic division \(S\)-rings is a homomorphism \(\psi\colon K \to \mathcal{D}'\), where \(K\) is a local subring of \(\mathcal{D}\) containing \(\varphi(S)\) with maximal ideal \(\ker \psi\), such that \(\psi\circ\varphi = \varphi'\). Two subhomomorphisms \(\psi_1\colon K_1 \to \mathcal{D}\) and \(\psi_2\colon K_2 \to \mathcal{D}\) are equivalent if there is a subring \(K_0\) of \(\mathcal{D}\) contained in \(K_1 \cap K_2\) such that \(\psi_1\) and \(\psi_2\) agree on \(K_0\) and it is local with maximal ideal \(\ker(\psi_1) \cap K_0 = \ker(\psi_2) \cap K_0\). A \emph{specialization} \(\mathcal{D} \to \mathcal{D}'\) of epic division \(S\)-rings is an equivalence class of subhomomorphisms. The archetypal examples are \(\mathcal{D} = \bbq\) and \(\mathcal{D}' = \F_p\) for \(S = \Z\), with the local subrings \(K\) being the localization of \(\Z\) at a prime \(p\).

The \emph{universal} division \(S\)-ring is an epic division \(S\)-ring \(\mathcal{D}\) such that for every other epic division \(S\)-ring \(\mathcal{D}'\) there exists a  specialization \(\mathcal{D} \to \mathcal{D}'\). It is unique up to $S$-isomorphism and we will denote it by $\D_S$. If the associated map \(\varphi\colon S \to \mathcal{D}_S\) is injective, we say that \(\mathcal{D}_S\) is the \emph{universal division \(S\)-ring of fractions}. In general, we will say that  $S\to \D$ is \emph{universal} if the division closure of the image of $S$ in $\D$ is isomorphic (as a $S$-ring) to $\D_S$.

The universal division \(S\)-ring \(\mathcal{D}_S\) need not exist in general, even if \(S\) is a domain (\cite{passmanUniversalFieldsFractions1982}). One class of rings that possess a universal division ring of fractions is the class of Sylvester domains (\cite[Thm. 7.5.13]{cohnFreeIdealRings2006}). To define a Sylvester domain, we consider Sylvester matrix rank functions and the inner rank. A \emph{Sylvester matrix rank function} \(\rk\) on an unital ring \(S\) is a non-negative real-valued function defined on the set \(\Mat(S)\) of all matrices over \(S\) satisfying:
\begin{enumerate}[label=(SMat\arabic*), itemindent=3em]
    \item \(\rk(A) = 0\) if \(A\) is a zero matrix and \(\rk(1) = 1\);
    \item \(\rk(AB) \leq \min\{\rk(A), \rk(B)\}\) for any pair of matrices \(A\) and \(B\) that can be multiplied;
    \item \(\rk\begin{pmatrix} A & 0\\ 0 &B\end{pmatrix} = \rk(A) + \rk(B)\) for any matrices \(A\) and \(B\);
    \item \(\rk\begin{pmatrix} A & C\\ 0 & B\end{pmatrix} \geq \rk(A) + \rk(B)\) for any matrices \(A\), \(B\) and \(C\) of appropriate sizes.
\end{enumerate}
Since \(\rk(\operatorname{Id_n}) = n\) by (SMat3), one concludes that \(\rk(A) \leq \min\{n,m\}\) for every matrix \(A \in \Mat_{n\times m}(S)\) by property (SMat2) and the identities \(A = \operatorname{Id}_n A = A\operatorname{Id}_m\). If \(S\) is a division ring, then any matrix can be put in row-echelon or column-echelon form by multiplication with invertible matrices, an operation that doesn't change the rank by (SMat2). It follows then from (SMat4) that there is a unique Sylvester matrix rank function on \(S\), given by the number of linearly independent rows (with a left \(S\)-action) or columns (with a right \(S\)-action). We include a proof of the following general fact for which we could find no reference:

\begin{lem}\label{lem:lem-monotonic} Let \(B\) be any submatrix of a matrix \(A \in \Mat(S)\) and \(\rk\) be any Sylvester matrix rank function on \(S\). Then, \(\rk(B) \leq \rk(A)\).
\end{lem}
\begin{proof} Any submatrix \(B\) of \(A \in \Mat_{n\times m}(S)\) can be obtained by removing rows or columns of \(A\). Hence, it suffices to prove the claim for \(A = \begin{pmatrix} B & a\end{pmatrix}\) and for \(A' = \begin{pmatrix}B \\ a'\end{pmatrix}\) for arbitrary columns \(a \in \Mat_{n\times 1}(S)\) and rows \(a' \in \Mat_{1\times m}(S)\). Since \[B = \begin{pmatrix} B & a\end{pmatrix}\cdot \begin{pmatrix} \operatorname{Id}_{n-1} \\ 0\end{pmatrix} = \begin{pmatrix} \operatorname{Id}_{m-1} & 0\end{pmatrix}\begin{pmatrix} B\\ a'\end{pmatrix}\,,\] the lemma follows from property (SMat2).
\end{proof}

Every Sylvester matrix rank function defines a dimension function \(\dim\) for finitely presented \(S\)-modules \(M \simeq S^m/AS^n\) with \(A \in \Mat_{n\times m}(S)\) through \(\dim M = m - \rk(A)\). This dimension function is an example of a \emph{Sylvester module rank function} (see \cite[Sec. 2.1]{jaikin-zapirainExplicitConstructionUniversal2020}), and this correspondence gives a bijection between the sets of Sylvester matrix rank functions and Sylvester module rank functions.

\begin{exmp}\label{exmp-rk-G} Let \(G\) be a pro-\(p\) group and $G>U_1>U_2>\ldots$  be a sequence of open normal subgroups of \(G\)   such that \(\bigcap_{i=1}^\infty U_i = \{1\}\). Then, for any matrix \(A\) over the completed group algebra \(\Falgebra{G}\), if \(A_i\) denotes its reduction modulo \(U_i\), the limit \[\rk_G(A) = \lim_{i\to \infty} \frac{\rk_{\F_p} A_i}{|G\colon U_i|}\] exists, where \(A_i\) is seen as a linear operator \(\F_p^{m|G\colon U_i|} \to \F_p^{n|G\colon U_i|}\) after some choice of basis for \(\F_p[G/U_i]\), and defines a Sylvester matrix rank function on \(S = \Falgebra{G}\) that does not depend on the choice of the chain \(U_i\) (\cite[Prop. 11.2]{campbell_l2-betti_2019}). We denote the associated Sylvester module dimension function by \(\dim_G\). Observe that for a finitely presented $\Falgebra{G}$-module $M$ we have
$$\dim_G M=\lim_{i\to \infty}   \frac{\dim_{\F_p} \F_p \otimes_{\Falgebra{U_i}} M}{|G\colon U_i|} .$$
\end{exmp}

If a ring $S$ has a universal  division ring $u:S\to \D_S$, we denote by $\rk_S$ the induced Sylvester rank function:
$$\rk_S(M)=\rk_{\D_S}(u(M))\ (M \textrm{\ is a matrix over\ } S).$$
It is characterized uniquely  by the following universal property: for every division $S$-ring $\gamma:S\to \mathcal E$ and every matrix $M$ over $S$,
$\rk_{\mathcal E}(\gamma(M))\le \rk_{S}(M)$. 
\begin{prop} \label{extension}
Let $S$ be a ring and $u:S\to \D_S$  its universal division ring. Let $\alpha: S\to S$ be an automorphism. Then there exists a unique $\widetilde \alpha :\D_S\to \D_S$ such that $u\circ \alpha=\widetilde \alpha \circ u$.
\end{prop}
\begin{proof}
The universal property of $\rk_S$ implies that $\rk_S=\rk_S\circ \alpha$. Hence $u\circ \alpha: S\to \D_S$ is also universal. Thus, $u$ and $u\circ \alpha$ are $S$-isomorphic. This implies the existence of $\widetilde \alpha$.
\end{proof}
The \emph{inner rank} \(\irk_S(A)\) of a non-zero matrix \(A \in \Mat_{n\times m}(S)\) is the smallest non-negative integer \(k\) such that \(A\) factors as a product \(BC\) with \(B \in \Mat_{n\times k}(S)\) and \(C \in \Mat_{k \times m}(S)\). If the inner rank of a square \(n\times n\) matrix \(A\) is \(n\), we call \(A\) a \emph{full matrix}. A ring \(S\) is called a \emph{Sylvester domain} if the inner rank \(\irk\) is a Sylvester matrix rank function on \(S\). Note that \(\irk\) always satisfies the conditions (SMat1) and (SMat2), and every Sylvester matrix rank function \(\rk\) on \(S\) satisfies \(\rk(A) \leq \irk(_SA)\) by (SMat2). In fact, \(\irk\) satisfies an even stronger property than (SMat1): \(\irk A = 0\) if and only if \(A\) is the zero matrix. Every Sylvester domain \(S\) possesses a universal division ring of fractions \(\mathcal{D}_S\) such that \(\irk = \rk_{\mathcal{D}_S}\). For the construction of \(\mathcal{D}_S\), see \cite[Sec. 7.4 and 7.5]{cohnFreeIdealRings2006}. Moreover, \(\mathcal{D}_S\) satisfies another universal property:

\begin{prop}[{cf. \cite[Thm. 7.5.13(e)]{cohnFreeIdealRings2006}}]\label{prop-universal-prop-sylvester} If \(S\) is a Sylvester domain and \(\psi\colon S \to R\) is a ring homomorphism such that the image through \(\psi\) of every full matrix \(A\) over \(S\) is invertible over \(R\), then \(\psi\) extends uniquely to a map \(\psi\colon \mathcal{D}_S \to R\).
\end{prop}

If \(\psi\colon S \to R\) is a homomorphism of rings and \(\rk\) is a Sylvester matrix rank function on \(R\), then the precomposition with \(\psi\) defines a Sylvester matrix rank function on \(S\), denoted \(\psi^\#\rk\). Given a Sylvester matrix rank function \(\rk'\) on \(S\), we say that \(R\) is an envelope for \(\rk'\) if there exists \(\psi\) and \(\rk\) as above such that \(\rk'  = \psi^\#\rk\). 

We recall that a ring \(\mathcal{U}\) is von Neumann regular if for every \(a \in \mathcal{U}\) there exists \(b \in \mathcal{U}\) such that \(aba = a\).  In particular, every division ring is von Neumann regular. In \cite{jaikin-zapirainExplicitConstructionUniversal2020}, the first author proved the following:

\begin{thm}[{\cite[Cor. 2.5]{jaikin-zapirainExplicitConstructionUniversal2020}}]\label{thm-jaikin-criterion-universality} Let \(S\) be a ring and \(\rk\) be a Sylvester matrix rank function on \(S\) with an envelope \(\phi\colon S \to \mathcal{U}\) which is a von Neumann regular ring. Assume that:
\begin{enumerate}[label=(\arabic*)]
    \item \(\Tor_1^S(\mathcal{U},\mathcal{U}) = 0\), and
    \item for any finitely generated left or right \(S\)-submodule \(M\) of \(\mathcal{U}\) and any exact sequence \(0 \to I \to S^n \to M \to 0\), \(I\) is a free \(S\)-module.
\end{enumerate}
Then \(\rk = \irk\). In particular, \(S\) is a Sylvester domain and the division closure of \(\phi(S)\) in \(\mathcal{U}\) is \(S\)-isomorphic to \(\mathcal{D}_S\), the universal division \(S\)-ring of fractions.
\end{thm}

In the same paper it is shown the following converse to Theorem~\ref{thm-jaikin-criterion-universality}:

\begin{prop}[{\cite[Prop. 2.2]{jaikin-zapirainExplicitConstructionUniversal2020}}]\label{prop:andrei-sylvester} Let \(S\) be a Sylvester domain and \(\mathcal{D}\) it's universal division \(S\)-ring of fractions. Then:
\begin{enumerate}[label=(\arabic*)]
    \item For any left (resp. right) \(S\)-submodule \(M\) (resp. \(N\)) of \(\mathcal{D}^r\), we have that \(\Tor_1^S(N,M) = 0\).
    \item For any finitely generated left or right \(S\)-submodule \(M\) of \(\mathcal{D}^r\) and any exact sequence \(0 \to I \to S^n \to M \to 0\), \(I\) is a (set-theoretic) union of submodules isomorphic to \(S^k\) where \[k = n- \dim M = \dim I\,,\] where \(\dim\) is the Sylvester module rank function associated to the inner rank in \(S\).
\end{enumerate}
\end{prop}
We can see Theorem~\ref{thm-jaikin-criterion-universality} as a homological criterion to determine whether \(S\) is a Sylvester domain and \(\phi\colon S \to \mathcal{U}\) is an universal embedding for \(S\).

\begin{rem}\label{rem-on-p-adic} A big technical difficulty in proving Theorem~\ref{main} is establishing the validity of the condition (2) of Theorem~\ref{thm-jaikin-criterion-universality} for the embedding \(\Falgebra{G} \to \mathcal{Q}\) we construct in Section~\ref{sec:group-embeddings}. If \(G\) is a non-trivial pro-\(p\) \(p\)-adic analytic group and \(\Falgebra{G}\) is a Sylvester domain, then \(\cd G \leq 2\) by Proposition~\ref{prop:cohomologial-dim-sylvester} and hence \(G \simeq \Z_p \rtimes \Z_p\) or \(G \simeq \Z_p\). In particular, \(G\) is free-by-cyclic and Theorem~\ref{main} applies, showing that those are precisely the \(p\)-adic Lie groups \(G\) for which \(\Falgebra{G}\) is a Sylvester domain. However, for any torsion-free \(p\)-adic Lie group \(G\) the completed group algebra \(\Falgebra{G}\) is a Noetherian domain, and hence it possesses a classical Ore ring of fractions \(\mathcal{Q}\). For \(G = \Z_p \rtimes \Z_p\) or \(G = \Z_p\) it is straightforward to see that the map \(\Falgebra{G} \to \mathcal{Q}\) satisfies the conditions of Theorem~\ref{thm-jaikin-criterion-universality}. Since \(\mathcal{Q}\) is flat over \(\Falgebra{G}\), we have \(\Tor_1^{\Falgebra{G}}(\mathcal{Q},\mathcal{Q}) = 0\). Moreover, any finitely generated \(G\)-submodule \(M\) of \(\mathcal{Q}\) is isomorphic to a finitely generated \(G\)-submodule of \(\Falgebra{G}\) by the Ore condition. Observe that any submodule of a finitely generated \(\Falgebra{G}\) is finitely generated and hence closed and profinite. Since \(\Falgebra{G}\) has global dimension at most \(2\), by \cite[Rem. after Thm. 3.5]{brumerPseudocompactAlgebrasProfinite1966} we have \[\CTor_1^{\Falgebra{G}}(I,\F_p) \simeq\Tor_1^{\Falgebra{G}}(I,\F_p) \simeq \Tor_2^{\Falgebra{G}}(M,\F_p) \leq \Tor_2^{\Falgebra{G}}(\Falgebra{G},\F_p) = 0\,,\] where \(\CTor\) denotes the derived functor of the completed tensor product. By \cite[Prop. 3.1]{brumerPseudocompactAlgebrasProfinite1966}, \(I\) is projective, and thus \(I\) is a free \(\Falgebra{G}\)-module since the latter is a local ring. For the fact that \(\Falgebra{G}\) is a local Noetherian domain of global dimension \(\cd G\), see \cite{ardakovRingTheoreticPropertiesIwasawa2006}. The fact that \(\rk_G = \rk_{\mathcal{Q}}\) is also true for any \(p\)-adic analytic group, and is a result of M. Harris (\cite[Lem. 1.10.1]{harrisAdicRepresentationsArising1979}).  
\end{rem}

\section{Sylvester domains and completed group algebras}\label{sec-Sylvester-virtual}
In this section we study pro-$p$ groups $G$ for which $\Falgebra{G}$ is a Sylvester domain or $\rk_G=\irk_{\Falgebra{G}}$.
It is clear that the second condition is stronger, but we believe, that in fact they are equivalent.
The main objective of this section is to show that  both conditions are commensurability invariants.
 As far as the authors are aware, no analogue of  this results is known for abstract group rings.
We start with a key lemma.

\begin{lem}\label{lem:sylvester-pro-p-is-coherent} Let \(G\) be a pro-\(p\) group and suppose that  \(\Falgebra{G}\) is a Sylvester domain with universal division ring of fractions \(\mathcal{D}_{\Falgebra{G}}\). Then, for every finitely generated left (right) \(\Falgebra{G}\)-submodule \(M\) of \(\mathcal{D}_{\Falgebra{G}}^m\) and short exact sequence \[0 \to I \to \Falgebra{G}^n \to M \to 0\,,\] the \(\Falgebra{G}\)-module \(I\) is free of finite rank.
\end{lem}

The main step for obtaining this result is: 


\begin{lem}\label{Iclosed}
The submodule $I$ is closed in $\Falgebra{G}^n$
\end{lem}

\begin{rem*} We provide two proofs of the lemma. The first one is much easier but requires an additional condition that $\irk_{\Falgebra{G}}=\rk_G$. We notice that this is the case needed for the proof of Theorem \ref{main}. The second proof of Lemma \ref{Iclosed} does not require any additional hypothesis but uses a non-trivial result, proved in the appendix, that  $\mathcal{D}_{\Falgebra{G}}$ admits a Hausdorff ring topology such that the embedding \(\Falgebra{G} \to\mathcal D_{\Falgebra{G}} \) is continuous.\end{rem*}

\begin{proof}[First proof of Lemma~\ref{Iclosed}]
In this proof we not only assume that $\Falgebra{G}$ is a Sylvester domain but that $\irk_{\Falgebra{G}}=\rk_G$.
Let $$G>N_1>N_2>\ldots$$ be a chain of normal open subgroups of $G$ with trivial intersection.
 By \cite[Proposition 2.15]{Ja24}, 
  for every finitely generated left $\Falgebra{G}$-module $N$,
$$ \dim_G N= \lim_{i \to \infty} \frac{\dim_{\F_p} \F_p \otimes_{\Falgebra{N_i}}N}{|G\colon N_i|}.$$
 If \(\overline{I}\) denotes the closure of \(I\) in \(\Falgebra{G}^n\), let \(\overline{M} = \Falgebra{G}^n/\overline{I}\). Since the image \(J_i\) of \(I\) and \(\overline{I}\) in \(\F_p[G/N_i]^n\) coincide, we get
\begin{align*}
    \dim_{ \mathcal D_{\Falgebra{G}} }(\mathcal D_{\Falgebra{G}}  \otimes_{\Falgebra{G}} \overline{M} )
    &= \dim_G \overline{M}\\
    &= \lim_{i \to \infty} \frac{\dim_{\F_p} \F_p \otimes_{\Falgebra{N_i}} \overline{M}}{|G\colon N_i|}\\
     &=\lim_{i \to \infty} \frac{\dim_{\F_p} \F_p[G/N_i]^n/J_i}{|G\colon N_i|}\\
    &=\lim_{i \to \infty} \frac{\dim_{\F_p} \F_p \otimes_{\Falgebra{N_i}} M}{|G\colon N_i|}\\
    &= \dim_G M = \dim_{ \mathcal D_{\Falgebra{G}}}(\mathcal D_{\Falgebra{G}}  \otimes_{\Falgebra{G}} M)\,.
\end{align*}
The embedding \(M \to\mathcal D_{\Falgebra{G}}^m\) factors through the map \(M \to\mathcal D_{\Falgebra{G}}  \otimes_{\Falgebra{G}} M\) given by \(m \mapsto 1 \otimes m\), so the latter must also be injective. Hence, any \(x \in \overline{I}\) not in \(I\) is such that \(1 \otimes x\) is a non-zero element of
 \( \mathcal D_{\Falgebra{G}} \otimes_{\Falgebra{G}} M\). This element lies in the kernel of the surjection
  \(  \mathcal D_{\Falgebra{G} }\otimes_{\Falgebra{G}} M \to\mathcal D_{\Falgebra{G}}  \otimes_{\Falgebra{G}} \overline{M}\), which is an isomorphism by comparing dimensions. Therefore, \(I = \overline{I}\) is closed in \(\Falgebra{G}^n\).
\end{proof}

\begin{proof}[Second proof of Lemma~\ref{Iclosed}]

 We can identify \( \mathcal D_{\Falgebra{G}}\) with the universal localization of \(\Falgebra{G}\) at the set of all full matrices. Thus, 
by Theorem~\ref{hausd},   \(\mathcal D_{\Falgebra{G}}\) admits a Hausdorff ring topology such that the embedding \(\Falgebra{G} \to\mathcal D_{\Falgebra{G}} \) is continuous. Hence, if \(m_1\,,\ldots\,,m_n\) is a set of generators for \(M\), the map \(\Falgebra{G}^n \to M\) sending \((d_1,\ldots,d_n)\) to \(d_1m_1 + \cdots + d_nm_n\) is also continuous. In particular, the kernel \(I\) of this map is closed in \(\Falgebra{G}^n\).
 \end{proof}
 
\begin{proof}[Proof of Lemma~\ref{lem:sylvester-pro-p-is-coherent}]
We recall that by Proposition~\ref{prop:andrei-sylvester} the \(\Falgebra{G}\)-module \(I\) is the direct union of submodules $I_i$ isomorphic to \(\Falgebra{G}^k\).   We claim that \(I\) is itself isomorphic to \(\Falgebra{G}^k\).

Let \(J = I/I_GI\), where \(I_G\) is the augmentation ideal of \(\Falgebra{G}\). Since \(I\) is the direct union of submodules  isomorphic to \(\Falgebra{G}^k\), the \(\F_p\)-vector space \(J\) must be isomorphic to the direct union of subspaces of dimension at most \(k\). Therefore  \(\dim_{\F_p} J \leq k\) and there exists $i$ such that $J=I_i+I_GI$.  Since $I$ is closed by Lemma~\ref{Iclosed}, it is finitely generated.
Hence, by Nakayama's lemma, \(I\) must  be equal to  $I_i$.
 \end{proof}

The following observation  is an immediate consequence of the previous lemmas:

\begin{prop}\label{prop:cohomologial-dim-sylvester} Let \(G\) be a finitely generated pro-\(p\) group. If \(\Falgebra{G}\) is a Sylvester domain, then \(G\) is finitely presented and \(\cd G \leq 2\).
\end{prop}
\begin{proof} Apply Lemma~\ref{lem:sylvester-pro-p-is-coherent} with \(M\) being the augmentation ideal of \(\Falgebra{G}\). Observe that \(G\) is finitely presented if and only if \[H_2(G,\F_p) = \Tor_2^{\Falgebra{G}}(\F_p,\F_p) \simeq \Tor_1^{\Falgebra{G}}(M,\F_p) \leq \Tor_0^{\Falgebra{G}}(I,\F_p)\] is finite, that is, if \(I\) is finitely generated.
\end{proof}

Before we proceed, we recall the definition of a crossed product between a ring and a group.
 Let \(R\) be an associative and unital ring and \(G\) be a group. We say that $S$ is a \emph{crossed product} of $R$ and $G$ if $S=\oplus_{g\in G} S_g$  such that for every $g,h\in G$,  $S_gS_h\subseteq S_{gh}$, $S_1= R$ and for every $g\in G$ there exists a unit $u_g\in S_g$. We will write   \(S\cong R*G\). 
 It is clear that the multiplication is uniquely determined  by the rules \[u_{g_1}u_{g_2} = u_{g_1g_2}\tau(g_1,g_2)\,,\\ ru_{g_1} = u_{g_1}\sigma(g_1)(r)\,,\] where \(\tau\colon G \times G \to R^\times\) and \(\sigma\colon G \to \Aut(R)\) are functions satisfying the identities
\begin{equation}\label{eq:identities-crossed-product}
\begin{gathered}
    \tau(g_1g_2,g_3)\sigma(g_3)(\tau(g_1,g_2)) = \tau(g_1,g_2g_3)\tau(g_2,g_3)\,,\\
    \sigma(g_2)(\sigma(g_1)(r)) = \tau(g_1,g_2)^{-1}\sigma(g_1g_2)(r)\tau(g_1,g_2)\,,
\end{gathered}
\end{equation}
for every \(g_1,g_2,g_3 \in G\) and \(r \in R\) -- this follows directly from the associativity of the product in \(S\) (\cite[Lem. 1.1]{passmanInfiniteCrossedProducts2013}). Since $u_g$ are invertible in $S$, \(\{u_g\mid g \in G\}\) is a free basis of \(S\) as a left \(R\)-module.

Let \(G\) be a pro-\(p\) group,  \(U\)   a normal open subgroup of \(G\) and $T$ a transversal of \(U\) in \(G\). Since $\Falgebra{G}=\oplus_{t\in T} \Falgebra{U}t$, we obtain that 
  \(\Falgebra{G}\cong \Falgebra{U} * G/U\). For every \(g \in G\), let \(\overline{g}\) be it's representative in \(T\).   For any $t\in T$,   we put $u_{tU}=t$. In this case the maps \(\tau\) and \(\sigma\) can be explicitly described as follows:
 \(\sigma(g_1U)\) is conjugation by \(\overline{g_1}\) and \[\tau(g_1U,g_2U) = (\overline{g_1g_2})^{-1}\cdot \overline{g_1}\cdot \overline{g_2}\,.\]

\begin{prop}\label{prop:crossed-division-ring} Let \(G\) be a pro-\(p\) group and \(U\) a normal open subgroup of \(G\). Suppose \(\Falgebra{U}\) has a universal division ring of fractions \(\mathcal{Q}\). Then:
\begin{enumerate}[label=(\alph*)]
    \item The conjugation action of \(G\) on \(\Falgebra{U}\) extends to a homomorphism \(G \to \Aut(\mathcal{Q})\).
    \item For a fixed transversal \(T\) of \(G/U\), the maps from the crossed product decomposition \(\Falgebra{G} \simeq \Falgebra{U} * G/U\) can be extended to maps \(\sigma\colon G/U \to \Aut(\mathcal{Q})\) and \(\tau\colon G/U \times G/U \to \mathcal{Q}^\times\) satisfying the identities~(\ref{eq:identities-crossed-product}). In particular, the crossed product \(\mathcal{Q}*G/U\) exists and \(\Falgebra{G}\) embeds into it.
    \item If \(E\) is the ring of right \(\mathcal{Q}\)-endomorphisms of \(\mathcal{Q}*G/U\), then left multiplication induces an embedding of rings \(\mathcal{Q}*G/U \to E\) such that it makes \(E\) a free \(\mathcal Q*G/U\)-module on both sides.
    \item As a right \(\Falgebra{G}\)-module, the crossed product \(\mathcal{Q}*G/U\) is isomorphic to the induced module \(\mathcal{Q} \otimes_{\Falgebra{U}} \Falgebra{G}\).
\end{enumerate}
\end{prop}
\begin{proof} (a) This follows from Proposition \ref{extension}.

(b) Since \(U \subseteq \mathcal{Q}^\times\), one can take the same \(\tau\) as in the crossed product decomposition of \(\Falgebra{G}\). It is then only a matter of checking whether or not the map \(G/U \to \Aut(\mathcal{Q})\) induced through \(\sigma\) by \(T\) also satisfies the second identity in~(\ref{eq:identities-crossed-product}). However, for each fixed pair \(g_1,g_2 \in G\), both the left and the right hand side of that identity define ring automorphisms of \(\mathcal{Q}\) which coincide on \(\Falgebra{U}\). Since the inclusion of \(\Falgebra{U}\) into \(\mathcal{Q}\) is an epimorphism of rings, both sides must indeed be equal as automorphisms of \(\mathcal{Q}\).

(c) Let $T$ be a transversal of $U$ in $G$. Then  \(T = \{g_1,\ldots,g_n\}\) is a $\mathcal Q$-basis of $\mathcal{Q}*G/U$. for every $i=1,\ldots, n$ we define $\gamma_i\in E$ such that $\gamma_i(g_i)=\delta_{ij}g_j$.
We identify the elements of the ring $\mathcal{Q}*G/U$ with its images in $E$.

We have that $E=\bigoplus_{i=1}^n (\mathcal{Q}*G/U) \circ \gamma_i$. To see it, it is enough to show that if $\alpha_1,\ldots, \alpha_n\in  \mathcal{Q}*G/U $ and 
$\alpha=\sum_{i=1}^n\alpha_i\circ \gamma_i=0$, then  $\alpha_1=\ldots =\alpha_n=0 $. 
This follows from the equality $\alpha(g_i)=\alpha_i g_i$. We also have that $E=\bigoplus_{i=1}^n \gamma_i\circ (\mathcal{Q}*G/U)  $, for which it is enough to show that if $\alpha_1,\ldots, \alpha_n\in  \mathcal{Q}*G/U $ and 
$\alpha=\sum_{i=1}^n\gamma_i\circ \alpha_i=0$, then  $\alpha_1=\ldots =\alpha_n=0 $. 
Write $\alpha_i=\sum_{j=1}^n \alpha_{ij}g_j$. Then
 $$\alpha(g_k)=\sum_{g_i^{-1}g_jg_k\in U}\alpha_{ij}\tau(g_k,g_k)g_i.$$
 Now it is clear that if all $\alpha(g_k)=0$ then for every $i,j$, $\alpha_{ij}=0$.
 
(d)  Just observe that the map \( \mathcal{Q} \otimes_{\Falgebra{U}} \Falgebra{G} \to  \mathcal{Q}*G/U \) sending $a\otimes b$ to $ab$ is an isomorphism of right $\Falgebra{G}$-modules, whose inverse is given by sending \(q_1g_1 + \cdots + q_ng_n\) to \(q_1 \otimes g_1 + \cdots + q_n \otimes g_n\).

\end{proof}

\begin{thm}\label{thm:sylvester-virtual} Let \(G\) be a torsion-free finitely generated pro-\(p\) group. Then the following are equivalent:
\begin{enumerate}[label=(\Alph*)]
    \item \(\Falgebra{G}\) is a Sylvester domain.
    \item \(\Falgebra{U}\) is a Sylvester domain for every open subgroup \(U\) of \(G\).
    \item \(\Falgebra{U}\) is a Sylvester domain for some open subgroup \(U\) of \(G\).
\end{enumerate}
Moreover, in case one of the above holds and \(U\) is an open normal subgroup of \(G\) with a universal embedding \(\Falgebra{U} \to \mathcal{Q}\), then \(\mathcal{Q}* G/U\) is a division ring and the induced map \(\Falgebra{G} \to \mathcal{Q} * G/U\) is a universal embedding.
\end{thm}
\begin{proof} \textit{(A) \(\implies\) (B)} Suppose first that \(\Falgebra{G}\) is a Sylvester domain with universal division ring of fractions \(\mathcal{Q}\). We want to show that the induced embedding \(\Falgebra{U} \to \Falgebra{G} \to \mathcal{Q}\) satisfies the hypothesis of Theorem~\ref{thm-jaikin-criterion-universality}. 

First, we claim that the right \(\Falgebra{G}\)-module \(\mathcal{Q} \otimes_{\Falgebra{U}} \Falgebra{G}\) is a submodule of \(\mathcal{Q}^r\) for some \(r\). Indeed, as a left \(\mathcal{Q}\)-vector space, it has a basis \(\{1 \otimes g\mid g \in T\}\) for some transversal \(T\) of \(G/U\). One then checks that the map that sends \(\sum_{g \in T} r_g \otimes g\) to \((r_gg)_{g \in T}\) in \(\mathcal{Q}^{|G\colon U|}\) is right \(\Falgebra{G}\)-equivariant. This implies that for every finitely generated right \(\Falgebra{U}\)-submodule \(M\) of \(\mathcal{Q}\), the right \(\Falgebra{G}\)-module \(M \otimes_{\Falgebra{U}} \Falgebra{G}\) is a submodule of \(\mathcal{Q}^{|G\colon U|}\). Hence \[\Tor_1^{\Falgebra{U}}(M,\mathcal{Q}) \simeq \Tor_1^{\Falgebra{G}}(M \otimes_{\Falgebra{U}} \Falgebra{G}, \mathcal{Q}) = 0\] by Proposition~\ref{prop:andrei-sylvester}. Since \(\mathcal{Q}\) is the direct limit of its finitely generated submodules, we conclude that \(\Tor_1^{\Falgebra{U}}(\mathcal{Q},\mathcal{Q})= 0\).

Now take a short exact sequence \[0 \to I \to \Falgebra{U}^n \to M \to 0\,.\] Tensoring once more with \(\Falgebra{G}\) over \(\Falgebra{U}\), we get a short exact sequence \[0 \to I\otimes_{\Falgebra{U}} \Falgebra{G} \to \Falgebra{G}^n \to M \otimes_{\Falgebra{U}} \Falgebra{G} \to 0\,.\] Since \(M \otimes_{\Falgebra{U}} \Falgebra{G}\) is a submodule of \(\mathcal{Q}^{|G\colon U|}\), by Lemma~\ref{lem:sylvester-pro-p-is-coherent} the module \(I \otimes_{\Falgebra{U}} \Falgebra{G}\) is free of finite rank. By tensoring over \(\Falgebra{G}\) with \(\F _p\), we conclude that \(I\) is finitely generated over \(\Falgebra{U}\).

Then, \(I\) is itself a finitely generated \(\Falgebra{U}\)-submodule of \(\mathcal{Q}^n\), so one can repeat the prior argument with the short exact sequence \[0 \to J \to \Falgebra{U}^k \to I \to 0\] to conclude that \(I\) must be finitely presented and hence of type \(FP_\infty\) over \(\Falgebra{U}\). Therefore, by \cite[Lem. 2.1]{brumerPseudocompactAlgebrasProfinite1966} we conclude that
 \[\CTor_1^{\Falgebra{U}}(I,\F_p) \simeq \Tor_1^{\Falgebra{U}}(I,\F _p) \simeq \Tor_1^{\Falgebra{G}}(I \otimes_{\Falgebra{U}} \Falgebra{G}, \F _p) = 0\,,\] where \(\CTor\) is the derived functor of the completed tensor product of profinite modules. Since $I$ is profinite, \(I\) must be projective over \(\Falgebra{U}\) by \cite[Prop. 3.1]{brumerPseudocompactAlgebrasProfinite1966} and hence free (and of finite rank) by Kaplansky's theorem on projective modules over local rings. By Theorem~\ref{thm-jaikin-criterion-universality}, \(\Falgebra{U}\) is a Sylvester domain and its division closure \(\mathcal{Q}_0\) inside \(\mathcal{Q}\) is \(\Falgebra{U}\)-isomorphic to it's universal division ring of fractions. 

\textit{(B) \(\implies\) (C)} is immediate.

\textit{(C) \(\implies\) (A)} Now, suppose that \(\Falgebra{U}\) is a Sylvester domain with universal division ring of fractions \(\mathcal{Q}\) for some open subgroup of \(G\). It suffices to prove the theorem for the case \(|G\colon U| = p\), so we may assume \(U\) is a normal subgroup of \(G\). Then, by Proposition~\ref{prop:crossed-division-ring}, the embedding \(\Falgebra{U} \to \mathcal{Q}\) induces an embedding \(\Falgebra{G} \to \mathcal{Q} * G/U \to E\) where \(E\) is the ring of right \(\mathcal{Q}\)-endomorphisms of \(\mathcal{Q}*G/U\), and we will check the conditions of Theorem~\ref{thm-jaikin-criterion-universality} for this embedding. The ring \(E\) is von Neumann regular as it is isomorphic to a matrix ring over a division ring, and since \(\mathcal{Q}* G/U\) is isomorphic to the induced module \(\mathcal{Q} \otimes_{\Falgebra{U}} \Falgebra{G}\) as right \(\Falgebra{G}\)-modules, we have:
\begin{align*}
\Tor_1^{\Falgebra{G}}(E,E) &\simeq \Tor_1^{\Falgebra{G}}((\mathcal{Q}*G/U)^p,(\mathcal{Q}*G/U)^p)\\
&\simeq \Tor_1^{\Falgebra{G}}(\mathcal{Q}*G/U,\mathcal{Q}*G/U)^{p^2} \\
&\simeq \Tor_1^{\Falgebra{G}}(\mathcal{Q} \otimes_{\Falgebra{U}} \Falgebra{G},\mathcal{Q}*G/U)^{p^2}\\
&\simeq \Tor_1^{\Falgebra{U}}(\mathcal{Q}, \mathcal{Q}^p)^{p^2} = 0\,.
\end{align*}

Now, take any finitely generated right \(\Falgebra{G}\)-submodule \(M\) of \(E \simeq (\mathcal{Q} * G/U)^p\) and a short exact sequence \[0 \to I \to \Falgebra{G}^n \to M \to 0\,.\] As an \(\Falgebra{U}\)-module, \(I\) must be free of finite rank by Lemma~\ref{lem:sylvester-pro-p-is-coherent}. Hence, \(I\) is finitely generated over \(\Falgebra{G}\). Taking another short exact sequence \[0 \to J \to \Falgebra{G}^k \to I \to 0\] and repeating this argument, one obtains that \(J\) must also be finitely generated, so that \(I\) is finitely presented. In particular, \[  \Tor_1^{\Falgebra{G}}(I,\F _p) \simeq \Tor_2^{\Falgebra{G}}(M,\F _p)\,.\]

By Proposition \ref{prop:cohomologial-dim-sylvester}, $U$ is of cohomological dimension 2. By Serre's theorem \cite{Se65} (see also \cite{haran_proof_1990}), $G$ is also of cohomological dimension 2.  Thus, since \(M\) is a finitely presented submodule of \((\mathcal{Q}*G/U)^p\), there is an exact sequence of \(\Tor\) groups: \[0 \to \Tor_2^{\Falgebra{G}}(M, \F _p) \to \Tor_2^{\Falgebra{G}}(\mathcal{Q} * G/U, \F _p)^p \simeq \Tor_2^{\Falgebra{U}}(\mathcal{Q},\F _p)^p\,.\] By Proposition~\ref{prop:andrei-sylvester}, the \(\Falgebra{U}\)-module \(\mathcal{Q}\) has weak dimension at most \(1\), so that all the \(\Tor\) groups above vanish. Since \(\Tor_1^{\Falgebra{G}}(I,\F _p) = 0\) and $I$ is finitely generated, the \(\Falgebra{G}\)-module \(I\) is free of finite rank and one can apply Theorem~\ref{thm-jaikin-criterion-universality}.

Observe that \(\mathcal{Q}*(G/U)\) is then a domain since it embeds into a division ring. Given that it is also a finite dimensional \(\mathcal{Q}\)-algebra, it must be itself a division ring, so we've shown that \(\mathcal{Q}*(G/U)\) is the universal division ring of fractions of \(\Falgebra{G}\). 
\end{proof}

\begin{cor}\label{cor-atiyah-invariant} Let \(G\) be a torsion-free finitely generated pro-\(p\) group. Then the following are equivalent:
\begin{enumerate}[label=(\Alph*')]
    \item \(\rk_G = \irk_{\Falgebra{G}}\).
    \item \(\rk_U = \irk_{\Falgebra{U}}\) for every open subgroup \(U\) of \(G\).
    \item \(\rk_U = \irk_{\Falgebra{U}}\) for some open subgroup \(U\) of \(G\).
\end{enumerate}
\end{cor}
\begin{proof} We first observe that if \(\rk_U = \irk_{\Falgebra{U}}\), then the inner rank is a Sylvester matrix rank function on \(\Falgebra{U}\) and therefore it is a Sylvester domain having a universal division ring of fractions.

(A') \(\implies\) (B') Let \(\mathcal{Q}\) be the universal division ring of fractions of \(\Falgebra{G}\), \(A\) be a \(n\times m\) matrix over \(\Falgebra{U}\) (which we also see as a matrix over \(\Falgebra{G}\)) and define the left \(\Falgebra{U}\)-module \(M = \Falgebra{U}^m/\Falgebra{U}^nA\). Note that \(\Falgebra{G} \otimes_{\Falgebra{U}} M \simeq \Falgebra{G}^m/\Falgebra{G}^n A\). Hence:

\begin{align*}
    \dim_U M &= \frac{\dim_U \Falgebra{G} \otimes_{\Falgebra{U}} M}{|G\colon U|}\\
    &= \dim_G \Falgebra{G} \otimes_{\Falgebra{U}} M\\
    &= \dim_{\mathcal{Q}} \mathcal{Q} \otimes_{\Falgebra{G}} \Falgebra{G} \otimes_{\Falgebra{U}} M\\
    &= \dim_{\mathcal{Q}} \mathcal{Q} \otimes_{\Falgebra{U}} M\,.\\
\end{align*}

(B') \(\implies\) (C') Is immediate.

(C') \(\implies\) (A') We've shown in Theorem~\ref{thm:sylvester-virtual} that if \(\mathcal{Q}\) is the universal division ring of fractions of \(\Falgebra{U}\), then \(\mathcal{Q} * G/U\) is the universal division ring of fractions of \(\Falgebra{G}\). Hence, for any \(n\times m\) matrix \(A\) over \(\Falgebra{G}\) which we see as a matrix over \(\mathcal{Q}*(G/U)\) and as a \(n|G\colon U|\times m|G\colon U|\) matrix over \(\Falgebra{U}\):
\[\irk_{\Falgebra{G}} A = \rk_{\mathcal{Q}*(G/U)} A = \frac{\rk_{\mathcal{Q}} A}{|G\colon U|} = \frac{\irk_{\Falgebra{U}} A}{|G\colon U|}= \frac{\rk_U A}{|G\colon U|} = \rk_G A\,.\qedhere\]
\end{proof}

\section{Mild flag pro-\texorpdfstring{\(p\)}{p} groups}\label{sec:strongly-free-sequences}

Let \(G\) be a pro-\(p\) group with a closed normal subgroup \(N\) such that \(N = F(x_1,x_2,\ldots)\) is a free pro-\(p\) group of at most countable rank and \(G/N \simeq \Z_p\). Let \(g \in G\) be such that \(gN\) topologically generates \(G/N\), and consider the group algebra \(S = \Falgebra{G}\). We want to decompose it as a skew power series algebra over \(R = \Falgebra{N}\) with the automorphism \(\sigma\) of \(S\) is given by \(x \mapsto g^{-1}xg\) and the derivation \(\delta\) is \(\sigma - \operatorname{id}\). If we let \(s = g - 1 \in S\), observe that we get the relation \[xs = s\sigma(x) + \delta(x)\] for any \(x \in S\). 

To better understand the filtered structure of \(S\) and \(R\) and their respective associated graded rings, we fix the notation for the following filtration on \(S\): if \(I_G\) is the augmentation ideal of \(S\), we let \(S_k = I_G^k\) for \(\geq 0\) and \(S_k = S\) for \(k < 0\). We define the continuous function \(w\colon S \to [0,1]\) through: \begin{equation}\label{eq-defining-valuation} 
    w(a) = p^{-\sup\{k \in \Z \mid a \in S_k\}}\,.
\end{equation} It is directly verified that \(w(a)\) is a \emph{filtration function} in the sense of \cite[Def. I.2.1.1]{lazardGroupesAnalytiquesPadiques1965} (or a valuation in the sense of \cite[Sec. 2.2]{ershovGroupsPositiveWeighted2013}), that is, \(w\) satisfies:
\begin{enumerate}[label=(\roman*)]
    \item \(w(a) = 0\) if and only if \(a = 0\);
    \item \(w(1) = 1\);
    \item \(w(a + b) \leq \max\{w(a),w(b)\}\) for any pair \(a,b \in S\).
    \item \(w(ab) \leq w(a)w(b)\).
\end{enumerate}
The conditions above also imply a stronger version of (iii) called the strong ultrametric inequality: if moreover \(w(a) \neq w(b)\), then \(w(a+b) = \max\{w(a),w(b)\}\). Following \cite[I.2.2.1]{lazardGroupesAnalytiquesPadiques1965}, we shall say that \(w\) is a \emph{valuation} if \(w\) satisfies a stronger version of (iv): \(w(ab) = w(a)w(b)\) for any pair \(a,b \in S\). This is the definition also adopted in \cite{labuteMildPropgroupsGalois2006}.

\begin{prop}\label{prop-decomposition-power-series} The completed group algebra \(S\) is isomorphic to the right skew power series ring \(R[\![s; \sigma, \delta]\!]\).\end{prop}
\begin{proof} From the homeomorphism \(G \sim N \times G/N\), we obtain that \(S\) is the free profinite \(R\)-module on the profinite space \(G/N \simeq \{1,g,g^2,\ldots\}\). Through a base change, we get that the set of powers of \(s = g - 1\) also form a topological basis for \(S\) over \(R\), from which we can identify it as a topological abelian group with the group of formal power series in \(s\) over \(R\). Hence, it suffices to show that the formula~(\ref{eq-power-series-mult}) defining the multiplication always converges. For this, it is also sufficient to check that \(\delta\) is topologically nilpotent in the induced filtration of \(R\), that is, \(\delta^k\) converges to zero pointwise. Observe that \(I_G\) and \(R\) are \(\sigma\)-invariant because \(\sigma\) is an inner automorphism of \(S\) and \(N\) is normal in \(G\).

To show that \(\delta\) is topologically nilpotent, it also suffices to show that \(\delta(S) \subseteq I_G\) and that \(\delta(I_G) \subseteq I_G^2\), for then we inductively get from the inclusions \(\delta(I_G^k) \subseteq \delta(I_G^{k-1})I_G + I_G^{k-1}\delta(I_G)\subseteq I_G^{k+1}\) for all \(k \geq 2\) that \(\delta^k(S) \subseteq I_G^k\). This gives us that \(\delta^k\) converges to zero uniformly and hence is topologically nilpotent in \(S\) and thus in \(R\).
\end{proof}

The proof of Proposition~\ref{prop-decomposition-power-series} also shows:

\begin{cor}\label{cor-derivation-retracting} For every free-by-\(\Z_p\) group \(G\) and any non-zero \(a \in S\), we have \(w(\delta(a)) < w(a)\), where \(w\) is the function defined in~(\ref{eq-defining-valuation}).
\end{cor}

To describe some cases in which \(w(a)\) is actually a valuation, we will make use of the following definition given in \cite[Def. 1.1]{forre_strongly_2011}:

\begin{defn}\label{defn:strongly-free} Let \(\F_p\langle a_1,\ldots,a_n\rangle\) be the free \(\F_p\)-algebra on \(n\) generators with augmentation ideal
\(I = (a_1,\ldots,a_n)\). Choose \(l\) elements \(\rho_1,\ldots,\rho_l \in I\) and let \(J = (\rho_1,\ldots,\rho_l)\) be the two-sided ideal generated by them. We say that the sequence of elements \(\rho_1,\ldots,\rho_l\) is \emph{strongly free} if \(J/JI\) is a free \(\F_p\langle a_1,\ldots,a_n\rangle/J\)-module with basis \(\{\rho_i + JI\}\).

Let \(F = F(g_1,\ldots,g_n)\) be a free pro-\(p\) group on \(n\) generators. With respect to the filtration induced by augmentation ideal, the graded ring \(\Gr(\Falgebra{F})\) is isomorphic to the free \(\F_p\)-algebra on \(n\) generators \(a_i = \overline{g_i-1}\). A minimal pro-\(p\) presentation \[\langle g_1,\ldots,g_n \mid r_1,\ldots,r_l\rangle\] is a called a \emph{strongly free presentation} if the homogeneous components \(\rho_i = \overline{r_i - 1}\) form a strongly free sequence in \(\Gr(\Falgebra{F})\). A pro-\(p\) group is \emph{mild} if it has a strongly free presentation.
\end{defn}

To highlight the choice of filtration on \(\Falgebra{F}\), one says in that case that \(G\) is a mild pro-\(p\) group with respect to the \(p\)-Zassenhaus filtration, that is, the filtration by dimension series mod \(p\) -- cf. \cite[Def. 4.4]{minac_mild_2022}, \cite[Lem. 1.3 and Rem. 1.6]{forre_strongly_2011} and \cite[Def. 1.1]{labuteMildPropgroupsGalois2006}. We will concern ourselves solely with mild groups with respect to this filtration.

We observe that if \(\Gr(\Falgebra{G})\) is an integral domain, then \(w\) must be a valuation. The mildness condition gives us:

\begin{prop}[{\cite[Prop. 4.5]{minac_mild_2022}, cf. \cite[Thm. 2.11]{gartnerHigherMasseyProducts2015}}]\label{prop:sf-graded} If \[G \simeq \langle g_1,\ldots,g_n\mid r_1,\ldots,r_l\rangle\] is a strongly free presentation, then \(\Gr(\Falgebra{G}) \simeq \Gr(\Falgebra{F})/J\), where \(J\) is the two-sided ideal generated by the \(\rho_i = \overline{r_i-1}\).
\end{prop}

Hence, our strategy is to use mildness and an explicit description of \(\Gr(\Falgebra{F})/J\) to obtain that \(\Gr(\Falgebra{G})\) is a domain, and thus that \(w\) is a valuation on \(\Falgebra{G}\). Before establishing mildness, we prove a technical lemma that will give us a ``canonical'' minimal presentation for free-by-\(\Z_p\) pro-\(p\) groups. 

\begin{defn}\label{def:flag-presentation} A \emph{flag presentation} of a pro-\(p\) group \(G\) is a finite presentation given by the quotient of the free pro-\(p\) group \(F\) on a set of generators \(\{x_1,\ldots,x_n,g\}\) by \(l\) relations of the form \begin{equation}\label{eq:flagpresentation}[x_i,\underbrace{g,\ldots,g}_{a_i\text{ times}}] = h_i\text{ for }h_i \in \Phi(\widetilde{N})\text{ and }1\leq i \leq l\,,\end{equation} where the \(a_i\) are positive integers and \(\widetilde{N}\) is the normal subgroup of \(F\) generated by \(\{x_1,\ldots,x_n\}\). If all the \(a_i\) are equal to \(1\), we shall say that the flag presentation is \emph{mild}. A pro-\(p\) group with a mild flag presentation will be called a \emph{mild flag pro-\(\bm{p}\) group}.\end{defn}

If \(G\) has a mild flag presentation as in~(\ref{eq:flagpresentation}), then \(G\) is a mild group. Indeed, if \(l = 0\), the presentation shows that the group is free and vacuously satisfies the hypothesis of mildness. Otherwise, let \(\eta_i\) and \(\gamma\) be the elements in \(H^1(G,\F_p)\) that are dual to the basis \(\{x_i,g\mid 1 \leq i \leq n\}\) of \(G/\Phi(G)\):\[\eta_i(x_j) = \delta_{i,j}\,, \eta_i(g) = 0\,, \gamma(x_i) = 0\,, \gamma(g) = 1\,.\] There is an \(\F_p\)-vector space decomposition of \(H^1(G,\F_p) \simeq V \oplus W\) such that \(V = \langle \gamma \rangle\) and \(W = \langle \eta_i\mid 1 \leq i \leq n\rangle\), and the identities in~(\ref{eq:flagpresentation}) show that the restriction of the cup product to \(V \otimes W \to H^2(G,\F_p)\) is surjective and that \(\gamma\cup \gamma = 0\) by \cite[Prop. 3.9.13]{neukirchCohomologyNumberFields2008}. Hence, \(G\) satisfies the cup-product criterion of \cite[Prop. 5.8]{minac_mild_2022} to being a mild group (cf. \cite[Sec. 6]{forre_strongly_2011}). Later on we shall see that mild flag presentations are also strongly free presentations, and can be used to describe the graded ring \(\Gr(\Falgebra{G})\).
 
\begin{lem}\label{lem:virtual-flag} Every finitely generated free-by-\(\Z_p\) pro-\(p\) group \(G\) has a flag presentation. Moreover, \(G\) has a normal open subgroup \(U\), inverse image of an open subgroup of \(\Z_p\), such that \(U\) has a mild flag presentation.
\end{lem}
\begin{proof}
Let \(G \simeq N\rtimes G/N\) be a free-by-\(\Z_p\) pro-\(p\) group with \(G/N \simeq \overline{\langle gN \rangle} \simeq \Z_p\) and \(N\) free pro-\(p\). Let \(A = N/\Phi(N)\) and consider it as an \(\Falgebra{G/N}\)-module through the action of \(g\). Since \(\Falgebra{G/N}\) is a PID and \(A\) is finitely generated over it, it can be decomposed as a direct sum \(\Falgebra{G/N}^k \oplus \bigoplus_{i=1}^l \Falgebra{G/N}/(d_i)\) with $0\ne d_i\in I_{G/N}$ for all $i$. It's Pontryagin dual \(\Hom(N/\Phi(N), \bbq/\Z)\) must then decompose as \(Q^k \oplus \bigoplus_{i=1}^l \Falgebra{G/N}/(d_i)\) where \(Q\) is an injective \(\Falgebra{G/N}\)-module. 

Since $0\ne d_i\in I_{G/N}$, we obtain that \(H^1(G/N,\Falgebra{G/N}/(d_i)) \simeq \F_p\) for all \(i\). By the Lyndon-Hochschild-Serre spectral sequence, we have: \begin{align}\label{eq-LHS-sequence}
    H^2(G,\F_p) &\simeq H^1(G/N, H^1(N,\F_p))\notag\\ &\simeq H^1(G/N, \Hom(N/\Phi(N),\bbq/\Z)) \\ &\simeq \bigoplus_{i=1}^l H^1(G/N, \Falgebra{G/N}/(d_i)) \simeq \F_p^l\,.\notag
\end{align}

Choose representatives \(x_1,\ldots,x_k,z_1,\ldots,z_l\) in \(N\) for the \(\Falgebra{G}\)-cyclic generators of each factor in \(N/\Phi(N)\) such that \(\{x_1,\ldots,x_k,z_1,\ldots,z_l,g\}\) is a minimal generating set for \(G\) -- this is possible because the map \(A \to G/\Phi(G)\) is a homomorphism of \(G/N\)-modules. If \(A\) is a free \(\Falgebra{G/N}\)-module, then~(\ref{eq-LHS-sequence}) shows that \(H^2(G,\F_p) = 0\) and therefore \(G\) is a free pro-\(p\) group with the mild flag presentation \[G \simeq \langle x_1\,\ldots,x_k,g\mid \varnothing\rangle\,.\] Hence, we can assume that \(A\) is not a free \(\Falgebra{G/N}\)-module, and in particular that \(H^2(G,\F_p)\) is non-zero.

Since the \(G/N\)-action on the finite module \(\Falgebra{G/N}/(d_i)\) is unipotent, there exists some positive \(a_i\) such that \([z_i,\underbrace{g,\ldots,g}_{a_i\text{ times}}]\) is a fixed point, or equivalently, there exists \(h_i\) in \(\Phi(N)\) such that \begin{equation}\label{eq-we-found-flag}
    [z_i,\underbrace{g,\ldots,g}_{a_i + 1\text{ times}}] = h_i\,.
\end{equation}
Then, by the isomorphism in~(\ref{eq-LHS-sequence}), we find that the identities~(\ref{eq-we-found-flag}) for \(1 \leq i \leq l\) form a generating set of relations for \(G\), yielding the desired flag presentation.

By taking in \(G\) the intersection \(U\) of the stabilizers of the \(G\)-action in each finite module \(\Falgebra{G/N}/(d_i)\), we can assume that each non-free direct factor of \(A\) is isomorphic to \(\F_p\), that is,
\[A \simeq \Falgebra{U/N}^{k|G\colon U|} \oplus \bigoplus_{i=1}^{l'} \F_p\] as \(\Falgebra{U}\)-modules. We note that \(N \leq U\) and that \(U \simeq N\rtimes U/N \simeq N \rtimes \Z_p\). By repeating the steps above, we obtain a mild flag presentation for \(U\).
\end{proof}

\begin{rem} Conversely, every pro-\(p\) group \(G\) with a flag presentation is also free-by-\(\Z_p\). If \(N\) is the normal subgroup of \(G\) generated by the \(x_i\), then \(G/N \simeq \Z_p\) so that \(G \simeq N \rtimes \Z_p\). Taking the non-negative powers of \(g\) as a continuous section \(G/N \to G\), the Reidemeister-Schreier rewriting process shows that \(N\) has a presentation with a generating set converging to \(1\) (\cite[Sec. 2.4]{ribesProfiniteGroups2010}) given by all the elements \(x_{i,k} = [x_i,\underbrace{g,\ldots,g}_{k\text{ times}}]\) and relations \[x_{i,a_i+k} = \left(\prod_{j=0}^{k-1} g^{j-k+1}x_{i,a_i+k-1-j}g^{k-1-j}\right)\cdot g^{-k}h_ig^{k}\] for \(1 \leq i \leq l\) and \(k\geq 0\). These relations allows us to eliminate the generators \(x_{i,a_i+k}\) for \(1 \leq i \leq l\) and \(k \geq 0\), showing that \(N\) is free pro-\(p\). Alternatively, the LHS spectral sequence gives us 
\begin{align*}
    l &\geq \dim_{\F_p} H^2(G,\F_p) \\&= \dim_{\F_p} H^2(N,\F_p)^{G/N} \oplus H^1(G/N, H^1(N,\F_p)) \\&\geq l + \dim_{\F_p} H^2(N,\F_p)^{G/N}\,,
\end{align*} where the last inequality is obtained by a computation of \(H^1(G/N, H^1(N,\F_p))\) as in Lemma~\ref{lem:virtual-flag}. Hence \(H^2(N,\F_p)^{G/N} = 0\), and since the cohomology groups are discrete torsion \(G/N\)-modules, one gets \(H^2(N,\F_p) = 0\) as desired.
\end{rem}

We now turn to the problem of showing that mild flag presentations of \(G\) are strongly free, and to characterize the graded ring of \(\Falgebra{G}\). We recall that \(\frakL = \bigoplus_{i\geq 1} D_i(F)/D_{i+1}(F)\) has the structure of a free restricted Lie algebra over \(\F_p\) in the sense of \cite[Sec. V.7]{jacobson_lie_1979}. A free generating set of this Lie algebra is the image \(\xi_1,\ldots,\xi_n,\gamma\) of \(x_1,\ldots,x_n,g\), and we can identify \(\Gr(\Falgebra{F})\) with the universal restricted enveloping algebra \(\mathcal{U}_\frakL\) of \(\frakL\) (\cite[Thm. 6.5]{lazard_sur_1954} and \cite[Thm. A.3.5]{lazardGroupesAnalytiquesPadiques1965}). If the elements \(\rho_1,\ldots,\rho_l\) associated to a presentation of \(G\) lie in \(\frakL\) and \(\frakr\) is the restricted ideal of \(\frakL\) they generate, then for \(J\) the ideal of \(\Gr(\Falgebra{F})\) generated by the \(\rho_i\) the \(\Gr(\Falgebra{F})/J\)-modules \(J/J\Gr(I)\) and \(M = \frakr/[\frakr,\frakr] + \frakr^{[p]}\) are isomorphic. We recall Lazard's elimination theorem for free (unrestricted) Lie algebras:

\begin{thm}[{\cite[Prop. 10]{bourbaki_lie_1989}}] Let \(R\) be a non-zero commutative ring, \(X\) a set and \(S\) a subset of \(X\). If \(\frakL(X)\) is the free Lie \(R\)-algebra on \(X\), then \(\frakL(X)\) is isomorphic as a Lie algebra to the direct sum \(\frakL(S) \oplus \frakh\), where \(\frakh\) is the ideal of \(\frakL(X)\) generated by \(X\smallsetminus S\) and is isomorphic to the free Lie algebra on the set \[\left\{[x,s_1,\ldots,s_k]\mid x \in X\smallsetminus S\,, s_1,\ldots,s_k \in S\right\}\,.\]
\end{thm}

Note that since \(\frakL\) is a free restricted Lie algebra it is also a free unrestricted \(\F_p\)-Lie algebra on the \(p\)-th powers \(\xi_i^{p^j}\,, \gamma^{p^j}\) of the restricted basis \(\xi_1,\ldots,\xi_n\,,\gamma\). The following is an immediate consequence of Lazard's elimination theorem:

\begin{thm}[{\cite[Sec. 1.2]{stohr_restricted_2001}}]\label{thm-lazard-elimination} If \(\frakh\) is the restricted ideal generated by the \(\xi_i\), then \(\frakL \simeq \langle \gamma^{p^i} \mid i \geq 0\rangle \oplus \frakh\) and \(\frakh\) is itself a free restricted Lie algebra with free generating set \[\{[\xi_i,\underbrace{\gamma,\cdots,\gamma}_{k \text{ times}}]\mid 1 \leq i \leq n\,, k\geq 0\}\,.\]
\end{thm}

From this, we are able to deduce a restricted variant of \cite[Thm 3.3]{labuteMildPropgroupsGalois2006}, keeping the above notation and using the fact that the universal restricted enveloping algebra of \(\frakL/\frakh\) is a polynomial algebra generated by \(\gamma\).

\begin{prop}\label{thm-labute-sf}  If \(\rho_1,\ldots,\rho_l\) are homogeneous elements of \(\frakh\) with respect to the canonical grading of \(\frakL\) which are linearly independent over \(\F_p[\gamma]\) modulo \([\frakh,\frakh] + \frakh^{[p]}\), then they are strongly free. Moreover, the restricted Lie algebra \(\frakh/(\rho_1,\ldots,\rho_l)\) is free, so that \[\frakL/(\rho_1,\ldots,\rho_l) \simeq \frakh/(\rho_1,\ldots,\rho_l) \rtimes \langle \gamma^{p^i} \mid i \geq 0\rangle\] is a free-by-(free of rank 1) restricted Lie algebra.
\end{prop}
\begin{proof} If \(\frakr\) and \(J\) denote the restricted ideal of \(\frakL\) and the two-sided ideal of \(\Gr(\Falgebra{F})\) respectively generated by the \(\rho_i\), we must show that \(M = \frakr/[\frakr,\frakr] + \frakr^{[p]}\) is a free \(\Gr(\Falgebra{F})/J\)-module with basis given by the image of the \(\rho_i\). Observe that \(\frakr\) is generated, as a restricted ideal of \(\frakh\), by the elements \[\rho_{i,k} = [\rho_i,\underbrace{\gamma,\ldots,\gamma}_{k\text{ times}}]\] with \(k \geq 0\) and \(1 \leq i \leq l\). Hence, it suffices to show that these elements are part of a free restricted generating set of \(\frakh\), for which it is enough to show that the  \(\rho_1,\ldots,\rho_l\) are \(\F_p[\gamma]\)-independent modulo \([\frakh, \frakh] + \frakh^{[p]}\). Since this holds by assumption, we are done. Moreover, as the \(\rho_{i,k}\) form part of a free restricted basis of \(\frakh\), the quotient \(\frakh/(\rho_1,\ldots,\rho_l)\) is a free restricted Lie algebra, so the latter part follows.
\end{proof}
 
\begin{cor}\label{cor:sfcommutators} The \(n\) elements \(\rho_i = [\xi_i,\gamma]\) are strongly free, as so is, modulo \([\frakh,\frakh] + \frakh^{[p]}\), any \(\F_p\)-linearly independent subset of the \(\F_p\)-subspace that they span. \hfill\qedhere
\end{cor}
\begin{proof} The first part follows from Theorem~\ref{thm-lazard-elimination} and Proposition~\ref{thm-labute-sf}. For the second part, note that an \(\F_p\)-linear independent subset of their \(\F_p\)-span will remain \(\F_p[\gamma]\)-linear independent modulo \([\frakh,\frakh] + \frakh^{[p]}\).
\end{proof}

\begin{cor}\label{cor-mildflag-sf} Every mild flag presentation of a free-by-\(\Z_p\) pro-\(p\) group \(G\) is strongly free, and \(\Gr(\Falgebra{G})\) is the skew polynomial ring generated by \(\overline{g-1}\) over the free algebra \(\Gr(\Falgebra{N})\). In particular, it is a domain.
\end{cor}
\begin{proof} The \(l\) relations in~(\ref{eq:flagpresentation}) of a mild flag presentation lie in the \(\F_p\)-span of \([\xi_i,\gamma]\) modulo \([\frakh,\frakh] + \frakh^{[p]}\) in the language of Corollary~\ref{cor:sfcommutators}. This can be seen by observing that: (1) the \(\rho_i\) have degree two and (2) writing them in the basis of \(D_2(F)/D_3(F)\) given by the commutators \([x_i,x_j]\), \([x_i,g]\) and possibly the squares \(x_i^2\) and \(g^2\) if \(p=2\) for \(1 \leq i < j \leq n\), the image of \(\Phi(\widetilde{N})\) in this quotient does not include expressions involving \([x_i,g]\) and \(g^2\). The last part follows from the universal property of restricted enveloping algebras and the identification \(\Gr(\Falgebra{G}) \simeq \Gr(\Falgebra{F})/J = \mathcal{U}_{\frakL/\frakr}\) of Proposition~\ref{prop:sf-graded}.
\end{proof}

\begin{cor} The completed group algebra \(S\) is a domain for every mild flag free-by-\(\Z_p\) pro-\(p\) group \(G\).
\end{cor}
\begin{proof} We have shown that the function \(w\colon S \to [0,1]\) defined in~(\ref{eq-defining-valuation}) is a multiplicative valuation. In particular, \(w(ab) = w(a)w(b) \neq 0\) for any \(a\), \(b\) non zero.
\end{proof}

Assume now that \(G\) has a mild flag presentation \[G \simeq \langle x_1,\ldots,x_n,g\mid [x_i,g] = h_i \in \Phi(\widetilde{N})\,, 1 \leq i \leq l\rangle\,.\] In particular, \(G\) is mild with respect to the dimension series and \(w\) is a valuation on \(S\). Let \(R_k\) be as in the introduction, that is, the induced filtration of \(R\) with respect to the filtration \(S_k\) of \(S\). Then, \(w\) restricts to a valuation function on \(R\) as well.

If we set \(x_{i,j} = [x_i,\underbrace{g,\ldots,g}_{j\text{ times}}]\) and \(X_{i,j} = x_{i,j} - 1\), then the group algebra \(R = \Falgebra{N}\) is isomorphic to the Magnus algebra \(\F_p\langle\!\langle X_{i,j}\ldots\mid j=0\text{ if }i\leq l\rangle\!\rangle\), the \(\F_p\)-algebra of noncommutative power series in at most countable variables \[X_{1,0},\ldots,X_{l,0},X_{l+1,0},X_{l+1,1},\ldots,X_{n,0},X_{n,1},\ldots\,,\] this follows from the finitely generated case of \cite[Prop. II.3.1.4]{lazardGroupesAnalytiquesPadiques1965} by taking the inverse limit of finite sets of generators. Moreover, as a subring of \(S = \Falgebra{G}\), it inherits the valuation \(w\). Since this ring has a simple description by power series, it would be desirable to have the property that \begin{equation}\label{eq:weight-function} w\left(\sum_{m \in \operatorname{Mon}(\{X_{i,j}\})} \lambda_m m\right) = \max\{w(m) \mid \lambda_m \neq 0\}\,,\end{equation} where \(\operatorname{Mon}(\{X_{i,j}\})\) denotes the set of all noncommutative monomials in the \(X_{i,j}\). However, this is not true for arbitrary valuations on Magnus algebras, and the valuations on \(R\) that do satisfy~(\ref{eq:weight-function}) are called \emph{weight functions} with respect to the basis \(\{X_{i,j}\}\) in \cite[Sec. 2.4]{ershovGroupsPositiveWeighted2013}. By \cite[Prop. 3.2]{ershovGroupsPositiveWeighted2013}, \(w\) is a weight function with respect to \(\{X_{i,j}\}\) if and only if the restricted \(\F_p\)-Lie algebra generated by \(\overline{X_{i,j}}\) is free in this basis. This is exactly what we've shown in Proposition~\ref{thm-labute-sf}, hence:

\begin{cor}\label{cor-val-wf} The induced valuation \(w\) on \(R\) is a weight function with respect to the basis \(\{X_{i,j}\}\) of \(R\) associated to a mild flag presentation.
\end{cor}

\section{The embeddings of the group algebra}\label{sec:group-embeddings}

We keep the notation of Section~\ref{sec:strongly-free-sequences}, and throughout this section we assume that \(G\)  has a mild flag presentation \[G \simeq \langle x_1,\ldots,x_n,g\mid [x_i,g] = h_i \in \Phi(\widetilde{N})\,, 1 \leq i \leq l\rangle\,.\] 
We recall that \(R = \Falgebra{N}\) and \(S = \Falgebra{G} \simeq R[\![s;\sigma,\delta]\!]\) by Proposition~\ref{prop-decomposition-power-series}. In this section we embed the ring $S$ into a division ring $\mathcal Q$.

We set \(x_{i,j} = [x_i,\underbrace{g,\ldots,g}_{j\text{ times}}]\). Then \(N\) is the free pro-\(p\) group on the set \(\{x_{i,j}\}\). Put \(\Lambda = {\mathbb{F}_p}\langle \{a_{i,j}\}\rangle\) be the free associative algebra over \(\mathbb{F}_p\) in the same number of variables \(a_{i,j}\) and consider the power series ring \(\Lambda[\![t]\!]\) in a variable \(t\). Through the universal property of the Magnus algebra, there exists an unique ring homomorphism \(\iota\colon R\to \Lambda[\![t]\!]\) sending \(x_{i,j}-1\) to \(a_{i,j}t^{w_{i,j}}\), where \(w_{i,j} = -\log_p w(x_{i,j}-1)\) is the logarithm of the valuation defined in~(\ref{eq-defining-valuation}) by the filtration \(R_k = R\cap I_G^k\) of \(R\).

\begin{lem}\label{lem-magnus-valuation-embedding} The map \(\iota\) is an embedding of topological rings that preserves the valuation (i.e. an isometric embedding).
\end{lem}
\begin{proof} We must show that \(r \in R\) belongs to \(R_k\) if and only if it's image \(\iota(r) \in \Lambda[\![t]\!]\) belongs to \(t^k\Lambda[\![t]\!]\). Since we've shown in Corollary~\ref{cor-val-wf} that the valuation \(w\) induced by the filtration \(R_k\) is a weight function, we see that the valuation of \(r\) expressed as a non-commutative power series equals the valuation of its monomial of least total weighted degree (with the weight of \(x_{i,j}\) being \(w_{i,j}\)). Hence, it suffices to check this condition only for monomials \(x_{i_1,j_1}\cdots x_{i_m,j_m}\), which is immediate. Moreover, any valuation preserving map must be injective, from which the lemma follows.
\end{proof}

\begin{prop}\label{prop-valuation-embedding} Both \(\sigma\) and \(\delta\) extend from \(R\) to \(\Lambda[\![t]\!]\) in such a way that \(\sigma\) is an automorphism, \(\delta\) is a \(\sigma\)-derivation, \(\sigma(t) = t\), \(\delta(t) = 0\), \[\sigma(a_i) = \frac{\iota(\sigma(y_i))}{t^{w_i}} \quad\text{and}\quad \delta(a_i) = \frac{\iota(\delta(y_i))}{t^{w_i}}\,.\] Moreover, \(\delta(\Lambda[\![t]\!]) \subseteq t\Lambda[\![t]\!]\) and therefore it is also topologically nilpotent on \(\Lambda[\![t]\!]\). In particular, the embedding \(R \hookrightarrow \Lambda[\![t]\!]\) induces an embedding \(S \hookrightarrow \Lambda[\![t]\!][\![s; \sigma, \delta]\!]\).
\end{prop}
\begin{proof} Once we set that \(\sigma(t) = t\) (and thus \(\delta(t) = 0\)), we can define \(\sigma\) arbitrarily on the generators \(a_i\) of \(\Lambda\), and for this map to extend the \(\sigma\) from \(R\) the formula must be as given. The same applies for the derivation \(\delta = \sigma - \operatorname{Id}\). From Corollary~\ref{cor-derivation-retracting} and Lemma~\ref{lem-magnus-valuation-embedding}, combined with with the fact that \(\iota\) and \(\delta\) thus defined commute, we conclude that \(\delta(\Lambda[\![t]\!]) \subseteq t\Lambda[\![t]\!]\) and thus the successive compositions \(\delta^i(\lambda)\) converge to zero for any element \(\lambda \in \Lambda[\![t]\!]\). Hence, the formulas in~(\ref{eq-power-series-mult}) are well defined for \(a_i\) and \(b_j\) in \(\Lambda[\![t]\!]\) and allows us to construct the ring \(\Lambda[\![t]\!][\![s;\sigma,\delta]\!]\), into which \(S\) embeds.
\end{proof}

As we stated in Section~\ref{subsec-intro-division-rings}, the free algebra \(\Lambda\) is a Sylvester domain, and therefore it has a universal division \(\Lambda\)-ring of fractions \(\mathcal{D}\).

\begin{prop} The maps \(\sigma\) and \(\delta\) extend from \(\Lambda[\![t]\!]\) to \(\mathcal{D}[\![t]\!]\). Moreover, \(\delta(\mathcal{D}[\![t]\!]) \subseteq t\mathcal{D}[\![t]\!]\), and thus we get an embedding \(\Lambda[\![t]\!][\![s;\sigma,\delta]\!] \hookrightarrow \mathcal{D}[\![t]\!][\![s; \sigma, \delta]\!]\).
\end{prop}
\begin{proof} Since \(\delta = \sigma - 1\), it suffices to show that \(\sigma\) extends from \(\Lambda[\![t]\!]\) to \(\mathcal{D}[\![t]\!]\). Let us consider \(\sigma\) as a map \(\sigma\colon \Lambda\to \mathcal{D}[\![t]\!]\) through the usual embedding \(\Lambda[\![t]\!] \to \mathcal{D}[\![t]\!]\). If we show that the image of every full matrix \(A \in \Mat_{n\times n}(\Lambda)\) through \(\sigma\) is invertible, by the Proposition~\ref{prop-universal-prop-sylvester} we get an extension of \(\sigma\) to a map \(\sigma\colon \mathcal{D} \to \mathcal{D}[\![t]\!]\). Moreover, this suffices to define \(\sigma\) on \(\mathcal{D}[\![t]\!]\), since we must have \(\sigma(t) = t\).

By Proposition~\ref{prop-valuation-embedding}, for each generator \(a_{i,j} \in \Lambda\) there exists \(b_{i,j} \in \Lambda[\![t]\!]\) such that \(\sigma(a_{i,j}) = a_{i,j} + b_{i,j}t\). This implies that for every matrix \(A \in \operatorname{Mat}_{n\times n}(\Lambda)\) there exists a matrix \(B \in \operatorname{Mat}_{n\times n}(\mathcal{D}[\![t]\!])\) such that \(\sigma(A) = A + Bt\). However, a matrix over \(\mathcal{D}[\![t]\!]\) is invertible if and only if it's projection over \(\mathcal{D}\) is invertible, which is the case when \(A\) is a full matrix. Therefore, we can extend \(\sigma\) (and hence \(\delta\)) from \(\Lambda[\![t]\!]\) to \(\mathcal{D}[\![t]\!]\).

To show that \(\delta(\mathcal{D}[\![t]\!]) \subseteq t\mathcal{D}[\![t]\!]\) and thus it is still topologically nilpotent, we will use the fact that \(\mathcal{D}\) is the division closure of \(\Lambda\) in \(\mathcal{D}\) and hence can be recursively obtained by adding inverses to \(\Lambda\) and closing under the ring operations. This implies that it suffices to show that \(\delta(a^{-1}) \in t\mathcal{D}[\![t]\!]\) whenever \(\delta(a) \in t\mathcal{D}[\![t]\!]\) for \(a \in \mathcal{D}\) non-zero. In this case, write \(\delta(a) = tb\) for \(b \in \mathcal{D}[\![t]\!]\), and observe that we must have \[\delta(a^{-1}) = -\sigma(a)^{-1}\delta(a)a^{-1} = -\sigma(a)^{-1}ba^{-1}t\,.\qedhere\]\end{proof}

We already have the following chain of embeddings for the group algebra \(S\): \[
    S \simeq R[\![s;\sigma,\delta]\!] \hookrightarrow \Lambda[\![t]\!][\![s;\sigma,\delta]\!] \hookrightarrow \mathcal{D}[\![t]\!][\![s;\sigma,\delta]\!]\,.
\] In order to construct our candidate for the universal division \(S\)-ring of fractions, we will complete this chain by showing the existence of a universal \(\mathcal{D}[\![t]\!][\![s;\sigma,\delta]\!]\)-ring of fractions. To shorten the notation, we denote the ring \(\mathcal{D}[\![t]\!][\![s;\sigma,\delta]\!]\) simply by \(\bbd\).

Let \(J\) be the two-sided ideal of \(\bbd\) generated by \(t\) and \(s\). By the construction of the embedding, the restriction of the \(J\)-adic topology to \(S\) as a subspace of \(\bbd\) coincides with the topology given by the powers of the augmentation ideal on \(S\). Since \(\delta(\mathcal{D}[\![t]\!]) \subseteq t\mathcal{D}[\![t]\!]\), the skew relation defining \(s\) gives us \[\mathcal{D}[\![t]\!]s \subseteq s\mathcal{D}[\![t]\!] + t\mathcal{D}[\![t]\!]\,,\] and therefore \(J\) is contained in the right ideal generated by \(t\) and \(s\). Inductively we have the identities:
\begin{align}
    J &= s\bbd + t\bbd\,,\notag\\
    J^2 &= s^2\bbd + st\bbd + t^2\bbd\,,\notag\\
    \vdots&= \vdots\notag\\
    J^n &= \sum_{i=1}^n s^it^{n-i}\bbd\label{eq-jadic-ideals}\\
    \vdots &= \vdots\notag
\end{align}
and therefore the intersection \(\bigcap_{i=1}^\infty J^i\) is trivial: a power series on \(s\) belongs to \(J^n\) only if the coefficient of \(s^i\) is a multiple of \(t^{n-i}\). In other words, the \(J\)-adic topology on \(\bbd\) is Hausdorff. It also follows from~(\ref{eq-jadic-ideals}) that the \(J\)-adic topology on \(\bbd\) is complete.

\begin{prop}\label{prop-graded-polynomial} The maps \(\sigma\) and \(\delta\) induce maps \[\overline{\sigma}\colon J^n/J^{n+1} \to J^n/J^{n+1}\quad\text{and}\quad \overline{\delta}\colon J^n/J^{n+1} \to J^{n+1}/J^{n+2}\,\] such that \(\overline{\delta}\) is a right \(\overline{\sigma}\)-derivation on \[\Gr \bbd \simeq \bbd/J \oplus J/J^2 \oplus J^2/J^2 \oplus \cdots\,.\] 
If \(\overline{t} = t+J^2\) and \(\overline{s} = s + J^2\), then the graded ring \(\Gr \bbd\) is isomorphic to the skew polynomial ring \(\mathcal{D}[\overline{t}][\overline{s}; \overline{\sigma}, \overline{\delta}]\).
\end{prop}
\begin{proof} From~(\ref{eq-jadic-ideals}) and the fact that \(\delta(\mathcal{D}) \subseteq t\mathcal{D}[\![t]\!]\) we see that \(J^n\) is \(\sigma\)-invariant and that \(\delta(J^n) \subseteq J^{n+1}\). Hence, the maps \(\overline{\sigma}\) and \(\overline{\delta}\) are well defined, and direct calculations on representatives of homogeneous elements shows that indeed \(\overline{\delta}\) is a right \(\overline{\sigma}\)-derivation on the subring \(\mathcal{D}[t]\) of \(\Gr(\bbd)\) generated by \(\mathcal{D}\) and \(\overline{t}\).

Hence, it remains to check that the powers \(\overline{s}^i = s^i + J^{i+1}\) for \(i \geq 0\) form a basis of \(\Gr(\bbd)\) over \(\mathcal{D}[t]\), since the skew commutation relations for \(\overline{s}\) are inherited from those of \(\bbd\). From~(\ref{eq-jadic-ideals}), we can deduce that \[J^n/J^{n+1} = \bigoplus_{i=1}^n s^it^{n-i}\mathcal{D}\] as \(\mathcal{D}\)-modules. Therefore, if \(\sum r_i\overline{s}^i = 0\) is any relation among the powers of \(\overline{s}\), isolating the homogeneous components in \(\Gr\bbd\) gives us that \(r_i = 0\).
\end{proof}

\begin{cor}\label{cor-existence-Ore} The ring \(\bbd\) is a Noetherian domain. In particular, it has a classical Ore ring of fractions \(\mathcal{Q}\).
\end{cor}
\begin{proof} To show that it is a domain, note that the product of homogeneous elements in \(\bbd\) is non-zero in the graded ring, and see, for instance, \cite[Cor. D.IV.5]{nastasescuGradedRingTheory1982} for lifting the Noetherian property from \(\Gr \bbd\) to \(\bbd\). Then, the existence of \(\mathcal{Q}\) follows from \cite[Thm. 2.1.15]{robson_noncommutative_2001}.
\end{proof}

\section{Homological finiteness properties}\label{sec:hom-finite-prop}

Throughout this section, we fix a mild flag free-by-cyclic pro-\(p\) group \(G\), with free kernel \(N = F(x_1,x_2,\ldots)\). We let \(R = \Falgebra{N}\) and \(S = \Falgebra{G}\). By Proposition~\ref{prop-decomposition-power-series}, we can identify \(S\) with a skew power series algebra \(R[\![s;\sigma,\delta]\!]\) over \(R\), which can be embedded into the division ring \(\mathcal{Q}\) of Corollary~\ref{cor-existence-Ore}, the classical ring of fractions of the ring \(\bbd = \mathcal{D}[\![t]\!][\![s;\sigma,\delta]\!]\) where \(\mathcal{D}\) is the universal division ring of fractions of the free \(\F_p\)-algebra \(\Lambda = \F_p\langle a_1,a_2,\ldots\rangle\). 

The objective now is to show that $S$ is a Sylvester domain and that the embedding \(S \hookrightarrow\mathcal  Q \) is universal.
 First, however, we need to prove some homological vanishing results about \(S\)-submodules of \(\mathcal{Q}\) and \(\Gr S\)-submodules of \(\Gr \bbd \simeq \mathcal{D}[\overline{t}][\overline{s};\overline{\sigma},\overline{\delta}]\).

\subsection{Filtered modules}

If \(M\) is a finitely generated right (respectively, left) \(S\)-submodule of \(\mathcal{Q}\), then there exists an element \(c \in \bbd\) such that \(cM\) (respectively, \(Mc\)) is a finitely generated right (respectively, left) \(S\)-submodule of \(\bbd\). This gives us a bijective correspondence between   isomorphism classes of \(S\)-submodules of \(\mathcal{Q}\) and \(S\)-submodules of \(\bbd\), so we can always assume that a finitely generated \(S\)-submodule of \(\mathcal{Q}\) is contained in \(\bbd\).

If we consider \(S\) as an filtered ring with the filtration \(S_{k} = I_G^k\) and \(S_{-k} = S\) for \(k \geq 0\), the \(S\)-module \(\bbd\) becomes a filtered \(S\)-module (and even a filtered ring) with the positive filtration given by the powers of \(J\), it's two sided ideal generated by \(t\) and \(s\). Hence, for every \(S\)-submodule \(M \subseteq \bbd\), we can endow \(M\) with an induced filtered structure: \[M_{k} = M \cap J^k\,,\quad M_{-k} = M\,,\quad \forall k \geq 0\,.\] 

Since we must state general results about filtered rings that will apply both to \(S\) and \(\bbd\), we will denote a general filtered ring by \(\Omega\) with a filtration by \(\Omega\)-submodules \(\Omega_k\), \(k \in \Z\). Following the terminology in \cite[Sec. 4.1]{jaikin-zapirainExplicitConstructionUniversal2020}, a filtration \(\{M_k\}\) on a \(\Omega\)-module \(M\) is:
\begin{enumerate}[itemindent=3em]
    \item[(\emph{Separated})] The intersection \(\bigcap_{k \in \Z}  M_k\) is trivial;
    \item[(\emph{Complete})] \(M\) is complete with respect to the metric \[d_M(a,b) = \inf\{p^{-k}\mid a-b \in M_k\}\,;\]
    \item[(\emph{Bounded})]  Each quotient \(M_k/M_{k+1}\) is an \(\Omega/\Omega_{1}\)-module of finite length.
\end{enumerate}

In particular, the filtrations on the \(S\)-modules \(M = S\) and \(M = \bbd\) are separated and complete, and moreover the filtration on \(M = S\) is bounded. Observe that the induced filtration on an \(S\)-submodule of a separated module is again separated.

 If \(\Omega\) is bounded and complete as a left \(\Omega\)-module and \(M\) is a finitely generated and separated \(\Omega\)-module, then \(M\) is also complete (\cite[Prop. 4.2]{jaikin-zapirainExplicitConstructionUniversal2020}).

We recall that a homomorphism \(\varphi\colon M \to N\) of filtered \(\Omega\)-modules is called \emph{strict} if \(\varphi(M_k) = \varphi(M) \cap N_k\) for every \(k \in \Z\), that is, if the filtrations on \(\operatorname{Im}(\varphi)\) induced by \(\varphi\) and by \(N\) coincide.  A sequence of filtered \(\Omega\)-modules \(M \to N \to Q\) is strict exact if the sequence is exact and each homomorphism is strict. Observe that an strict exact sequence as above induces an exact sequence \(\Gr M \to \Gr N \to \Gr Q \) of \(\Gr \Omega\)-modules.

To obtain free resolutions in the category of finitely generated filtered \(\Omega\)-modules and strict homomorphisms, we must make a slight change to natural filtration of the free \(\Omega\)-modules \(\bigoplus_{I} \Omega\). Let \(V\) be a multiset of integers, that is, a set of integers that allows for repetition of elements. We say that \(V\) is bounded if for every \(k \in \Z\) the multiset \(\{\!\{v \in V \mid v \leq k\}\!\}\) is finite. In particular, every element of a bounded multiset \(V\) is belongs \(\Z_{\geq k} = \{z \in \Z\mid z \geq  k\}\) for some \(k = k(V) \in \Z\). If necessary, one can realize a multiset \(V\) as a proper set \(\bm{V}\) by taking the disjoint union of singletons of its repeating elements. If \(V\) is bounded, \(\bm{V}\) can be made into a profinte set by adjoining a point at infinity.

For any separated, bounded and complete \(\Omega\)-module \(M\), let \(M(V)\) be the \(\Omega\)-module \(\bigoplus_{v \in V} M\) with the following filtration: \[M(V)_k = \bigoplus_{v \in V} M_{k-v}\,.\]
Then, for \(M = \Omega\), we have that \(\Omega(V)\) is a free \(\Omega\)-module satisfying the following universal property: identifying \(V\) with the set of elements \((0,\ldots,1,\ldots) \in \Omega(V)\), any function \(f\colon V \to N\) from \(V\) to a filtered \(\Omega\)-module \(N\) such that \(f(v) \in N_v\) extends uniquely to a strict homomorphism \(\varphi\colon \Omega(V) \to N\). Moreover, observe that \(\Omega(V)\) is bounded as a filtered \(\Omega\)-module if and only if \(V\) is bounded as a multiset of integers.

Let \(\widehat{M(V)} = \varprojlim M(V)/M(V)_k\) denote the completion of \(M(V)\) for each multiset \(V\). If \(N\) is a separated, complete and bounded \(\Omega\)-module and \(M(V) \to N\) is a strict homomorphism, then there is a unique strict extension \(\widehat{M(V)} \to N\). We recall the following result of \cite{jaikin-zapirainExplicitConstructionUniversal2020}:

\begin{prop}[{\cite[Prop.~4.5 and Cor.~4.6]{jaikin-zapirainExplicitConstructionUniversal2020}}]\label{prop-continuous-resolution} Let \(M\) be a complete and bounded filtered \(\Omega\)-module for a complete and bounded filtered ring \(\Omega\). Then, there are bounded multisets of integers \(V_i\) and a strict exact sequence of filtered \(\Omega\)-modules \[\cdots \to \widehat{\Omega(V_i)} \to \cdots \to \widehat{\Omega(V_1)} \to \widehat{\Omega(V_0)} \to M \to 0\,.\]
\end{prop}

While the modules \(\widehat{\Omega(V)}\) are no longer free over \(\Omega\) even if \(V\) is bounded, if \(\Omega\) is pseudocompact in the sense of \cite{brumerPseudocompactAlgebrasProfinite1966} then they are still flat with respect to complete tensor product \[M\cotimes_{S}N = \varprojlim_{k \in \Z} M/M_k \otimes_\Omega N/N_k\] of filtered \(\Omega\)-modules \(M\) and \(N\) (\cite[Cor. 1.3 and Lem. 2.1]{brumerPseudocompactAlgebrasProfinite1966}), for in this case \(\widehat{\Omega(V)}\) can be identified as a topological \(\Omega\)-module with the free pseudocompact \(\Omega\)-module \(\Omega[\![\bm{V}]\!]\) on set \(\bm{V}\) (\cite[p. 444]{brumerPseudocompactAlgebrasProfinite1966}). This applies in particular to \(\Omega = S\) and to \(\Omega = \bbd\).

Let us denote by \(\CTor_i^\Omega(-,-)\) the \(i\)-derived bifunctor obtained from the complete tensor product, hereby referred to as the \emph{\(i\)-th continuous Tor group}, in order to differentiate it from the usual \(i\)-derived functor \(\Tor_i^\Omega(-,-)\) obtained from the usual tensor product \(-\otimes_\Omega -\). Therefore, we already know that \(\CTor_i^\Omega(\widehat{\Omega(V)},M) = \CTor_i^\Omega(N,\widehat{\Omega(V)}) = 0\) for any \(i > 0\) and \(\Omega\)-modules \(M\) and \(N\). The following proposition gives us a way to compare both functors:

\begin{prop}[{\cite[Prop. 4.9]{jaikin-zapirainExplicitConstructionUniversal2020}}]\label{prop-fp-acyclic} Let \(M\) be a separated, filtered and complete \(\Omega\)-module for some filtered ring \(\Omega\) and \(V\) a bounded multiset of integers. Then, the following hold:
\begin{enumerate}[label=(\arabic*)]
    \item If \(M\) is finitely presented, then the natural maps \(M\otimes_\Omega \widehat{\Omega(V)} \to M\cotimes_\Omega \widehat{\Omega(V)}\) and \(M\cotimes_\Omega \widehat{\Omega(V)} \to \widehat{M(V)}\) are strict isomorphisms.
    \item Let \(k \geq 2\). If \(M\) is of type \(FP_k\), then \(\Tor_{k-1}^\Omega(M, \widehat{\Omega(V)}) = 0\).
\end{enumerate}
\end{prop}

This result implies the following proposition, well known to the specialist, but for which we have found no reference in the literature. Observe that the isomorphisms of Proposition~\ref{prop-fp-acyclic} are all functorial in \(M\). First we show a partial form of \cite[Lem. XX.6.3]{langauth.Algebra2002}:

\begin{lem}\label{lem-partial-exactness} Let \(\Omega\) be any ring, \(M\) a \(\Omega\)-module and suppose that the \(\Omega\)-modules on a exact sequence \[\cdots\to X_j \to \cdots \to X_1 \to X_0 \to 0\] satisfy \(\Tor_i^\Omega(M, X_j) = 0\) for every \(j\) and every \(1 \leq i \leq k-2\) with \(k \geq 2\). Then, \[M\otimes_\Omega X_{k} \to M\otimes_\Omega X_{k-1} \to \cdots \to M\otimes_\Omega X_1 \to M\otimes_\Omega X_0 \to 0\] is exact.
\end{lem}
\begin{proof} We already have exactness up to \begin{equation}\label{eq-M-tensor-1st-step} M\otimes_\Omega X_2 \to M\otimes_\Omega X_1 \to M \otimes_\Omega X_0 \to 0\,.\end{equation} Let \(Z_i\) be the image of \(X_i\) in \(X_{i-1}\), that is, the kernel of \(X_{i-1}\to X_{i-2}\). We have the exact sequence:
\[0 = \Tor_1^\Omega(M,X_0) \to M\otimes_\Omega Z_2 \to M\otimes_\Omega X_1 \to M\otimes_\Omega X_0 \to 0\,,\] which shows that \(M\otimes_\Omega Z_2\) is the kernel of \(M \otimes_\Omega X_1 \to M\otimes_\Omega X_0\). Since it is also fits in the exact sequence \[M\otimes_\Omega X_3 \to M\otimes_\Omega X_2 \to M\otimes_\Omega Z_2 \to 0\] by right exactness, we are able to extend the sequence~(\ref{eq-M-tensor-1st-step}) by one degree more.

Observe that \(Z_2\) also satisfies \(\Tor_i^\Omega(M,Z_2) = 0\) now for every \(1 \leq i \leq k-3\), since we have the exact sequence \[0 = \Tor_{i+1}^\Omega(M, X_0) \to \Tor_i^\Omega(M,Z_2) \to \Tor_i^\Omega(M,X_1) = 0\,.\] Hence, by taking \(X_0' = Z_2\) and \(X_i' = X_{i+1}\) for \(i > 0\), we are reduced to showing that the exact sequence \[\cdots\to  X_i' \to \cdots \to X_2' \to X_1' \to X_0' \to 0\] induces the exact sequence \[M \otimes_\Omega X_{k-1}' \to M \otimes_\Omega X_{k-2}' \to \cdots \to M \otimes_\Omega X_1' \to M \otimes_\Omega X_0' \to 0\,,\] which now follows from induction on \(k\) since we've already proven the base step \(k = 2\).
\end{proof}

\begin{prop}\label{pro-acyclic-tor-resolution} If \(\Omega\) is a bounded and complete filtered ring, \(M\) is a filtered \(\Omega\)-module of type \(FP_k\) for some \(k \geq 2\), \(N\) is a separated, bounded and complete \(\Omega\)-module and \begin{equation}\label{eq-continuous-resolution}
    \cdots \to \widehat{\Omega(V_i)} \to \cdots \to \widehat{\Omega(V_1)} \to \widehat{\Omega(V_0)} \to N \to 0\tag*{\(\star\)}
\end{equation} is a strict exact sequence with each \(V_i\) bounded (such as in Proposition~\ref{prop-continuous-resolution}), then \(Tor_i^\Omega(M,N)\) can be computed as the \(i\)-th homology of the complex \[\cdots \to \widehat{M(V_i)} \to \cdots \to \widehat{M(V_1)} \to \widehat{M(V_0)} \to 0\] for \(i \leq k-1\), where this complex is obtained by applying \(M \otimes_\Omega -\) to the sequence~(\ref{eq-continuous-resolution}). In particular, one has \(\CTor_i^\Omega(M,N) \simeq \Tor_i^\Omega(M,N)\) for \(i \leq k-1\).
\end{prop}
\begin{proof} This can be thought as a ``partial'' analogue of how one may compute derived functors through acyclic resolutions (cf. \cite[Thm. XX.6.2]{langauth.Algebra2002}). We consider another resolution of \(N\) by projective \(\Omega\)-modules:
\begin{equation}\label{eq-projective-N}
\cdots \to P_i \to \cdots \to P_1 \to P_0 \to N \to 0\,,\tag*{\(\star\star\)}
\end{equation} and we will show that there exists a morphism of complexes \(P_i \to \widehat{\Omega(V_i)}\) extending the identity on \(N\) and satisfying \[H_i(M\otimes_\Omega \star) \simeq H_i(M\otimes_\Omega \star\star)\quad\text{ for all }0 \leq i \leq k -1\,.\]

Since \(\Tor_i^\Omega(M,N)\) doesn't depend on the choice of projective resolution~(\ref{eq-projective-N}), we can inductively choose the \(P_i\) such that we have a commuting diagram with surjective columns:
\begin{center}
    \begin{tikzcd}[row sep=large]
\cdots \arrow[r] & P_i \arrow[r] \arrow[d, two heads] & \cdots \arrow[r] & P_1 \arrow[r] \arrow[d, two heads] & P_0 \arrow[r] \arrow[d, two heads] & N \arrow[d, "\operatorname{id}"] \\
\cdots \arrow[r] & \widehat{\Omega(V_i)} \arrow[r]         & \cdots \arrow[r] & \widehat{\Omega(V_1)} \arrow[r]         & \widehat{\Omega(V_0)} \arrow[r]         & N                               
\end{tikzcd}
\end{center}

For each \(i \geq 0\), we let \(X_i\) be the kernel of the map \(P_i \to \widehat{\Omega(V_i)}\). Then, Proposition~\ref{prop-fp-acyclic}(2) and the long exact sequence \[\cdots \to \Tor_{i+1}^\Omega(M,\widehat{\Omega(V_j)}) \to \Tor_i^\Omega(M,X_j) \to \Tor_i^\Omega(M, P_j) \to \cdots\] shows that \(\Tor_i^\Omega(M,X_j) = 0\) for every \(1 \leq i \leq k-2\). Hence, by Lemma~\ref{lem-partial-exactness}, we know that the diagram \begin{center}
    \begin{tikzcd}[row sep=large]
M \otimes_\Omega X_k \arrow[r] \arrow[d, hook]      & M \otimes_\Omega X_{k-1} \arrow[r] \arrow[d, hook]      & \cdots \arrow[r] & M \otimes_\Omega X_0 \arrow[r] \arrow[d, hook]      & 0 \\
M \otimes_\Omega P_k \arrow[r] \arrow[d, two heads] & M \otimes_\Omega P_{k-1} \arrow[r] \arrow[d, two heads] & \cdots \arrow[r] & M \otimes_\Omega P_0 \arrow[r] \arrow[d, two heads] & 0 \\
M \otimes_\Omega \widehat{\Omega(V_k)} \arrow[r]         & M\otimes_\Omega \widehat{\Omega(V_{k-1})} \arrow[r]          & \cdots \arrow[r] & M \otimes_\Omega \widehat{\Omega(V_0)} \arrow[r]         & 0
\end{tikzcd}
\end{center} has an exact top row and each column is a short exact sequence because of the vanishing of \(\Tor_1(M,\widehat{\Omega(V_i)})\). Then, applying the Snake Lemma to each column one finds the exact sequences \[0 = H_i(M\otimes_\Omega X_*) \to H_i(M\otimes_\Omega \star) \to H_i(M\otimes_\Omega \star\star) \to H_{i-1}(M \otimes_\Omega X_*) = 0\,,\] yielding the desired isomorphism for every \(i \leq k-1\). The last part follows from the isomorphisms in part (1) of Proposition~\ref{prop-fp-acyclic}.
\end{proof}

We also state the following result of \cite{jaikin-zapirainExplicitConstructionUniversal2020}, which we will use to lift vanishing results about the \(\Tor\) groups over \(\Gr S\) to those over \(S\):

\begin{thm}[{\cite[Thm. 4.10]{jaikin-zapirainExplicitConstructionUniversal2020}}]\label{thm:andrei410} Let \(\Omega\) be a filtered ring and \(k \geq 1\). Assume that \(\Omega\) is complete and bounded as a left and right \(\Omega\)-module. Let \(M\) and \(N\) be complete filtered right and left \(\Omega\)-modules respectively and assume that:
\begin{enumerate}[label=\upshape (\arabic*)]
    \item \(M\) is of type \(FP_{k+1}\),
    \item \(N\) is bounded, and
    \item \(\Tor_k^{\Gr \Omega}(\Gr M, \Gr N) = 0\).
\end{enumerate}
Then \(\Tor_k^\Omega(M, N) = 0\).
\end{thm}

\subsection{Graded modules over \texorpdfstring{\(\bm{\Gr S}\)}{Gr S}}\label{subsec:gradedring}

We have shown in Corollary~\ref{cor-mildflag-sf} that the ring \(\Gr S\) is a skew  polynomial ring over the graded ring \(\Gr R\), the latter which can be identified with \(\Lambda\). Hence, we shall see \(\Gr S\) as isomorphic to \(\Lambda[\widetilde{s}; \widetilde{\sigma},\widetilde{\delta}]\) where \(\widetilde{s}= \overline{s} = \overline{g-1}\). Moreover, since this isomorphism is compatible with the isomorphism \(\Gr \bbd \simeq \mathcal{D}[\overline{t}][\overline{s};\overline{\sigma},\overline{\delta}]\) under the injective map \(\Gr S \to \Gr \bbd\), the maps \(\widetilde{\sigma}\) and \(\widetilde{\delta}\) are restrictions of \(\overline{\sigma}\) and \(\overline{\delta}\) to \(\Gr S\) respectively. Hence, there are unique extensions \(\widetilde{\sigma},\widetilde{\delta}\) from \(\Lambda\) to \(\mathcal{D}\) that that make the following diagram commute:
\begin{center}
    \begin{tikzcd}[row sep=large]
\mathcal{D} \arrow[d, "{\widetilde{\sigma},\widetilde{\delta}}"', dotted] \arrow[r, "\iota"] & {\mathcal{D}[t]} \arrow[d, "{\overline{\sigma},\overline{\delta}}"] \\
\mathcal{D} \arrow[r, "\iota"']                                                              & {\mathcal{D}[t]\,.}                                                   
\end{tikzcd}
\end{center}

The ring \(\Gr S = \Lambda[\widetilde{s};\widetilde{\sigma}, \widetilde{\delta}]\) can then be embedded in the skew polynomial ring \(\mathcal{D}[\widetilde{s};\widetilde{\sigma}, \widetilde{\delta}]\) through the inclusion \(\Lambda \to \mathcal{D}\), and such embedding is an epic ring homomorphism. We note that:

\begin{lem}[{cf. \cite[Lem. 2.1]{HL22}}]\label{lem:exteding-scalars} The left (resp. right) \(\Gr S\)-modules \(\Gr S \otimes_\Lambda \mathcal{D}\) (resp. \(\mathcal{D} \otimes_\Lambda \Gr S\)) and \(\mathcal{D}[\widetilde{s};\widetilde{\sigma}, \widetilde{\delta}]\) are isomorphic.
\end{lem}
\begin{proof} Consider the additive map \(\mu\colon \Gr S \otimes_\Lambda \mathcal{D} \to \mathcal{D}[\widetilde{s};\widetilde{\sigma}, \widetilde{\delta}]\) given by the unique extension of the \(\Lambda\)-balanced map \(m\colon \Gr S \times \mathcal{D} \to \mathcal{D}[\widetilde{s};\widetilde{\sigma}, \widetilde{\delta}]\) defined by \[m(\widetilde{s}^n, r) = \widetilde{s}^nr\,.\]
Indeed, it is bi--additive, satisfies \(m(\widetilde{s}^n, r'r) = m(\widetilde{s}^nr', r)\) for any \(r' \in \Lambda\) and therefore extends to an abelian group homomorphism \(\mu\). It is a direct verification that \(\mu\) is left \(\Gr S\)-linear, and to see that it is an isomorphism, it suffices to observe that it is also right \(\mathcal{D}\)-linear and maps the basis \(\{\widetilde{s}^n \otimes 1\}\) of \(\Gr S\otimes_\Lambda \mathcal{D}\) to the basis \(\{\widetilde{s}^n\}\) of \(\mathcal{D}[\widetilde{s};\widetilde{\sigma}, \widetilde{\delta}]\). The proof for the right module case is similar.
\end{proof}

From this and the fact that \(\mathcal{D}[\widetilde{s};\widetilde{\sigma}, \widetilde{\delta}]\) is an Ore domain (\cite[Thm. 1.2.9 and 2.1.15]{robson_noncommutative_2001}), the same argument in \cite{HL22} gives us:

\begin{prop}[{cf. \cite[Lem. 2.7 and 2.8]{HL22}}]\label{prop-1st-vanishing} Let \(M\) (resp. \(M'\)) be a left (resp. right) \(\Gr S\)-module and \(N\) (resp. \(N'\)) be a left (resp. right) \(\Gr S\)-submodule of \(\operatorname{Ore}(\mathcal{D}[\widetilde{s};\widetilde{\sigma},\widetilde{\delta}])^k\) for some \(k \geq 1\). Then the following holds:
\begin{enumerate}[label=(\arabic*)]
    \item \(\Ext^3_{\Gr S}(M,M') = 0\);
    \item \(\mathcal{D}[\widetilde{s};\widetilde{\sigma}, \widetilde{\delta}]\) has projective dimension at most \(1\) as a left and right \(\Gr S\)-module;
    \item Every left or right \(\Gr S\)-submodule of \(\mathcal{D}[\widetilde{s};\widetilde{\sigma}, \widetilde{\delta}]^k\) for \(k\geq 1\) has projective dimension at most \(1\);
    \item Every finitely generated left or right \(\Gr S\)-submodule of \(\operatorname{Ore}(\mathcal{D}[\widetilde{s};\widetilde{\sigma}, \widetilde{\delta}])^k\) for \(k\geq 1\) has projective dimension at most \(1\).
    \item \(\Tor_2^{\Gr S}(\mathcal{D}[\widetilde{s};\widetilde{\sigma}, \widetilde{\delta}], M) = 0\);
    \item \(\Tor_1^{\Gr S}(\mathcal{D}[\widetilde{s};\widetilde{\sigma}, \widetilde{\delta}], N'') = 0\) for every left \(\mathcal{D}[\widetilde{s};\widetilde{\sigma}, \widetilde{\delta}]\)-module \(N''\);
    \item \(\Tor_1^{\Gr S}(\mathcal{D}[\widetilde{s};\widetilde{\sigma}, \widetilde{\delta}], N) = 0\);
    \item \(\Tor_1^{\Gr S}(\operatorname{Ore}(\mathcal{D}[\widetilde{s};\widetilde{\sigma}, \widetilde{\delta}]), N) = 0\);
    \item \(\Tor_1^{\Gr S}(N',\operatorname{Ore}(\mathcal{D}[\widetilde{s};\widetilde{\sigma}, \widetilde{\delta}])) = 0\);
    \item \(\Tor_1^{\Gr S}(\operatorname{Ore}(\mathcal{D}[\widetilde{s};\widetilde{\sigma}, \widetilde{\delta}]), \operatorname{Ore}(\mathcal{D}[\widetilde{s};\widetilde{\sigma}, \widetilde{\delta}])) = 0\).
\end{enumerate}
\end{prop}
\begin{proof} (1) Since the skew polynomial ring \(\Gr S\) over the free algebra \(\Gr R\) has global dimension at most \(2\) by \cite[Thm. 7.5.3]{robson_noncommutative_2001} and \cite[Cor. 2.5.2]{cohnFreeIdealRings2006}, every left or right module \(M\) has a projective resolution of length at most \(2\). Hence, the claim follows.

(2) Let \(0 \to P_1 \to P_0 \to \mathcal{D} \to 0\) be a projective resolution of the left \(\Lambda\)-module \(\mathcal{D}\) over \(\Lambda\), which exists because \(\Lambda\) is a free algebra and hence has global dimension at most \(1\). Applying the functor \(\Gr S\otimes_\Lambda -\) to this sequence we obtain the sequence \[0 \to \Gr S\otimes_\Lambda P_1 \to \Gr S\otimes_\Lambda P_0 \to \Gr S\otimes_\Lambda \mathcal{D}\to 0\] which is again exact by the freeness of \(\Gr S\) over \(\Lambda\) (by both sides!). Since each \(\Gr S\otimes_\Lambda P_i\) is again projective, this is a projective resolution of \(\Gr S\otimes_\Lambda \mathcal{D}\) over \(\Gr S\). We know that \(\Gr S\otimes_\Lambda \mathcal{D}\) is isomorphic to \(\mathcal{D}[\widetilde{s};\widetilde{\sigma}, \widetilde{\delta}]\) as a left \(\Gr S\)-module by Lemma~\ref{lem:exteding-scalars}, so the claim for the left module \(\mathcal{D}[\widetilde{s};\widetilde{\sigma}, \widetilde{\delta}]\) follows. The proof for \(\mathcal{D}[\widetilde{s};\widetilde{\sigma}, \widetilde{\delta}]\) as a right \(\Gr S\)-module is analogous, exchanging the functor \(\Gr S\otimes_\Lambda -\) for \(-\otimes_\Lambda \Gr S\).

(3) Let \(A\) be an \(\Gr S\)-submodule of \(\mathcal{D}[\widetilde{s};\widetilde{\sigma}, \widetilde{\delta}]\) and \(B\) be another \(\Gr S\)-module. The short exact sequence \[0 \to A \to \mathcal{D}[\widetilde{s};\widetilde{\sigma}, \widetilde{\delta}] \to C \to 0\] induces a long exact sequence \[\cdots \to \Ext^2_{\Gr S}(\mathcal{D}[\widetilde{s};\widetilde{\sigma}, \widetilde{\delta}], B) \to \Ext^2_{\Gr S}(A,B) \to \Ext^3_{\Gr S}(C,B) \to \cdots\,,\] where the first term vanishes by (2) and the third term vanishes by (1). Hence, \(\Ext^2_{\Gr S}(A,B) = 0\) for all \(\Gr S\)-modules \(B\), and the claim follows from \cite[Prop. 8.6]{rotman_introduction_2008}.

(4) Every finitely generated \(\Gr S\)-submodule of \(\operatorname{Ore}(\mathcal{D}[\widetilde{s};\widetilde{\sigma}, \widetilde{\delta}])\) is \(\Gr S\)-isomorphic to an \(\Gr S\)-submodule of \(\mathcal{D}[\widetilde{s};\widetilde{\sigma}, \widetilde{\delta}]\) by taking the smallest common multiple of the denominators of a generating set. Hence, the claim follows from (3).

(5) The right \(\Gr S\)-module \(\mathcal{D}[\widetilde{s};\widetilde{\sigma}, \widetilde{\delta}]\) has projective dimension at most \(1\) by (2), so the claim follows.

(6) By Lemma~\ref{lem:exteding-scalars}, the right \(\Gr S\)-modules \(\mathcal{D}[\widetilde{s};\widetilde{\sigma}, \widetilde{\delta}]\) and \(\mathcal{D} \otimes_\Lambda \Gr S\) are isomorphic. Since \(\Gr S\) is a free left \(\Lambda\)-module, we have isomorphisms:
\begin{align*}
    \Tor_1^{\Gr S}(\mathcal{D}[\widetilde{s};\widetilde{\sigma}, \widetilde{\delta}], N'') &\simeq \Tor_1^{\Gr S}(\mathcal{D} \otimes_\Lambda \Gr S, N'')\\
    &\simeq \Tor_1^\Lambda(\mathcal{D}, N'')\text{, by \cite[Cor. 10.72]{rotman_introduction_2008}}\\
    &\simeq \Tor_1^{\mathcal{D}}(\mathcal{D},N'')\text{, by \cite[Thm. 2.5(2)]{HL22}}\\
    &\simeq 0\text{, because }\mathcal{D}\text{ has dimension }0\,.
\end{align*}

(7) The short exact sequence \(0 \to N \to \operatorname{Ore}(\mathcal{D}[\fraks])^k \to C \to 0\) induces the long exact sequence
\begin{multline*}
\cdots \to \Tor_2^{\Gr S}(\mathcal{D}[\widetilde{s};\widetilde{\sigma}, \widetilde{\delta}], C) \to 
\Tor_1^{\Gr S}(\mathcal{D}[\widetilde{s};\widetilde{\sigma}, \widetilde{\delta}], N) \to \\
\Tor_1^{\Gr S}(\mathcal{D}[\widetilde{s};\widetilde{\sigma}, \widetilde{\delta}], \operatorname{Ore}(\mathcal{D}[\fraks])^k) \to \cdots\,.\end{multline*}
 The first term vanishes by (5) and the third term vanishes by (6), so that \[\Tor_1^{\Gr S}(\mathcal{D}[\widetilde{s};\widetilde{\sigma}, \widetilde{\delta}], N) = 0\] as well.

(8) Let \[\cdots \to P_k \to \cdots \to P_0 \to N \to 0\] be a projective resolution of \(N\) over \(\Gr S\). The natural isomorphism of functors \[\operatorname{Ore}(\mathcal{D}[\widetilde{s};\widetilde{\sigma}, \widetilde{\delta}]) \otimes_{\Gr S} - \simeq \operatorname{Ore}(\mathcal{D}[\widetilde{s};\widetilde{\sigma}, \widetilde{\delta}]) \otimes_{\mathcal{D}[\widetilde{s};\widetilde{\sigma}, \widetilde{\delta}]} (\mathcal{D}[\widetilde{s};\widetilde{\sigma}, \widetilde{\delta}] \otimes_{\Gr S} -)\] gives us an isomorphism of complexes \begin{equation}\label{eq:iso-complex-ore}\operatorname{Ore}(\mathcal{D}[\widetilde{s};\widetilde{\sigma}, \widetilde{\delta}]) \otimes_{\Gr S} P_* \simeq \operatorname{Ore}(\mathcal{D}[\widetilde{s};\widetilde{\sigma}, \widetilde{\delta}]) \otimes_{\mathcal{D}[\widetilde{s};\widetilde{\sigma}, \widetilde{\delta}]} (\mathcal{D}[\widetilde{s};\widetilde{\sigma}, \widetilde{\delta}] \otimes_{\Gr S} P_*)\,.\tag{\(*\)}\end{equation} Since the Ore localization is flat by \cite[Prop. 2.1.16(ii)]{robson_noncommutative_2001}, the homology of the right hand side of~(\ref{eq:iso-complex-ore}) is isomorphic to the homology of \(\mathcal{D}[\widetilde{s};\widetilde{\sigma}, \widetilde{\delta}] \otimes_{\Gr S} P_*\). Hence, we get isomorphisms \(\Tor_i^{\Gr S}(\operatorname{Ore}(\mathcal{D}[\widetilde{s};\widetilde{\sigma}, \widetilde{\delta}]), N) \simeq \Tor_i^{\Gr S}(\mathcal{D}[\widetilde{s};\widetilde{\sigma}, \widetilde{\delta}], N)\) for all \(i\geq 0\). In particular, \(\Tor_1^{Gr S}(\operatorname{Ore}(\mathcal{D}[\widetilde{s};\widetilde{\sigma}, \widetilde{\delta}]), N) = 0\) by (7).

Finally, (9) Is analogous to (8), and for (10) note that the proof of parts (5)-(8) also holds in the right \(\Gr S\)-module setting, and that (10) is a special case of (8).
\end{proof}

Observe that the embedding \(\Gr S \to \Gr \bbd\) (induced by \(\Gr \iota\)) can be extended to an embedding \(\mathcal{D}[\widetilde{s};\widetilde{\sigma}, \widetilde{\delta}] \to \operatorname{Ore}(\Gr \bbd)\) satisfying \(\widetilde{s} \mapsto \overline{s}\) in the following way: the image of \(\Lambda\) in the latter is contained in \(\mathcal{D}[t]\), and the map \(\mathcal{D}[t] \to \mathcal{D}\) defined by \(t \mapsto 1\) splits this homomorphism. Hence, the image of every full matrix is again full, and the universal property of \(\mathcal{D}\) gives us the desired extension. Therefore, every finitely generated \(\Gr S\)-submodule of \(\operatorname{Ore}(\Gr \bbd)\) is isomorphic to a finitely generated \(\Gr S\)-submodule of \(\operatorname{Ore}(\mathcal{D}[\widetilde{s};\widetilde{\sigma}, \widetilde{\delta}])^k\) for some \(k \geq 1\), from which we obtain:

\begin{cor}\label{cor:graded-tor-vanishing} \ 
\begin{enumerate}[label=\upshape (\arabic*)]
    \item \(\Gr \bbd\) has projective dimension at most \(1\) as a left and right \(\Gr S\)-module;
    \item Every left or right \(\Gr S\)-submodule of \(\Gr \bbd\) has projective dimension at most \(1\);
    \item Every finitely generated left or right \(\Gr S\)-submodule of \(\operatorname{Ore}(\Gr \bbd)\) has projective dimension at most \(1\).
    \item \(\Tor_2^{\Gr S}(\Gr \bbd, N) = 0\) for every left \(\Gr S\)-module \(N\);
    \item \(\Tor_1^{\Gr S}(\Gr \bbd, N) = 0\) for every left \(\Gr \bbd\)-module \(N\);
    \item \(\Tor_1^{\Gr S}(\Gr \bbd, N) = 0\) for every left \(\Gr S\)-submodule \(N \leq \operatorname{Ore}(\Gr \bbd)\);
    \item \(\Tor_1^{\Gr S}(\operatorname{Ore}(\Gr \bbd), N) = 0\) for every left \(\Gr S\)-submodule \(N \leq \operatorname{Ore}(\Gr \bbd)\);
    \item \(\Tor_1^{\Gr S}(M,\operatorname{Ore}(\Gr \bbd)) = 0\) for every right \(\Gr S\)-submodule \(M \leq \operatorname{Ore}(\Gr \bbd)\);
    \item \(\Tor_1^{\Gr S}(\operatorname{Ore}(\Gr \bbd), \operatorname{Ore}(\Gr \bbd)) = 0\).
\end{enumerate}
\end{cor}

\begin{rem} The vanishing results of Corollary~\ref{cor:graded-tor-vanishing} can also be proven directly, through an argument similar to that of Proposition~\ref{prop-1st-vanishing}, once one has shown that \(\Gr \iota \) makes \(\Lambda[\overline{t}]\) into a free \(\Lambda\)-module on both sides. Indeed, the set \[\{\overline{t}^k, w\mid k \geq 1\,, w\text{ a monomial in the letters }a_1,a_2,\ldots\}\] is a basis for \(\Lambda[\overline{t}]\) as a right and as a left \(\Lambda\)-module.
\end{rem}

\subsection{\texorpdfstring{$\Falgebra{G}$}{Fp[[G]]} is a Sylvester domain}

Since every free-by-\(\Z_p\) pro-\(p\) group has cohomological dimension at most \(2\), our ring \(S = \F_p[\![G]\!]\) has ``continuous'' global dimension at most \(2\) in the sense of A. Brumer (\cite[Sec. 3]{brumerPseudocompactAlgebrasProfinite1966}). By the characterization of \(P\) being a projective profinite \(S\)-module if and only if \(\CTor_1^S(\F_p,P) = 0\) given in \cite[Prop.~3.1]{brumerPseudocompactAlgebrasProfinite1966}, through dimension shifting we can improve the results of Proposition~\ref{prop-continuous-resolution}:

\begin{cor}\label{cor-dim-2-resolution} Let \(M\) be a complete and bouded filtered \(S\)-module. Then, there are bounded multisets of integers \(V_0\), \(V_1\) and \(V_2\) and a strict exact sequence of filtered \(S\)-modules \[0 \to \widehat{S(V_2)} \to \widehat{S(V_1)} \to \widehat{S(V_0)} \to M \to 0\,.\]
\end{cor}
\begin{proof} Take the exact sequence of Proposition~\ref{prop-continuous-resolution}. If \(K_i\) denotes the kernel of \(\widehat{S(V_i)} \to \widehat{S(V_{i-1})}\), then we know that \[\CTor_{1}^S(\F_p,K_i) \simeq \cdots \simeq \CTor_{i+1}^S(\F_p, K_0) \simeq \CTor_{i+2}^S(\F_p, M)\,.\] Since the continuous global dimension of \(S\) is \(2\), we have that \(\CTor_1^S(\F_p, K_1) = 0\), that is, \(K_1\) is a projective and hence free \(S\)-module.

Inducing a filtration on \(K_1\) from \(\widehat{S(V_1)}\), it becomes again a separated, complete and filtered \(S\)-module with a strict inclusion map \(K_1 \to \widehat{S(V_1)}\). Since \(\widehat{S(V_1)} \simeq S[\![(\bm{V_1}\cup\{\infty\},\infty)]\!]\) is countably based, then so is \(K_1\), which we identify with a free profinite \(S\)-module \(S[\![(X,*)]\!]\) on a countable profinite pointed set \((X,*)\). By counting how many elements of \(X\) have a non-trivial image in each homogeneous component of \(\Gr K_1\), one obtains a multiset of integers \(V\) such that the induced map \(S(V) \to K_1\) is injective and has a dense image. This implies that \(\widehat{S(V)} \to K_1\) is an isomorphism, completing the proof.
\end{proof}

We can now prove the main finiteness result of this section.

\begin{prop}\label{prop:mainbound} Let \(M\) be a finitely generated left (resp. right) \(S\)-submodule of \(\mathcal{Q}\). Then, \(\Tor_1^S(\F_p, M)\) (resp. \(\Tor_1^S(M,\F_p)\)) is finite and \(\Tor_2^S(\F_p, M) = 0\) (resp. \(\Tor_2^S(M,\F_p) = 0\)).
\end{prop}
\begin{proof} We follow the ideas of \cite[Prop. 4.13]{jaikin-zapirainExplicitConstructionUniversal2020}. We will prove the proposition for left modules only, since the proof for right modules is analogous. We can assume that \(M \subseteq \bbd\), so that \(M\) is a filtered, separated, bounded and complete \(S\)-module. Let us take \begin{equation}0 \to \widehat{S(V_2)} \overset{d_2}{\longrightarrow} \widehat{S(V_1)} \overset{d_1}{\longrightarrow} \widehat{S(V_0)} \overset{d_0}{\longrightarrow} M \to 0\tag{\(\dagger\)}\end{equation} the strict exact sequence of Corollary~\ref{cor-dim-2-resolution}. Since \(\Gr \widehat{S(V_i)} \simeq (\Gr S)(V_i)\), we have another exact sequence \begin{equation}
    0 \to (\Gr S)(V_2) \overset{\overline{d_2}}{\longrightarrow} (\Gr S)(V_1) \overset{\overline{d_1}}{\longrightarrow} (\Gr S)(V_0) \overset{\overline{d_0}}{\longrightarrow} \Gr M \to 0\,.\tag{\(\Gr \dagger\)}
\end{equation}

We know, through Proposition~\ref{prop-graded-polynomial}, that \(\Gr M\) is a \(\Gr S\)-submodule of the skew polynomial ring \[\mathcal{D}[\overline{t}][ \overline{s};\overline{\sigma},\overline{\delta}] \simeq \Gr \bbd\quad\text{ where }\overline{t} = t + J\text{ and }\overline{s} = s + J\,.\] By Corollary~\ref{cor:graded-tor-vanishing}, we know that \[\Tor_2^{\Gr S}(\Gr \bbd, \Gr M) = \Tor_1^{\Gr S}(\Gr \bbd, \Gr M) = 0\,.\] Since \(\Tor_i^{\Gr S}(\Gr \bbd, (\Gr S)(V_j)) = 0\) for every \(i > 0\) by the freeness of each \((\Gr S)(V_j)\), we can apply Lemma~\ref{lem-partial-exactness} with \(k = 4\) to deduce that \(\Gr \bbd \otimes_{\Gr S} \Gr \dagger\) is exact. Observe that we have \[\bbd \cotimes_S \widehat{S(V_i)} \simeq \widehat{\bbd(V_i)}\] by Proposition~\ref{prop-fp-acyclic}. 

We now claim that the sequence 
\begin{equation}\label{eq:tensor-d}
    0 \to \widehat{\bbd(V_2)} \overset{\operatorname{id}\cotimes d_2}{\longrightarrow} \widehat{\bbd(V_1)} \overset{\operatorname{id}\cotimes d_1}{\longrightarrow} \widehat{\bbd(V_0)} \overset{\operatorname{id}\cotimes d_0}{\longrightarrow} \bbd \cotimes_S M \to 0\tag{\(\bbd \cotimes_S \dagger\)}
\end{equation}
is again exact. There is a commutative diagram
\begin{center}
\begin{tikzcd}
0 \arrow[r] & \Gr\widehat{\bbd(V_2)} \arrow[r, "\overline{\operatorname{id} \cotimes d_2}"]                            & \Gr\widehat{\bbd(V_1)} \arrow[r, "\overline{\operatorname{id}\cotimes d_1}"]                             & \Gr\widehat{\bbd(V_0)}                            \\
0 \arrow[r] & \Gr\bbd \otimes_{\Gr S} (\Gr S)(V_2) \arrow[r, "\operatorname{id} \otimes \overline{d_2}"'] \arrow[u] & \Gr\bbd \otimes_{\Gr S} (\Gr S)(V_1) \arrow[r, "\operatorname{id} \otimes \overline{d_1}"'] \arrow[u] & \Gr\bbd \otimes_{\Gr S} (\Gr S)(V_0) \arrow[u]
\end{tikzcd}
\end{center}
where the vertical arrows are isomorphisms by \cite[Lem. 4.8]{jaikin-zapirainExplicitConstructionUniversal2020}. Since we've shown the bottom row is exact, the same holds for the top row, showing by \cite{jaikin-zapirainExplicitConstructionUniversal2020}[Prop. 4.4] that the sequence~\ref{eq:tensor-d} is exact at \(\widehat{\bbd(V_i)}\) for \(i \in \{1,2\}\). That the map \(\operatorname{id} \cotimes d_0\) is surjective is also clear, so it remains to show exactness at \(\widehat{\bbd(V_0)}\).

Let \(x \in \ker \operatorname{id} \cotimes_S d_0\). For each \(k \geq 0\), we have the following short exact sequence: \[0 \to \widehat{S(V_{1})}/d_1^{-1}(\widehat{S(V_{0})}_k) \overset{d_{1}}{\to} \widehat{S(V_0)}/\widehat{S(V_0)}_k \overset{d_0}{\to} M/d_0(\widehat{S(V_0)}_k) \to 0\,.\] Tensoring with \(\bbd/J^k\) over \(S\), we obtain the exactness of
\begin{align*}\bbd/J^k \otimes_S \widehat{S(V_{1})}/d_1^{-1}(\widehat{S(V_{0})}_k) \overset{\operatorname{id} \otimes_S d_{1}}{\to}&  \bbd/J^k \otimes_S \widehat{S(V_0)}/\widehat{S(V_0)}_k\\ &\overset{\operatorname{id}\otimes_S d_0}{\to} \bbd/J^k \otimes_S M/d_0(\widehat{S(V_0)}_k)\,.\end{align*}

Let \(x_k\) be the image of \(x\) in the middle term. Since \(x_k\) belongs to the kernel of \(\operatorname{id} \otimes_S d_0\), there exists \(y_k \in \bbd/J^k \otimes_S \widehat{S(V_{1})}/d_1^{-1}(\widehat{S(V_{0})}_k)\) such that \((\operatorname{id} \otimes_S d_{1})(y_k) = x_k\). We then choose a sequence \(z_k \in \bbd\cotimes_S \widehat{S(V_{1})}\) such that each \(z_k\) maps to \(y_k\) in the quotient. By construction, we have \(\lim_{k \to \infty} d_{1}(z_k) = x\). However, since \(\operatorname{id} \cotimes_S d_{1}\) is a filtered homomorphism, the topology on \(\bbd \cotimes_S \widehat{S(V_{1})}\) coincides with the topology induced by \(\operatorname{id} \cotimes_S d_{1}\). In particular, the sequence \(z_k\) converges to some \(z\) such that \(d_{1}(z) = x\), as desired. The exactness at \(\widehat{\bbd(V_0)}\) follows, and thus~(\ref{eq:tensor-d}) is exact as we claimed.

By Proposition~\ref{pro-acyclic-tor-resolution}, the complex~(\ref{eq:tensor-d}) is a resolution of \(M\)  over \(\bbd\) that can be used to compute \[\Tor_i^{\bbd}(\mathcal{D}, \bbd \cotimes_S M) \simeq \CTor_i^{\bbd}(\mathcal{D},\bbd \cotimes_S M)\,.\] The isomorphisms of complexes
\begin{align*}
    \mathcal{D} \otimes_{\bbd} \left(\bbd \cotimes_ S \dagger\right) & \simeq \mathcal{D} \cotimes_{\bbd} \left(\bbd \cotimes_ S \dagger\right)\,,\\
    &\simeq \left(\mathcal{D} \cotimes_{\bbd} \bbd\right) \cotimes_S \dagger\,,\\
    &\simeq \mathcal{D} \cotimes_S \dagger
\end{align*}
also gives us the isomorphisms \[\CTor_i^{\bbd}(\mathcal{D},\bbd \cotimes_S M) \simeq \CTor_i^S(\mathcal{D}, M)\simeq H_i(\mathcal{D} \cotimes_S \dagger)\,.\]

Since \(M\) is a finitely generated \(S\)-module, then \(\bbd\cotimes_S M\) is a finitely generated \(\bbd\)-module (\cite[Cor. 4.3]{jaikin-zapirainExplicitConstructionUniversal2020}), so that \begin{equation}\label{eq:fdimf1}\dim_{\mathcal{D}} \Tor_1^{\bbd}(\mathcal{D}, \bbd\cotimes_S M) < \infty\,.\end{equation} On one hand, Shapiro's isomorphism (\cite[Cor. 10.72]{rotman_introduction_2008}) gives us:
\[ \Tor_2^{\mathcal{D}[\overline{t}][\overline{s};\overline{\sigma},\overline{\delta}]}(\mathcal{D}, \mathcal{D}[\overline{t}][\overline{s};\overline{\sigma},\overline{\delta}] \otimes_{\Gr S} \Gr M) \simeq \Tor_2^{\Gr S}(\mathcal{D}, \Gr M) \simeq 0\text{, by Corollary~\ref{cor:graded-tor-vanishing}}\,.\]
On the other hand, since \[\Gr(\bbd \cotimes_S M) \simeq \Gr(\bbd \otimes_S M) \simeq \mathcal{D}[\overline{t}][\overline{s};\overline{\sigma},\overline{\delta}] \otimes_{\Gr S} \Gr M\text{ by \cite[Lem. 4.8]{jaikin-zapirainExplicitConstructionUniversal2020}}\,,\] and as a finitely generated module over the Noetherian ring \(\bbd\) the module \(\mathcal{D}\) is of type \(FP_\infty\), we get by Theorem~\ref{thm:andrei410} that \begin{equation}\label{eq:fdimf2}\Tor_2^S(\mathcal{D}, \bbd \cotimes_S M) = 0\,.\end{equation}
Since \(\F_p\)-linear independent elements in \(\F_p \cotimes_S \widehat{S(-)}\) remains \(\mathcal{D}\)-linear independent in \(\mathcal{D} \cotimes_S \widehat{S(-)}\) through tensoring with \(\mathcal{D} \otimes_{\F_p}\), we have that \[\dim_{\F_p} \ker(\widehat{\F_p(V_n)} \to \widehat{\F_p(V_{n-1})}) \leq \dim_{\mathcal{D}} \ker(\widehat{\mathcal{D}(V_n)} \to \widehat{\mathcal{D}(V_{n-1})})\] for \(n \in \{1,2\}\). Then, the claims follow from equations~(\ref{eq:fdimf1}) and~(\ref{eq:fdimf2}).
\end{proof}

\begin{cor}\label{cor:condition2} Let \(M\) be a finitely generated left or right \(S\)-submodule of \(\mathcal{Q}\). Then for any exact sequence \(0 \to I \to S^d \to M \to 0\) of left \(S\)-modules, \(I\) is free of finite rank.
\end{cor}
\begin{proof} Again, we will only prove the corollary for left \(S\)-submodules. Since \(\F_p\) with trivial \(G\)-action is the only simple and pseudocompact \(S\)-module, by Proposition~\ref{prop:mainbound} and \cite[Cor. 3.2]{brumerPseudocompactAlgebrasProfinite1966}, we know that \(M\) has projective dimension at most \(1\) as an \(S\)-module. Hence, \(I\) is projective by \cite[Prop. 8.6]{rotman_introduction_2008} and thus free by Kaplansky's theorem on local rings. It's rank as a free \(S\)-module is then given by \[\dim_{\F_p} \F_p \cotimes_S I = d - \dim_{\F_p} \F_p \cotimes_S M + \dim_{\F_p} \Tor_1^S(\F_p,M)\,,\] which is once again seen to be finite by Proposition~\ref{prop:mainbound}.
\end{proof}

We can now prove the main theorem of this section.
\begin{thm}\label{sylvester}
Let $G$ be a torsion-free finitely generated virtually free-by-$\Z_p$ pro-$p$ group. Then $\Falgebra{G}$ is a Sylvester domain.
\end{thm}
\begin{proof}
By Lemma \ref{lem:virtual-flag} and Theorem \ref{thm:sylvester-virtual}, we can assume that $G$ is  a mild flag free-by-$\Z_p$ pro-\(p\) group. We use the notation of this section.

 We want to apply Theorem~\ref{thm-jaikin-criterion-universality} to the embedding \(S \to\mathcal Q \) constructed throughout Section~\ref{sec:group-embeddings}. The condition (2) was shown in Corollary~\ref{cor:condition2}, so it only remains to show that \[\Tor_1^S(\mathcal{Q},\mathcal{Q}) = 0\,.\] Since \(\mathcal{Q}\) is flat as a \(\bbd\)-module (\cite[Prop. 2.1.16]{robson_noncommutative_2001}), if \(P_* \to\mathcal \mathcal{Q} \) is a projective resolution of \(\mathcal{Q}\) over \(S\) we have the isomorphisms: 
 \begin{align*}
   \mathcal{Q}  \otimes_\bbd \Tor_i^S(\bbd,\mathcal D_{\Falgebra{G}} ) &=\mathcal{Q}  \otimes_\bbd H_i(\bbd \otimes_S P_*) \\
    &\simeq H_i (\mathcal{Q} \otimes_{\bbd} \bbd \otimes_S P_*) \\
    &\simeq H_i(\mathcal{Q} \otimes_S P_*) = \Tor_i^S(\mathcal{Q},\mathcal{Q})\,.
\end{align*}
Hence, if \(\Tor_i^S(\bbd,\mathcal{Q})\) vanishes, then so does \(\Tor_i^S(\mathcal{Q},\mathcal{Q})\). The same reasoning, now applied to the right factor, gives us  \[\Tor_i^S(\bbd,\bbd) = 0 \implies \Tor_i^S(\mathcal{Q},\mathcal{Q}) = 0\,.\]

Hence, because the \(\Tor\) functor commutes with direct limits, it suffices to show that for every finitely generated right and left \(S\)-submodules \(M\) and \(N\) of \(\bbd\), one has \(\Tor_1^S(M,N) = 0\). Note that \(\Gr M\) and \(\Gr N\) are \(\Gr S\)-submodules of \(\Gr \bbd\) in the filtration induced by the powers of the ideal \(J\) generated by \(t\) and \(s\). The short exact sequence \[0 \to \Gr M \to \Gr \bbd \to Q \to 0\] induces the long exact sequence \[\cdots \to \Tor_2^{\Gr S}(Q, \Gr N) \to \Tor_1^{\Gr S}(\Gr M, \Gr N) \to \Tor_1^{\Gr S}(\Gr \bbd, \Gr N) \to \cdots\] on the \(\Tor\) groups. Since by Corollary~\ref{cor:graded-tor-vanishing} the module \(\Gr N\) has projective dimension at most \(1\) and \(\Tor_1^{\Gr S}(\Gr \bbd, \Gr N) = 0\), we get \[\Tor_1^{\Gr S}(\Gr M, \Gr N) = 0\,.\] We have shown in Proposition~\ref{prop:mainbound} that \(M\) is of type \(FP_\infty\): given a presentation \[0 \to K \to S^n \to M \to 0\,,\] with \(n\) minimal we have that \(K\) is finitely generated and thus profinite as \(\Tor_1^S(M,\F_p) = \Tor_0^S(K,\F_p)\) is finite, an argument we can now repeat with \(M\) replaced by \(K\). Moreover, the vanishing of \[\CTor_1^S(K,\F_p) \simeq \Tor_1^S(K,\F_p) \simeq  \Tor_2^S(M,\F_p)\] shows that it is projective by \cite[Lem. 2.1(ii) and Prop. 3.1]{brumerPseudocompactAlgebrasProfinite1966} and hence free. Since \(S\) is bounded and \(N\) is finitely generated, it is also bounded, and thus the claim now follows from Theorem~\ref{thm:andrei410}.
\end{proof}

\section{The equality \texorpdfstring{$\rk_G=\irk_{\Falgebra{G}}$}{rk\_G = irk\_Fp[[G]]}}\label{sec:atiyah}
The main result of this section is the following theorem.
\begin{thm}\label{thm:approximation} 
Let \(G\) be a pro-\(p\) group, \((\mathcal{Q},w)\) a division ring with a non-archimedean discrete valuation \(w\colon\mathcal Q  \to \{p^k \mid k \in \Z\}\cup \{0\}\) and \(\varphi\colon \Falgebra{G} \to\mathcal Q \) a continuous embedding with respect to the topology on \(\mathcal{Q}\) induced by \(w\). 
Then \[\varphi^\# \rk_{\mathcal{Q}}\leq \rk_G\,.\]
\end{thm}

We split again the proof into several lemmas. For each \(i \in \Z\) define following subsets of \(G\):
\[G_i = \left\{g \in G\mid w(\varphi(g-1)) \leq p^{-i}\right\}\,.\]

Each \(G_i\) is a clopen and non-empty subset of \(G\), as \(1 \in G_i\) for every \(i\). Moreover, since \(G\) is compact, there exists a minimum \[k_0 = \min\{k \in \Z \mid w(\varphi(g - 1)) \leq p^{-k}\text{ for all } g \in G\}\,.\] This implies that \(G_i = G\) for every \(i \leq k_0\). We note that:

\begin{lem}\label{lem:positive-fil} The integer \(k_0\) is non-negative. In particular, \(G_0 = G\). Moreover, \(w(\varphi(g)) = 1\) for every \(g \in G\).
\end{lem}
\begin{proof} Suppose by contradiction that \(k_0 < 0\) and let \(g \in G\) be such that \(w(\varphi(g-1)) = p^{-k_0}\). Observe that since \(\Falgebra{G}\) embeds into a division ring, \(G\) is torsion free, and in particular \(g^p \neq 1\). Since \(\mathcal{Q}\) has characteristic \(p\), we get \[w(\varphi(g^p - 1)) = w(\varphi(g-1)^p) = p^{-pk_0} > p^{-k_0}\,,\] contradicting the definition of \(k_0\).

Suppose now that \(w(\varphi(g)) > 1\) for some \(g \in G\). Then, since \(w(\varphi(g)) \neq w(-1) = 1\), we'd have \[w(\varphi(g-1)) = \max\{w(\varphi(g)), w(-1))\} = w(\varphi(g)) > 1\,,\] a contradiction. On the other hand, if \(w(\varphi(g)) < 1\), then the same argument shows that \(w(\varphi(g-1)) = 1\). We get a contradiction by considering: \[w(\varphi(g^{-1} - 1)) = w(\varphi(g^{-1}(1-g))) = w(\varphi(g^{-1}))w(\varphi(g-1)) = w(\varphi(g^{-1})) > 1\,.\qedhere\]
\end{proof}

By the continuity of \(w\) it follows that the \(G_i\) form a chain of normal open subgroups of \(G\), for it is easily verified that each \(G_i\) contains \(1\) and is closed under products, inverses and conjugation. Since \(w(q) = 0\) if and only if \(q = 0\), we can also conclude that \(\bigcap_{i \geq 0} G_i = \{1\}\), that is, the chain \(G_i\) is residual in \(G\). Hence, this chain can be used to compute the rank function \(\rk_G\) of Example~\ref{exmp-rk-G}.  

Let \(\mathcal{Q}_i\) be the set of elements \(q \in\mathcal Q \) with \(v(q) \leq p^{-i}\). In particular, \(\mathcal{Q}_0\) is an \(\F_p\)-algebra containing \(\varphi(\Falgebra{G})\) with ideals \(\mathcal{Q}_i\) for each \(i \geq 0\).

\begin{lem}\label{lem:lemuniform} There exists a division ring extension \(\mathcal{Q} \leq\mathcal Q '\) and an extension \(w'\) of the valuation \(w\) to \(\mathcal{Q}'\) such that \(\mathcal{Q}'\) contains a central element \(z\) with \(w'(z) = p^{-1}\). In particular, \(\mathcal{Q}'_i = z^i\mathcal{Q}'_0\) for every \(i \in \Z\).
\end{lem}
\begin{proof} Consider the polynomial ring \(\mathcal{Q}[z]\) and define \(w'\colon\mathcal Q [z] \to \{p^k\mid k \in \Z\} \cup \{0\}\) to be \[w'\left(\sum_{i=0}^n q_iz^i\right) = \max\{w(q_i)p^{-i}\}\,.\] We check that \(w'\) is an additive non-archimedian valuation on \(\mathcal Q[z]\), for it is immediate that it is discrete.

Let \(r = \sum_{i=1}^n q_iz^i\) and \(r' = \sum_{i=1}^n q_i' z^i\) be two arbitrary elements in \(\mathcal{Q}[z]\). We have \(w'(r) = 0\) if and only if all of the \(q_i\) have valuation \(0\), that is, if and only if \(q_i = 0\) for all \(i\). Hence, \(w'(r) = 0\) implies \(r = 0\). Suppose the minimal \(w'\) valuation on \(r\) and \(r'\) is attained at the terms \(q_mz^m\) and \(q_n'z^n\) with \(m\) and \(n\) minimal for this properties. Then, for all \(0 \leq i\,,j \leq n\), we have:
\[w(q_m)p^{-m} \geq w(q_i)p^{-i}\,,\]
\[w(q'_n)p^{-n} \geq w(q'_j)p^{-j}\,.\]
Multiplying both inequalities implies:
\[w(q_mq'_n)p^{-m-n} \geq w(q_iq'_j)p^{-i-j}\quad \forall i\,,j\,.\]
In particular, for every \(1 \leq k \leq 2n\), we have:
\begin{equation}\label{note7-eq1}w(q_mq'_n)p^{-n-m} \geq \max\{w(q_iq'_j)\mid i + j = k\}p^{-k} \geq w\left(\sum_{i+j=k} q_iq'_j\right)p^{-k}\,.\end{equation}

Suppose now, by contradiction, that there exists \(i\neq m\) and \(j \neq n\) with \(i+j = m+n\) such that \(w(q_iq'_j) \geq w(q_mq'_n)\). Without loss of generality, we may suppose that \(i < m\), so that \(n < j\). By the minimality of \(m\) and \(n\), we must have \(w(q_i)p^{-i} < w(q_m)p^{-m}\) and \(w(q'_j)p^{-j} \leq w(q'_n)p^{-n}\). This implies:
\[w(q_i)/w(q_m) < p^{i-m} = p^{n-j} \leq w(q'_n)/w(q'_j)\,,\] and thus \(w(q_iq'_j) < w(q_mq'_n)\), a contradiction. Therefore, if \(i + j = m+n\), we have \(w(q_iq'_j) < w(q_mq'_n)\). Hence 
\begin{equation}\label{note7-eq2}w(q_mq'_n)p^{-n-m} = w\left(\sum_{i+j = m+n} q_iq_j\right)p^{-m-n}\,.\end{equation}

We have \[rr' = \sum_{k=0}^{2n} \left(\sum_{i+j = k} q_iq'_j\right) z^k\,.\] Since \[w\left(\sum_{i+j=k} q_iq'_j\right)p^{-k}\leq w(q_mq'_n)p^{-n-m} = w\left(\sum_{i+j=m+n} q_iq'_j\right)p^{-n-m}\text{ by~(\ref{note7-eq1}) and~(\ref{note7-eq2})}\,,\] the minimal \(w'\) valuation on \(rr'\) must be attained in the \(z\)-degree \(m+n\), with value exactly \(w'(r)w'(r')\).

At last, we must show that \(w'(r + r') \leq \max\{w'(r)\,, w'(r')\}\). We may suppose, without loss of generality, that \(w(q_m)p^{-m} = w'(r) \leq w'(r') = w(q'_n)p^{-n}\). Observe that in this case \[w(q_i + q'_i)p^{-i} \leq \max\{w(q_i)\,,w(q'_i)\}p^{-i} \leq w(q_m)p^{-m} = w'(r)\,.\] Since \[r+r' = \sum_{i=1}^n (q_i + q'_i)z^i\,,\] we are done.

Now \(\mathcal{Q}[z]\) is a Noetherian domain and thus we can construct the division ring of rational functions \(\mathcal{Q}' =\mathcal Q (z)\) by localizing at the non-zero elements. Hence, the valuation \(w'\) also extends to \(\mathcal{Q}'\), which concludes the proof.
\end{proof}

By virtue of Lemma~\ref{lem:lemuniform}, we can suppose that \(\mathcal{Q}\) contains a central element \(z\) of valuation \(p^{-1}\), for a division ring extension does not change the induced rank function on \(\Falgebra{G}\). Consider now \(\mathcal{Q}_0\), the ring of elements with non-negative valuation. We note that \(\varphi(\Falgebra{G}) \leq\mathcal \mathcal{Q} _0\) and the non-zero left or right ideals of \(\mathcal{Q}_0\) are the two-sided ideals \(z^i\mathcal{Q}_0\) for \(i \geq 0\): if \(z^iu = uz^i \in \mathcal{Q}_0\) with \(w(u) = 1\), then \(u^{-1} \in \mathcal{Q}_0\) and hence the ideals generated by \(z^iu\) and \(z^i\) coincide. This implies that \(\mathcal{Q}_0/z\mathcal{Q}_0\) is the only simple \(\mathcal{Q}_0\)-module, and the length of the \(\mathcal{Q}_0\)-modules \(\mathcal{Q}_0/z^i\mathcal{Q}_0\) are precisely \(i\).

Let \(A \in \Mat_{n\times n}(\Falgebra{G})\) be a matrix such that \(\varphi(A)\) is invertible in \(\mathcal{Q}\), that is, \(\varphi^\# \rk_{\mathcal{Q}}(A) = n\). We define the left \(\Falgebra{G}\)-module \(M = \Falgebra{G}^n/\Falgebra{G}^nA\) and the \(\mathcal{Q}_0\)-length of \(M\) over an open subgroup \(U\) to be \[l_U(M) = \len_{\mathcal{Q}_0}(\mathcal{Q}_0 \otimes_{\Falgebra{U}} M)\,,\] the length of the induced \(\mathcal{Q}_0\)-module. The proof henceforth will follow similar arguments to the ones in \cite[Sec. 5]{jaikin-zapirainExplicitConstructionUniversal2020}.

Let \(\dim_{G/G_i}\) and \(\varphi^\#\dim_{\mathcal{Q}}\) be the Sylvester module rank functions associated to \(\rk_{G/G_i}\) and \(\varphi^\#\rk_{\mathcal{Q}}\). Then, since it suffices to show that \(\rk_G(A) = \lim_{i \to \infty} \rk_{G/G_i} (A)= n\), by the relation \(\rk_G(A) = n - \dim_G(M)\) it only remains to show that \[\lim_{i \to \infty} \dim_{G/G_i}(M) = 0\,.\]

\begin{lem}\label{lem:lem-finite} We have that \(l_G(M)\) is finite.
\end{lem}
\begin{proof} Since \(A\) is invertible over \(\varphi\), we have \[\varphi^\#\dim_{\mathcal{Q}} M = \dim_{\mathcal{Q}}\mathcal Q  \otimes_{\Falgebra{G}} M = 0\,.\] 
Observing that \[\mathcal{Q} \otimes_{\Falgebra{G}} M  \simeq\mathcal Q \otimes_{\mathcal{Q}_0} \left(\mathcal{Q}_0 \otimes_{\Falgebra{G}} M\right)\,,\] 
this implies that the \(\mathcal{Q}_0\)-module \(\mathcal{Q}_0 \otimes_{\Falgebra{G}} M\) is torsion. Given that \(M\) is finitely generated and every proper quotient of \(\mathcal{Q}_0\) is of the form \(\mathcal{Q}_0/z^i\mathcal{Q}_0\) and thus has finite length, that module must have finite length over \(\mathcal{Q}_0\).
\end{proof}

\begin{lem}\label{lem:len-scaling} Let \(U \leq V\) be open subgroups of \(G\) with \(|V\colon U| = p\). Then, \(l_U(M) \leq pl_V(M)\). In particular, \(l_U(M)\leq  |G\colon U|l_G(M)\).
\end{lem}
\begin{proof} Let \(L =\mathcal{Q} _0 \otimes_{\Falgebra{U}} M\) and take any element \(g \in V \smallsetminus U\). Define the function \(\tau\colon \mathcal{Q} _0  \times M \to L\) by \[\tau(q,m) = qg \otimes g^{-1}m - q\otimes m\,.\] Since \(U\) must be normal in \(V\), we have that \(\tau(qu,m) = \tau(q,um)\) for every \(u \in U\), that is, \(\tau\) induces a homomorphism of left \(\mathcal{Q}_0\)-modules \(\psi\colon L \to L\) such that \[\psi(q\otimes m) = qg\otimes g^{-1}m - q\otimes m\,.\]

It is immediate to verify that \(\psi^i(q\otimes m) = \sum_{j=0}^i (-1)^j\binom{i}{j}qg^{i-j}\otimes g^{j-i}m\) and therefore \(\psi^p = 0\). Moreover, \(L/\psi(L) \simeq \mathcal{Q} _0  \otimes_{\Falgebra{V}} M\), and by applying \(\psi\) we find that this surjects onto \(\psi^i(L)/\psi^{i+1}(L)\) for every \(i\). Hence:
\[l_U(M) = \len_{\mathcal{Q}_0 } L \leq p\cdot \len_{\mathcal{Q}_0 }\mathcal{Q} _0 \otimes_{\Falgebra{V}} M = pl_V(M)\,.\qedhere\]
\end{proof}
Recall that $\overline{\mathcal Q}=\mathcal Q_0/\mathcal Q_1$ is a division ring.
\begin{lem}\label{lem:len-lemma} We have that  \[\dim_{\overline{\mathcal Q}} \overline{\mathcal Q} \otimes_{\Falgebra{G_i}} M \leq \frac{l_{G_i}(M)}{i}\,.\]
\end{lem}
\begin{proof} When we view \(M\) as a left \(\Falgebra{G_i}\)-module, it has a presentation of the form \[\Falgebra{G_i}^{n|G\colon G_i|}/\Falgebra{G_i}^{n|G\colon G_i|}B\] for some matrix \(B \in \Mat_{n|G\colon G_i| \times n|G\colon G_i|}(\Falgebra{G_i})\). By projecting the matrix onto \(\F_p\) and applying elementary matrix operations, we can find one such matrix \(B\) satisfying \[B = \begin{pmatrix}\operatorname{Id}_a + B_1 & B_2\\ B_3 & B_4\end{pmatrix}\] where the \(B_i\) are matrices with entries in the augmentation ideal \(I_{G_i}\) of \(\Falgebra{G_i}\) and \(n|G\colon G_i| - a = \dim_{\F_p} M/I_{G_i}\). Since the \(B_i\) vanish when acting on \(\overline{\mathcal Q}\), one gets that 
\[\dim_{\overline{\mathcal Q}} \overline{\mathcal Q} \otimes_{\Falgebra{G_i}} M = \dim_{\F_p} \F_p \otimes_{\Falgebra{G_i}} M = \dim_{\F_p} M/I_{G_i}M = n|G\colon G_i| - a\,.\]

By definition, we have that for any $x\in I_{G_i}$, \(v(\varphi(x)) \geq i\), so that \(\varphi(I_{G_i}) \leq z^i\mathcal{Q}_0\). Hence, there are matrices \(B_2'\) and \(B_4'\) over \(\mathcal{Q}_0\) such that we can express \(\varphi(B)\) as \[\varphi(B) = \underbrace{\begin{pmatrix} \operatorname{Id}_a + \varphi(B_1) & B_2'\\ \varphi(B_3) & B_4'\end{pmatrix}}_{C_1}\underbrace{\begin{pmatrix}\operatorname{Id}_a & 0\\ 0 & z^i\operatorname{Id}_{n|G\colon G_i| - a}\end{pmatrix}}_{C_2}\,.\] This implies that \[(\mathcal{Q}_0/z^i\mathcal{Q}_0)^{n|G\colon G_i| - a} =\mathcal{Q} _0^{n|G\colon G_i|}/\mathcal{Q}_0^{n|G\colon G_i|}C_2\] is a quotient of \(\mathcal{Q}_0 \otimes_{\Falgebra{G_i}} M \simeq\mathcal{Q}_0^{n|G\colon G_i|}/\mathcal{Q}_0^{n|G\colon G_i|}\varphi(B)\). Therefore, we obtain that
\begin{align*}
    l_{G_i}(M) &= \len_{\mathcal{Q}_0}\mathcal{Q} _0\otimes_{\Falgebra{G_i}} M\\
    &\geq \len_{\mathcal{Q}_0} (\mathcal{Q}_0/z^i\mathcal{Q}_0)^{n|G\colon G_i|-a}\\
    &= i(n|G\colon G_i|-a) = i \dim_{\overline{\mathcal Q}} \overline{\mathcal Q} \otimes_{\Falgebra{G_i}} M\,.\qedhere
\end{align*}
\end{proof}

\begin{proof}[Proof of Theorem \ref{main}] Note that, by Lemma~\ref{lem:lem-monotonic}, to show that \(\varphi^\#\rk_{\mathcal{Q}} \leq \rk_G\) it suffices to show that they coincide over square matrices invertible under \(\varphi\). Indeed, any matrix \(A \in \Mat_{n\times m}(\Falgebra{G})\) contains a maximal submatrix \(B \in \Mat_{k \times k}(\Falgebra{G})\) such that \[\varphi^\#\rk_{\mathcal{Q}}(B) = \varphi^\#\rk_{\mathcal{Q}} A = k\,,\] and if \(\rk_G(B) = k\) the lemma shows us that \(k \leq \rk_G(A)\).

Thus, let \(M = \Falgebra{G}^n/\Falgebra{G}^nA\) where \(A \in \Mat_{n\times n}(\Falgebra{G})\) is invertible under \(\varphi\). We want to show that $\dim_G M=0$.
We have that
\begin{align*}
   \frac{\dim_{\F_p} \F_p \otimes_{\Falgebra{G_i}} M}{|G\colon G_i|}  &= \frac{\dim_{\overline{\mathcal{Q}}} \overline{\mathcal{Q}} \otimes_{\Falgebra{G_i}} M}{|G\colon G_i|}\\
    &\leq \frac{l_{G_i}(M)}{i|G\colon G_i|}\text{, by Lemma~\ref{lem:len-lemma},}\\
    &\leq \frac{l_G(M)}{i}\text{, by Lemma~\ref{lem:len-scaling}}
\end{align*}
Since \(l_G(M)\) is finite by Lemma~\ref{lem:lem-finite}, we have \(\lim_{i \to \infty} \frac{l_G(M)}{i} = 0\), which concludes the proof.

Now we are ready to finish the proof of Theorem \ref{main}. By Lemma \ref{lem:virtual-flag} and Corollary  \ref{cor-atiyah-invariant}, we can assume that $G$ is  a mild flag free-by-\(\Z_p\) pro-\(p\) group.  Let \(\mathcal{Q}\) and \(\bbd = \mathcal{D}[\![t]\!][\![s;\sigma,\delta]\!]\) be as in Section~\ref{sec:group-embeddings}. Observe that the valuation \(w\) on \(\bbd\) naturally extends to \(\mathcal{Q}\) and the embedding \(\Falgebra{G} \to\mathcal{Q}\) is continuous by Lemma~\ref{lem-magnus-valuation-embedding}. By Theorem \ref{thm:approximation},  \(\varphi^\# \rk_{\mathcal{Q}}\leq \rk_G\,\).
In the proof of Theorem \ref{sylvester} we showed that $\varphi^\# \rk_{\mathcal{Q}}=\irk_{\Falgebra{G}}$. Since $\irk_{\Falgebra{G}}\ge \rk_G$, we have that $\irk_{\Falgebra{G}}= \rk_G$.
\end{proof}

\section{Abstract subgroups of  free-by-\texorpdfstring{$\Z_p$}{Zp} pro-\texorpdfstring{$p$}{p} groups} \label{abstract}
In this section we discuss some  consequences  that Theorem \ref{main} implies in the study of the L\"uck approximation for  finitely generated (abstract) subgroups of a free-by-$\Z_p$ pro-$p$ groups. 

Let us first recall the statement of the L\"uck approximation.
Let $\Gamma$ be a group and let $K$ be a field. For every matrix $A\in \Mat_{n\times m}(K[\Gamma])$ and every normal  subgroup $N$ of $\Gamma$ of finite index let us define  
$$\begin{array}{cccc}
\phi_{\Gamma/N}^A: & K[\Gamma/N]^n & \to & K[\Gamma/N]^m \\
 &(x_1,\ldots, x_n)&\mapsto & (x_1,\ldots, x_n)A\end{array}.$$
 This is a $K$-linear map between two finite-dimensional $K$-vector spaces. Thus, we  can define a Sylvester rank function of $K[\Gamma]$ by means of 
\begin{equation}\label{rkfinite} \rk_{\Gamma/N}(A)=\frac{\dim_K \im \phi_{\Gamma/N}^A}{|\Gamma:N|}=n-\frac{\dim_K \ker \phi_{\Gamma/N}^A}{|\Gamma:N|}.\end{equation}

\begin{Conjecture}[The L\"uck approximation conjecture] \label{conresfinite} Let $\Gamma$ be a group, $K$ be a field and $\Gamma>N_1>N_2>\ldots $   a  descending chain of normal subgroups of $\Gamma$ of finite index with trivial intersection.  Let $A$ be a matrix over $K[\Gamma]$. Then the following holds.
 \begin{enumerate} 
\item[(1)]  The sequence $\{\rk_{\Gamma/N_i}(A)\}_{i\ge 1}$ converges.
\item[(2)]  The limit $\displaystyle  \lim_{i\to \infty} \rk_{\Gamma/N_i}(A)$ does not depend on the chain  $\Gamma>N_1>N_2>\ldots $.
\item[(3)]  If moreover $\Gamma$ is locally indicable, then there exists a universal embedding $K[\Gamma]\to \mathcal Q$ and   $\displaystyle  \lim_{i\to \infty} \rk_{\Gamma/N_i}(A)=\rk_{K[\Gamma]}(A).$
 \end{enumerate}

\end{Conjecture}
If $K$ is of characteristic 0,  the parts (1) and (2) of Conjecture \ref{conresfinite} are known to be true \cite{Ja19}, and for part (3) we know that there exists an embedding  $\varphi: K[\Gamma]\to \mathcal Q$ into a division ring $\mathcal Q$ such that $\displaystyle  \lim_{i\to \infty} \rk_{\Gamma/N_i}(A)=\rk_{\mathcal Q}(\varphi(A))$ \cite{JL20}, but we still do not know whether $\varphi$ is universal in general \cite{Ja21}.  If $K$ is of positive characteristic, the parts (1), (2)  and (3) are only known when $\Gamma$ is amenable (\cite{campbell_l2-betti_2019, Ja21}).
 
Let now $G$ be a free-by-$\Z_p$ pro-$p$ group and $\Gamma$ a finitely generated (abstract) subgroup of $G$.  It is clear that $\Gamma$ is locally indicable and from \cite[Theorem 3.7]{Ja21} we know that there exist a universal division $\F_p[\Gamma]$-ring of fractions $\D_{\F_p[\Gamma]}$. The following theorem provides a particular case of Conjecture \ref{conresfinite} for $\Gamma$.
\begin{thm}
Let $G$ be a free-by-$\Z_p$ pro-$p$ group and $\Gamma$ a finitely generated (abstract) subgroup of $G$. Let $G>U_1>U_2>\ldots $ be a chain of normal open subgroups of $G$ with trivial intersections. Let $H_i=\Gamma\cap U_i$. Then for every matrix $A$ over $\F_p[\Gamma]$, $\displaystyle  \lim_{i\to \infty} \rk_{\Gamma/H_i}(A) =\rk_{\F_p[\Gamma]}(A).$
\end{thm}
\begin{proof}
Since the closure of $\Gamma$ in $G$ is a finitely generated free-by-$\Z_p$ pro-$p$ group, we can assume that  \(\Gamma\) is dense in $G$, so that $G$ is topologically finitely generated. Hence, this implies that the inclusion map \(\Gamma \to G\) induces isomorphisms \(\Gamma/H_i \simeq G/U_i\) for every \(i\geq 1\).

By Theorem \ref{main}, there exists $\D_{\Falgebra{G}}$. Denote by $\varphi:\F_p[\Gamma]\to \D_{\Falgebra{G}}$ the corresponding embedding. Theorem \ref{main} implies also that  for every matrix $A$ over $\F_p[\Gamma]$, 
$$\displaystyle  \lim_{i\to \infty} \rk_{\Gamma/H_i}(A) = \lim_{i \to \infty} \rk_{G/U_i}(A) =\rk_{\D_{\Falgebra{G}}}(\varphi(A)).$$
Thus, we have to show that $\varphi$ is universal. 
Let $N$ be a normal free pro-$p$ subgroup of $G$ such that $G/N\cong \Z_p$. Put  $H=N\cap \Gamma$. Denote by $\D_\Gamma$ (resp. \(\D_H\)) the division closure of $\F_p[\Gamma]$ (resp. \(\F_p[H]\)) in $\D_{\Falgebra{G}}$ and by $R$ the subring generated by $\D_H$ and $\F_p[\Gamma]$.

In \cite[Prop. 3.5]{Ja24}, it was shown that \(\D_H\) is the universal division ring of fractions of \(\F_p[H]\). We also claim that the induced surjective map \(\D_H * (\Gamma/H) \to R\) is an isomorphism. 
Let $\gamma_1,\ldots, \gamma_n$ lie in different classes modulo $H$.  
There exists an open normal subgroup $U$ of $G$ containing $N$ such that $\gamma_1,\ldots, \gamma_n$ lie in different clases modulo $U$. By Proposition \ref{prop:crossed-division-ring}(d), $\displaystyle\sum_{i=1}^n \D_H\gamma_i$ is direct. Hence, it sends linearly independent elements to linearly independent elements, establishing its injectivity. 

These two facts together imply that the ring $R$ is isomorphic to a crossed product $\D_{\F_p[H]}*(\Gamma/H)$.
Since \(\Gamma/H\) is   finitely generated,  \(R\) a Noetherian ring.   Thus, $\D_\Gamma$ is isomorphic to the Ore division ring of fractions of $R$. Hence, we obtain that $\D_\Gamma$ will be the universal division ring of fraction of $\F_p[\Gamma]$.
\end{proof}

\appendix
\section{Localization of profinite rings}
 
Let \(R\) be an associative ring and \(\Sigma\) a collection   of square matrices over \(R\).  The universal localization of \(R\) with respect to \(\Sigma\) is a ring \(R_\Sigma\) that comes with a ring homomorphism \(\lambda\colon R  \to R_\Sigma\) such that \(\lambda(A)\) is invertible for every \(A \in \Sigma\), and \(R_\Sigma\) is universal for that property.

 We denote by  $\widetilde \Sigma$ the set of all square matrices over \(R\) that become invertible under \(\lambda\). In general, $\widetilde \Sigma$ is greater than $\Sigma$. However $R_{\widetilde \Sigma}$ is $R$-isomorphic to $R_{\Sigma}$. We say that $\Sigma$ is \emph{complete} if $\Sigma=\widetilde \Sigma$. 

It is clear that a complete collection  of square matrices is \emph{multiplicative}, that is it contains 
 \(\GL_n(R)\) for every \(n\), is closed under products when defined and if \(A\,, B \in \Sigma\) then \(\begin{pmatrix} A & C\\ 0 & B\end{pmatrix} \in \Sigma\) for any matrix \(C\) of the appropriate size. It also has  the following property regarding diagonal summands on \(\Sigma\):
\begin{equation}\label{property-DS}
    \text{If }\begin{pmatrix}A & 0\\0 & B\end{pmatrix} \in \Sigma\text{, then }A\,,B \in \Sigma\,.\tag{DS}
\end{equation}

J. Beachy has shown in \cite{MR1344222} that \(R_\Sigma\) and \(\lambda\) can be constructed as follows. From now on we assume that $\Sigma$ is complete.
 Let \(R^n\) denote the rows of \(n\) elements in \(R\), \(^nR\) the columns of \(n\) elements in \(R\) and \(\Sigma_n\) denote the subset of \(n\)-by-\(n\) matrices in \(\Sigma\). Consider the set \(T_\Sigma\) given by the disjoint union of all the products \(R^n \times \Sigma_n \times {^nR}\), that is, the triples \((a,C,x)\) where \(a\) is a row matrix over \(R\) of length \(n\), \(C \in \Sigma_n\) and \(x\) is a column matrix over \(R\) of height \(n\).
We define two binary operations of sum and product in \(T_\Sigma\) as follows:
\[(a,C,x) + (b,D,y) = \left(\begin{pmatrix}
    a & b
\end{pmatrix}\,, \begin{pmatrix} C & 0 \\ 0 & D\end{pmatrix}\,, \begin{pmatrix} x \\ y\end{pmatrix}\right)\,,\]
\[(a,C,x)\cdot (b,D,y) = \left(\begin{pmatrix} a & 0\end{pmatrix}\,, \begin{pmatrix}C & -xb\\ 0 & D\end{pmatrix}\,, \begin{pmatrix} 0 \\ y \end{pmatrix}\right)\,,\]
where the matrices make sense because \(\Sigma\) is multiplicative, \(x\) is an \(n\times 1\) matrix and \(b\) is an \(1\times m\) matrix.

We define an equivalence relation \(\mathcal{R}_1 \subset T_\Sigma \times T_\Sigma\) on \(T_\Sigma\) by means of \((a,C,x) \sim_{\mathcal{R}_1} (b,D,y)\) if there exists invertible matrices \(U, V\) over \(R\) such that \(b = aU\), \(y = Vx\) and \(D = VCU\). Observe that in particular \(C\) and \(D\) must have the same size if \((a,C,x) \sim_{\mathcal{R}_1}(b,D,y)\). We denote the quotient \(T_\Sigma/\sim_{\mathcal{R}_1}\) by \(\Sigma^{-1}R\) and the equivalence class of a triple \((a,C,x)\) in \(\Sigma^{-1}R\) by \((a\colon C \colon x)\).

The operations of sum and product are well defined on the equivalence classes and descend to a sum and a product on the quotient \(T_\Sigma\). With the sum, it becomes a commutative semigroup. To obtain an abelian additive group, we further introduce a new equivalence relation.

Let \(\Sigma^{-1}_0 R\) be the subsemigroup of \(\Sigma^{-1}R\) generated by all the elements of the form \((a\colon C\colon 0)\) and \((0\colon C\colon y)\). It is shown in \cite{MR1344222} that all the elements of \(\Sigma^{-1}_0 R\) are either of that form or of the form \(\left(\begin{pmatrix}a & 0\end{pmatrix}\colon \begin{pmatrix} A & 0\\0 & D\end{pmatrix}\colon \begin{pmatrix} 0 \\ x\end{pmatrix}\right)\) for blocks of the appropriate sizes. This subsemigroup allows us to define the congruence relation \(\mathcal{R}_2 \subseteq \Sigma^{-1}R \times \Sigma^{-1}R\) on \(\Sigma^{-1} R\), where two elements \(\overline{p}\,, \overline{q} \in \Sigma^{-1} R\) are equivalent if there exists \(\overline{z}_1\,, \overline{z}_2 \in \Sigma_0^{-1} R\) such that \(\overline{p} + \overline{z}_1 = \overline{q} + \overline{z}_2\).

The quotient space \(\Sigma^{-1}R / \sim_{\mathcal{R}_2}\) inherits the sum and product, under which it becomes an associative ring with trivial element being the class \([(1:1:0)]\) and the identity element being the class \([(1:1:1)]\). The additive inverse of \([(a\colon C\colon x)]\) is \([(a\colon C \colon -x)]\). One of the main results of \cite{MR1344222} is that this quotient space is precisely the universal localization \(R_\Sigma\), where the map \(\lambda\colon R \to R_\Sigma\) is given by \(x \mapsto [(1\colon 1\colon x)]\).

Suppose now that \(R\) is a topological ring. In that case, all the spaces \(^n R\) and \(R^n\) have a natural product topology, and \(\Sigma_n\) also becomes a topological space as it is contained in \(\Mat_n(R) \simeq R^{n^2}\). Hence, \(T_\Sigma\) becomes a topological space endowed with the direct union topology on which the operations of sum and product defined are continuous. Under the quotient topology, those operations are still continuous on \(\Sigma^{-1} R\) and \(R_\Sigma\), so that \(R_\Sigma\) becomes a topological ring. We call this topology structure on $R_\Sigma$ the \emph{induced} topological structure from $R$. It is clear that \(\lambda\colon R \to R_\Sigma\) is a continuous ring homomorphism. 
The main result of the appendix is the following theorem.

\begin{thm} \label{hausd}
Let $R$ be a profinite ring and $\Sigma$ a collection of square matrices over $R$. Then \(R_{\Sigma}\) admits a Hausdorff ring topology such that \(\lambda\colon R \to R_{\Sigma}\) is continuous. In particular, every finitely generated \(R\)-submodule of \(R_{\Sigma}\) is profinite.
\end{thm}
\begin{proof}
Without loss of generality we assume that $\Sigma$ is complete.
Since \(R\) is profinite, then \(T_\Sigma\) is a Hausdorff space under the induced topology.

\begin{claim} We have that  \(\Sigma^{-1}R\) is Hausdorff.
\end{claim}
\begin{proof} We must show that the equivalence relation \(\mathcal{R}_1\) is a closed subspace of \(T_\Sigma\times T_\Sigma\). Since \(T_\Sigma\) is given the union topology and each equivalence class under \(\mathcal{R}_1\) is entirely contained in one of the subsets \(T_n = R^n \times \Sigma_n \times {^nR}\), it suffices to prove that \(X_n = (\mathcal{R}_1 \cap (T_n \times T_n))\)  is closed in \(T_n \times T_n\).

Consider the continuous maps \(\psi\colon T_n \times T_n \times \GL_n(R)^2 \to R^n \times \Mat_n(R) \times {^nR}\) and \(\pi\colon T_n \times T_n \times \GL_n(R)^2 \to T_n \times T_n\) given by \[\psi(a,C,x,b,D,y,U,V) = (aU - b\,, VCU - D\,, Vx - y)\,,\] \[\pi(a,C,x,b,D,y,U,V) = (a,C,x,b,D,y)\,.\]
It is clear that \(X_n = \pi(\psi^{-1}(0,0,0))\). Since \(\GL_n(R)^2\) is compact, \(\pi\) is a closed map, so we are done.
\end{proof}

The kernel of \(\Sigma^{-1}R \to R_\Sigma\) is the subsemigroup \[\Sigma^{-1}_0 R - \Sigma^{-1}_0 R = \{\overline{p} \in \Sigma^{-1}R\,\mid\, \exists \overline{z} \in \Sigma^{-1}_0R\text{ such that }\overline{p} + \overline{z} \in \Sigma^{-1}_0 R\}\,,\] which contains \(\Sigma^{-1}_0 R\) but it might be strictly larger.

\begin{claim}\label{lem-kernelclosed} We have that \(\Sigma^{-1}_0 R\) is closed in \(\Sigma^{-1} R\).
\end{claim}
\begin{proof} It again suffices to consider \(\Sigma^{-1}R_n = (R^n \times \Sigma_n \times {^n R})/\sim_{\mathcal{R}_1}\) and show that \(Y_n = \Sigma^{-1}_0 R \cap \Sigma^{-1}R_n\) is closed in \(\Sigma^{-1}R_n\), as we have a homeomorphism between \(\Sigma^{-1}R\) and the disjoint union \(\bigcup_{n \geq 1} \Sigma^{-1}R_n\).

Decompose \(Y_n = Z_I \cup Z_{II} \cup \bigcup_{i=1}^{n-1} Z_{III,i}\) where \[Z_I = \{(a\colon C \colon 0) \in Y_n\}\,,\]
\[Z_{II} = \{(0\colon C\colon x) \in Y_n\}\,,\]
\[Z_{III,i} = \left\{\left(\begin{pmatrix} a & 0\end{pmatrix}\colon \begin{pmatrix} C & 0 \\ 0 & D\end{pmatrix}\colon \begin{pmatrix}0 \\ x\end{pmatrix}\right) \in Y_n\,\mid\, 
\begin{array}{l} a \in R^i\,, x \in {^{n-i}R},\\ C \in \Sigma_i\,, D \in \Sigma_{n-i} \end{array}\right \}\, .\]
We only need to show that each one of the \(Z_I\), \(Z_{II}\) and \(Z_{III,i}\) is closed in \(\Sigma^{-1}R_n\).

The inverse image of \(Z_I\) in \(R^n \times \Sigma_n \times {^n R}\) is the set of triples of the form \((a,C,0)\), which immediately shows that it is closed. An analogous description holds for the inverse image of \(Z_{II}\), hence it is also closed. The only case requiring a finer description is that of \(Z_{III,i}\).

The inverse image of \(Z_{III,i}\) in \(R^n \times \Sigma_n \times {^n R}\) is \[\left\{\left(\begin{pmatrix} a & 0\end{pmatrix}
U\,, V\begin{pmatrix}C & 0\\ 0 & D\end{pmatrix}U\,, V\begin{pmatrix} 0 \\ x\end{pmatrix}\right)\,\mid\, U\,,V \in 
\GL_n(R)\right\}\,.\] Define \[Z_i = \left\{\left(\begin{pmatrix} a &0\end{pmatrix}\,, 
\begin{pmatrix}  C & 0\\ 0 & D\end{pmatrix}\,, \begin{pmatrix} 0 \\x\end{pmatrix}\right)\right\} \subseteq R^n \times \Sigma_n \times {^n R}\,.\] 
Then, the inverse image in question is the image of \(Z_i \times (\GL_n(R))^2\) under the continuous right group action map \(\mu\colon (R^n \times \Sigma_n \times {^nR})\times (\GL_n(R))^2 \to R^n \times \Sigma_n \times {^n R}\) given by \[\mu(b,A,y,U,V) = (bU,V^{-1}AU,V^{-1}y)\,.\] Since \(\Sigma\) satisfies property~(\ref{property-DS}), the set \(Z_i\) is closed, and because \((\GL_n(R))^2\) is compact the map \(\mu\) is closed. Hence, \(\mu(Z_i\times(\GL_n(R))^2)\) and \(Z_{III,i}\) are both closed.
\end{proof}

\begin{claim}\label{prop-topology-0}We have that  \(\{[(1\colon 1\colon 0)]\}\) is closed in \(R_{\Sigma}\).
\end{claim}
\begin{proof} We must show that \(\Sigma_0^{-1} R - \Sigma_0^{-1} R\) is a closed subset of \(\Sigma^{-1}R\). Consider the map \(\varphi\colon \Sigma^{-1} R \times \Sigma^{-1}_0 R \to \Sigma^{-1}R\) given by \(\varphi(\overline{p},\overline{z}) = \overline{p} + \overline{z}\). The inverse image \(\varphi^{-1}(\Sigma^{-1}_0 R)\) is closed in the product by Claim~\ref{lem-kernelclosed}. Since this inverse image is precisely \((\Sigma^{-1}_0 R - \Sigma^{-1}_0 R) \times \Sigma^{-1}_0 R\), it proves that \(\Sigma^{-1}_0 R - \Sigma^{-1}_0 R\) is closed in \(\Sigma^{-1} R\).
\end{proof}
   By Claim~\ref{prop-topology-0}, the ring topology constructed on \(R_{ {\Sigma}}\) is such that \(\{0\}\) is closed. Since \((R_{ {\Sigma}}, +)\) is a topological abelian group, it must also be Hausdorff.   
   
   For the last statement of the theorem, take one such \(M \leq R_{\Sigma}\) and observe that the kernel \(I\) of any surjection \(R^n \to M\) must be closed in \(R^n\). Therefore,  \(M \simeq R^n/I \)  is profinite.
     \end{proof}
\bibliographystyle{amsalpha}
\bibliography{freebycyclic}
\end{document}